\documentclass[12pt]{article}
\usepackage{enumitem}
\usepackage{todonotes}
\usepackage{mathtools}
\usepackage{thmtools,thm-restate}
\mathtoolsset{showonlyrefs=true}

\usepackage{hyperref}

\usepackage{url}
\usepackage{amsfonts}
\usepackage{amsmath,amssymb,color,tabu}
\usepackage{graphicx}
\usepackage{epstopdf}
\usepackage{enumitem}
\usepackage{multirow}
\usepackage{bm}
\usepackage[authoryear,sort]{natbib}
\setcitestyle{aysep={,}}
\usepackage{verbatim}
\usepackage{float}
\usepackage[toc,page]{appendix}
\usepackage[onehalfspacing]{setspace}
\usepackage{setspace}
\usepackage{rotating}
\usepackage{subcaption}
\usepackage{amsthm}
\usepackage{booktabs,dcolumn,caption}

\newtheorem{innercustomthm}{Theorem}
\newenvironment{customthm}[1]
  {\renewcommand\theinnercustomthm{#1}\innercustomthm}
  {\endinnercustomthm}

  \newtheorem{innercustomprop}{Proposition}
\newenvironment{customprop}[1]
  {\renewcommand\theinnercustomprop{#1}\innercustomprop}
  {\endinnercustomprop}

  \newtheorem{innercustomlem}{Lemma}
\newenvironment{customlem}[1]
  {\renewcommand\theinnercustomlem{#1}\innercustomlem}
  {\endinnercustomlem}

\usepackage{dsfont}

\usepackage{xr}
\usepackage{mathabx}

\usepackage{array}

\usepackage{algpseudocode,algorithm,algorithmicx}

\newcolumntype{L}[1]{>{\raggedright\let\newline\\\arraybackslash\hspace{0pt}}p{#1}}
\newcolumntype{C}[1]{>{\centering\let\newline\\\arraybackslash\hspace{0pt}}p{#1}}
\newcolumntype{R}[1]{>{\raggedleft\let\newline\\\arraybackslash\hspace{0pt}}p{#1}}

\newcommand{\PP}{{\mathbf{P}}}

\newcommand{\card}{\dim}

\newcommand{\llll}{\boldsymbol \lambda}

\newcommand{\yy}{\mbox{$\mathbf y$}}

\newcommand{\xx}{\mathbf x}
\newcommand{\ww}{\mathbf w}
\newcommand{\TT}{\mathbb T}

\newcommand{\ab}{\bm \varnothing}

\newcommand{\bLam}{\boldsymbol \Lambda}

\newcommand{\tV}{\tilde{V}}

\newcommand{\Real}{\mathbb{R}}

\newcommand{\eqd}{{\stackrel{d}{=}}}
\newcommand{\prob}{\mathbb{P}}
\newcommand{\qrob}{\mathbb{Q}}

\newcommand{\expe}{\mathbb{E}}

\newcommand{\KKo}{\mathbb{K}_{\textrm{o}}}
\newcommand{\KKa}{\mathbb{K}_{\textrm{a}}}

\newcommand{\mhomo}{$\mathcal{M}_{\textrm{s}}$}

\newcommand{\pcr}{\mathrm{FNP}}
\newcommand{\lpcr}{\mathrm{LFNR}}

\newcommand{\uti}{\mathrm{U}}

\newtheorem{theorem}{Theorem}

\newtheorem{fact}{Fact}

\newtheorem{definition}{Definition}
\newtheorem{example}{Example}

\newtheorem{lemma}{Lemma}

\newtheorem{proposition}{Proposition}
\newtheorem{remark}{Remark}

\newcommand\strG{\text{$*$} }
\newcommand\strGfull[1]{\text{$\textbf{A}{\mathbf P}_{#1}$}}
\newcommand\strA{\text{$\textbf{A}$} }
\newcommand{\spaceu}{\mathcal{S}_{\mathrm{u}}}
\newcommand{\spaceo}{\mathcal{S}_{\mathrm{o}}}
\newcommand{\lleq}{\preccurlyeq}
\newcommand{\TTp}{\TT^{*}}

\newcommand{\adset}{\mathcal{T}}

\newcommand{\fil}{\mathcal{F}}

\renewcommand{\algorithmicrequire}{\textbf{Input:}}
\renewcommand{\algorithmicensure}{\textbf{Output:}}
\numberwithin{lemma}{section}
\numberwithin{equation}{section}
\numberwithin{proposition}{section}
\numberwithin{remark}{section}

\newcommand{\ind}{\mathds{1}}
\let\hat\widehat
\let\tilde\widetilde

\newcommand{\VV}{\mathbf V}
\newcommand{\vv}{\mathbf v}
\newcommand{\uu}{\mathbf u}

\newcommand{\zz}{\mathbf z}

\title{Compound Sequential Change-point Detection in Parallel Data Streams}
\author{Yunxiao Chen\\
London School of Economics and Political Science\\
Xiaoou Li\\
University of Minnesota
}
\date{}

\begin{document}
\onehalfspacing
\maketitle

\abstract{
We consider sequential change-point detection in parallel data streams, where each stream has its own change point.
%Such problems are widely encountered in real-world problems.
Once a change is detected in a data stream, this stream is deactivated permanently.  The goal is to
maximize the normal operation of the pre-change streams, while controlling the proportion of post-change streams among the active streams at all time points. Taking a Bayesian formulation, we develop a compound decision framework for this problem. A procedure is proposed that is uniformly optimal among all sequential procedures which control the expected proportion of post-change streams at all time points.
We also investigate the asymptotic behavior of the proposed method when the number of data streams grows large.
%Several non-standard technical tools involving partially ordered spaces and monotone coupling of stochastic processes are developed for proving the optimality result.
Numerical examples are provided to illustrate the use and performance of the proposed method.
}


\section{Introduction}

Sequential change-point detection, which dates back to the pioneering work of \cite{page1954continuous,page1955test},
aims at the early detection of distributional changes in sequentially observed data. Methods for sequential change-point detection have received wide applications in various fields, including engineering, education, medical diagnostics, finance, among others, where a change point typically corresponds to a deviation of a data stream from its `normal' state. The classical methods for sequential change-point detection focus on the detection of one or multiple changes in a single data stream \citep{lorden1971procedures, page1954continuous,roberts1966comparison,shewhart1931economic,shiryaev1963optimum}. With the advances in information technology, large-scale streaming data become more common and
many recent developments tend to focus on change-point detection in multiple data streams \citep{chan2017optimal,chen2015graph,chen2019sequential,mei2010efficient,xie2013sequential,fellouris2016second}.

%%studies the detection of time point after which there is a distributional change in sequentially observed data.
%%Methods for online change-point detection have
%receives wide applications in various fields, including engineering, education, medical diagnostics, finance, among others, where a change point typically corresponds to a deviation of a data stream from its `normal' state.
%While early works on this topic \citep{page1954continuous,roberts1966comparison,shewhart1931economic,shiryaev1963optimum} study the change in a single data stream, many recent developments focus on change-point detection in multiple data streams \citep{chan2017optimal,chen2015graph,chen2019sequential,chen2020high,mei2010efficient,xie2013sequential}. \yc{This paragraph needs revision. We will include more literature on change detection.}

In this paper, we consider sequential change-point detection in multiple parallel data streams,
where each stream has its own change point. Once a change is detected in a data stream, this stream is deactivated permanently and its data are no longer collected. The goal is to
maximize the normal operation of the pre-change streams, while controlling the proportion of post-change streams among the active ones at all time points.
This problem is commonly encountered in the real world. One such example is the monitoring of item pool in standardized educational testing   \citep{choe2018sequential,cizek2016handbook,van2015bayesian,veerkamp2000detection}. In this application, each item corresponds to a data stream, for which data are collected sequentially from its use in test administrations over time.
A change point occurs when the item is leaked to future test takers. The goal is to detect and remove  changed items in an item pool that consists of hundreds or even thousands of items in a sequential fashion. Once a change point is detected for an item, test administrators would like to remove it from the item pool to ensure test fairness. On the other hand, it is important to maximize the usage of each item before its leakage, due to the cost of developing new items. There are many other applications, including multichannel spectrum sensing \citep{chen2020false} and credit card fraud detection \citep{dal2017credit}.

Despite its wide applications, this type of problems is rarely explored in the literature of multi-stream sequential change-point detection.
One exception is \cite{chen2020false}, who address a similar problem by proposing a sequential version of the Benjamini-Hochberg FDR control procedure \citep{benjamini1995controlling} for detecting and deactivating post-change data streams. However, no optimality theory is provided in \cite{chen2020false}.
% for their approach.
%is statistically challenging and
The challenges of developing optimality theory lie in the compound nature of the FDR-type risk measure and the stochastic control component due to the deactivation of data streams.
% and the lack of a suitable decision theory framework.
 In this paper, we formulate the problem under a compound decision theory framework and propose an optimal change-point detection procedure. Our contributions are summarized below.

First,  we formulate this problem under a Bayesian sequential change-point detection setting, which generalizes the classical Bayesian setting for single-stream change-point detection \citep{lai2001sequential} to parallel streams. Moreover, we introduce new performance metrics, borrowing ideas from the compound decision theory for multiple hypothesis testing \citep{benjamini1995controlling,brown2009nonparametric,cai2019covariate,efron2016computer,efron2019bayes,genovese2002operating,sun2007oracle,zhang2003compound}.
Specifically, we propose to control a local False Non-discovery Rate (FNR) at each time point, defined as the expected proportion of post-change streams among the active ones under the current posterior measure. This metric adapts the FNR for multiple hypothesis testing \citep{genovese2002operating} to  parallel-stream change-point detection. In addition, we introduce a compound stream utilization measure that is closely related to the classical notion of average run length \citep{lorden1971procedures}. A compound sequential detection procedure involves a trade-off between  local FNR and stream utilization at each time point and our objective is to maximize stream utilization,
while controlling the local FNR to be below a pre-specified threshold all the time. {Comparing with classical performance metrics for individual streams,}
% running online change-point detection algorithms independently,
 the proposed metrics better evaluate the risk of sequential decision at an aggregate level and thus is more suitable for large-scale streaming data.

 Second, we propose a sequential decision procedure that can control local FNR under any pre-specified threshold. Under a class of Bayesian change-point models, it is shown that this procedure is uniformly optimal among all the sequential detection procedures under the same local FNR constraint, in the sense that the proposed procedure has the highest  stream utilization at any time. We emphasize that this is a non-asymptotic result that applies to any finite number of data streams. This result implies that this compound change-point detection problem is very special, in the sense that a myopic decision rule that maximizes the next-step stream utilization under the local FNR constraint is also uniformly optimal throughout time. Phenomenon of this kind  does not hold in general for  stochastic control problems \citep{howard1960dynamic}. The proof of this uniform optimality result is  non-trivial, for which new mathematical tools are developed,
including the construction of monotone coupling over a partially ordered space \citep{thorisson2000regeneration} for {comparing stochastic processes with different dimensions due to the deactivation step.}
%{on a partially ordered space.}
Besides non-asymptotic optimality, asymptotic theory is also established to characterize the performance of the proposed method when the number of data streams grows large.

We point out that %under the current setting, multiple detections are made across the data streams.
the current setting is substantially different from most of the existing works on multi-stream sequential change-point detection, including \cite{mei2010efficient}, \cite{xie2013sequential}, \cite{chen2015graph}, \cite{chan2017optimal},
\cite{chen2019sequential}, and \cite{chen2020high}.
These works consider the detection of a {\em single} change point,  after which all (or part) of the data streams deviate from their initial states. On the other hand, the current work  detects multiple
change points in parallel streams. As the dimension of the action space at each time point grows exponentially with the number of data streams, the current problem tends to be computationally and theoretically more challenging.

\section{Compound Sequential Change-point Detection}\label{sec:problem-settings}
\subsection{Bayesian Change-point Model for Parallel Streams}
Consider in total $K$ parallel data streams. For each $k = 1, ..., K$, the observations from the $k$th stream are $X_{k,t}$, $t = 1, 2, ...$.
 {Each stream $k$ is associated with a change point, denoted by $\tau_k$, which takes value in $\{0\} \cup \{\infty\}\cup\mathbb{Z}_+$}. 
The random vector $(\tau_1,\cdots,\tau_K)$ is assumed to follow a known prior distribution. Given the change points, the data points $X_{k,t}$ from the $k$th stream at time $t$ are independent for different $t$ and $k$. It is further assumed that the pre- and post-change distributions of $X_{k,t}$ have the density functions $p_{k,t}(\cdot)$ and $q_{k,t}(\cdot)$ with respect to some baseline measure $\mu$. That is, $X_{k,t}$ has the following conditional density functions
\begin{equation}\label{eq:general-model}
	X_{k,t}\mid\tau_1,\cdots,\tau_K, \{X_{l,s}; 1\leq l\leq K, 1\leq s\leq t-1 \}
	\sim
	\begin{cases}
		p_{k,t} &\text{ if }t\leq \tau_k,\\
		q_{k,t}& \text{ if }t>\tau_k.
	\end{cases}
\end{equation}
{\begin{remark}
We assume that the prior distribution for the change points, and the pre- and post-change distributions are known, which is a standard assumption in single-stream Bayesian sequential change detection \citep[e.g.,][]{shiryaev1963optimum}. Similar assumptions are adopted in
recent developments on multi-stream sequential multiple testing \citep{song2019sequential} and multi-stream sequential change detection \citep{chen2020false}.

%classic single-stream sequential analysis problems (e.g., \citep{shiryaev1963optimum}) and also in recent developments in sequential multiple testing for multi-stream problems \citep{song2019sequential}.

%{It is a reasonable assumption when the pre- and post- change distribution can be specified or well-calibrated based on historical information. }
%For example, in the application of detecting compromised items in educational testing, the pre-change distribution are usually considered known as the items are usually calibrated through pilot studies, and the post-change distribution is usually specified to indicate the level of potential item leaking {\color{red} (is this true? This sentence can be removed?)}.

% This assumption may be difficult to calibrate the pre- or post- distribution with high accuracy (see, e.g., \cite{mei2006sequential,lai2010sequential}).

When these distributions are unknown, the current results provide the oracle procedure and theoretical guidance for the development and analysis of the sequential change detection procedures. In addition, the proposed procedure can be extended to handle the unknown distribution scenario via an empirical Bayes approach \citep[see e.g.,][]{efron2008microarrays,jiang2009general,robbins1956empirical,zhang2003compound}.
%Under suitable model assumptions, we can adaptively estimate these unknown distributions and then proceed by plugging the estimated distributions into the proposed procedure.
%The consistency and optimality of such a procedure are worth further investigation.
Alternatively,  we can run the proposed procedure under the worst case model, if
%the unknown change point model has certain stochastic ordering properties and
such a model %worst case model
can be specified using domain knowledge.
This procedure will preserve some properties of the oracle one, when the change point model enjoys certain stochastic ordering properties.

\end{remark}
% Throughout the paper, we assume that $p_{k,t}$ and $q_{k,t}$ are known and have the same support.
% We remark that it is common  to assume that the prior distribution for change points, and the pre- and post-change distributions are all known \citep[see][]{shiryaev1963optimum,chen2020false}, though this is
% arguably a strong assumption. Procedures and theoretical results developed under this oracle model provide a benchmark for the unknown-distribution case. We believe that the proposed compound decision framework and procedure can be generalized to the unknown-distribution setting. Specifically, we may parameterize the unknown distributions and then simultaneously solve the sequential estimation and change-point detection problems. For example, we may generalize the proposed procedure under a full Bayesian setting by imposing hyper priors for the unknown parameters,  similar to the treatment in \cite{fearnhead2007line} for change-point detection in a single stream. %Alternatively, we may take an empirical Bayes approach, treating the parameters as fixed

%then use a full Bayes approach

Equation \eqref{eq:general-model} provides a general model for change points in parallel data streams. It contains some commonly used models as special cases. We provide two examples below.

\begin{example}[A partially dependent model]\label{example:model-partial}
Let $\tau_0$ be a non-negative random variable and $\tau_1,\cdots,\tau_K$ are i.i.d. conditional on $\tau_0$, with conditional distribution
$
\prob\left(\tau_k=m|\tau_0=m\right)=\eta$ and $\prob\left(\tau_k=\infty|\tau_0=m\right)=1-\eta
$ for $m=0,1,\cdots$, and some parameter $\eta\in [0,1]$.
\end{example}
This model describes the situation where there  is a single change point for all of the data streams. After the change point, all or part of the data streams have a distributional change. If we further let $p_{k,t}$ be the density function of standard normal distribution $N(0,1)$, and $q_{k,t}$ be the density function of $N(\mu,1)$ for some $\mu>0$. Then this model becomes a Bayesian formulation of the change-point models studied in \cite{mei2010efficient}, \cite{xie2013sequential}, and \cite{chan2017optimal}. An interesting boundary case is $\eta=1$, where all the change points $\tau_1=\cdots=\tau_K$ are the same. This case can be viewed as a single change point affecting all the data streams.
% \begin{example}
% Assume $\tau_1,\cdots,\tau_K$'s are independent. In this case, the observations at different data streams are independent. That is, $\{X_{k,t}, t\geq 1\}$ are independent process for $k=1,\cdots, K$.
% \end{example}
% In the above model, the change points and data for different streams are independent. Consequently, the posterior distribution of change points and data are independent for different streams.
\begin{example}[An i.i.d. change-point model]\label{example:mhomo}
% \yc{Should we change it to ``i.i.d. change point model"? As the ``identical" part seems important according to the counter example}
Assume that $\tau_1,\cdots,\tau_K$ are i.i.d. geometrically distributed random variables with $\prob(\tau_k=m) = \theta (1-\theta)^m$ for $m=0,1,\cdots$ and $\theta\in(0,1)$. In addition, assume that  $p_{k,t}(x)=p(x)$ and $q_{k,t}(x)=q(x)$ for all $k,t,x$. This model is referred to as model \mhomo~in the rest of the paper.
\end{example}

We remark that the geometric distribution assumption is commonly adopted in
Bayesian change-point detection \citep[see, e.g., ][]{tartakovsky2014sequential}. We adopt this assumption for simplicity
as it leads to analytic posterior probabilities, and point out that it can be relaxed to  other known distributions.  See Section~\ref{subsec:post} for a  discussion about the calculation of posterior probabilities.

\subsection{Compound Sequential Change-point Detection}\label{subsec:procedure}

We now introduce a compound sequential change-point detection problem, which will be defined through an
index set process, $S_t \subset \{1, ..., K\}$, where $S_t$ indicates the set of active streams at time $t$. Specifically, if $k \in S_t$, then stream $k$ is active at time $t$; otherwise, it is deactivated.
We require the process to satisfy that $S_{t+1}\subset S_t$ for all $t = 1, 2, \cdots$, meaning that a stream is not allowed to be re-activated once turned off.
This requirement is consistent with many real-world applications. For example, in standardized educational testing, once an item is found to have leaked, it will be removed from the item pool permanently. At the beginning of data collection (i.e., $t = 1$), all the data streams are active, and thus $S_1 = \{1, ..., K\}$.

A sequential detection procedure $S_t$ is defined together with an information filtration, where the definition is inductive. We first let $\mathcal{F}_{1} = \sigma(X_{k,1}, k = 1, ..., K)$.
Then for any $t > 1$, we let $\mathcal{F}_{t} = \sigma(\mathcal{F}_{t-1}, S_{t}, X_{k,t},  k \in S_t)$, where
$S_t \subset \{1, ..., K\}$ is $\mathcal{F}_{t-1}$ measurable. % random variable and satisfies that if $I_{k,t} = 0$, then $I_{k,s} = 0$ for all $s > t$.
We say  $\{\mathcal{F}_{t}\}_{t = 1, 2, ...}$ is the information filtration, and the index set process $\{S_t\}_{t\geq 1}$ describes a compound sequential change-point detection procedure with respect to this information filtration.

Sometimes, it is  more convenient to represent the decision procedure by a random vector
%of stopping time under the filtration $\left\{\mathcal F_t\right\}_{t = 1, 2, ...}$, denoted by
$\TT=(T_1,\cdots, T_K)$, where $T_k \in \mathbb Z_{+}$ is defined as
$T_k = \sup\{t: k \in S_t\}.$
It is easy to check that $\{T_{k}=t\}\in \mathcal{F}_t$ for all $t$, and thus $T_k$ is a stopping time under the filtration $\{\mathcal{F}_{t}\}_{t = 1, 2,\cdots}$. The stopping time $T_k$ indicates the time up to which we collect data from the $k$th stream. In other words, starting from time $T_k+1$, the $k$th stream is deactivated and its data are no longer collected. The
index set at time $t$ is given by $S_t = \{k: T_k \geq t\}$.

The sigma field $\sigma(X_{k,s\wedge T_k}, s\leq t, k = 1, ..., K)$ is in $\mathcal{F}_t$, meaning that
our information filtration at time $t$ contains all the information from the streams when they are active.
Besides the information from observable data $X_{k,t}$,  the filtration $\mathcal{F}_t$ also contains information from the decision history, reflected by that $S_s$ is measurable with respective to $\mathcal{F}_t$, for all $s \leq t$.

In what follows, we introduce two compound performance metrics for this sequential decision problem.

\subsection{Local False Non-discovery Rate}

In this sequential decision problem, our primary goal is to control the proportion of post-change streams among the active ones at any time, where a smaller proportion indicates a better overall quality of the active streams.
This proportion can be viewed as a False Non-discovery Proportion (FNP) that is often considered in multiple hypothesis testing \citep{genovese2002operating}, but defined at each time point under the current sequential setting. More precisely, we define the FNP as
\begin{equation}\label{eq:pcr}
	\pcr_{t+1}(\TT) = \frac{\sum_{k \in S_{t+1}} \ind( \tau_k<t) }{|S_{t+1}| \vee 1} =\frac{ \sum_{k=1}^K \ind(T_{k}> t, \tau_k<t)}{ \big\{\sum_{k=1}^K \ind(T_{k}> t)\big\}\vee 1},
\end{equation}
where $t = 1, 2, ...$, $a\vee b=\max(a,b)$ and $|S|$ indicates the size of a set $S$.
In this definition, $|S_{t+1}|$ %\footnote{{\color{red} earlier version:}$\sum_{k=1}^K \ind\left(T_{k}> t\right)$}
represents the total number of active streams at time $t+1$, %\footnote{{\color{red} removed:}(recall that the $k$-th stream is used up to time $T_k$)},
and $\sum_{k \in S_{t+1}} \ind( \tau_k<t)$ %\footnote{{\color{red} earlier:}$\sum_{k=1}^K \ind\left(T_{k}> t, \tau_k<t\right)$}
represents the total number of active post-change streams at time $t+1$.
By having `$\vee 1$' in the denominator, $\pcr_{t+1}(\TT)$ is well-defined even when $|S_{t+1}| = 0$.
Finally, we let $\pcr_1=0$, as
$\tau_k\geq 0$ for $k = 1, 2, ..., K$.

Ideally, we would like to control the FNP to be below an acceptable threshold at any time point, which is not always possible as the change points are unknown.
%However, $\pcr_{t}(\TT)$ cannot be directly controlled, as the change points $\tau_k$ are not directly observable.
As an alternative, we control the Local False Non-discovery Rate (LFNR) which can be viewed as the best estimate of the FNP under the Bayesian sense.
The LFNR is defined as
\begin{equation}\label{eq:pcr2}
	\lpcr_{t+1}(\TT) =\expe(\pcr_{t+1}(\TT)\mid \mathcal F_t), ~t = 1, 2, \cdots.
\end{equation}
%We let $\lpcr_{1}(\TT) = 0$, as $\pcr_1=0$.
Since  $\pcr_1=0$, $\lpcr_{1}(\TT)$ is set to 0.

In what follows, we will focus on sequential change-point detection procedures defined in Section \ref{subsec:procedure} under the constraint that $\lpcr_{t}(\TT) \leq \alpha$ for all $t$ for some
pre-specified level $\alpha$ (e.g., $\alpha=1\%$).
More precisely, for a given $\alpha\in(0,1]$,
we consider the following class of compound sequential change-point detection procedures which controls the LFNR to be below or equal to $\alpha$  at any time,
$
	\adset_{\alpha}=\{
	\TT\in\adset: \lpcr_{t}(\TT) \leq \alpha \text{ a.s.}, \text{ for all } t=1,2,\cdots
	\},
$
%{\color{blue}
where $\adset$ denotes the entire set of compound sequential change-point detection procedures. %}
%For a level $\alpha\in[0,1]$, consider all the sequential detection-and-elimination procedures so that control the {\em conditional} PCR at level $\alpha$ at any time, denoted as
%\begin{equation}
%	\adset_{\alpha}=\left\{
%	\TT=(T_1,\cdots, T_K)\in\adset: \expe\left(\pcr_{t+1}(\TT)\Big|\mathcal{F}_{t}\right)\leq \alpha \text{ a.s.} \text{ for all } t=0,1,2,\cdots
%	\right\}.
%\end{equation}

We provide a few remarks. First, $\lpcr_{t+1}(\TT)$ is a random variable measurable with respect to $\mathcal{F}_t$. It depends on both the change-point model and the detection procedure $\TT$.  Second,
it is easy to observe that $\expe(\pcr_t(\TT))\leq \alpha$ for every $t$, for any $\TT\in\adset_{\alpha}$. That is, the unconditional expectation of FNP is also controlled at the same $\alpha$ level. {Finally, 
%We remark that
 %discuss the applicability of the proposed FNR-type risk measure and its differences and connections with false discovery rate (FDR)-type risk measures. Specifically,
 by replacing $\tau_k<t$ with $\tau_k\geq t$ and $ S_{t+1}$ with $S_{t}\setminus S_{t+1}$ in the definition of FNP,
	%$\textrm{LFNR}_{t+1}=\frac{\sum_{k\in S_t}W_{k,t}}{|S_t|\vee 1}$,
	 we can similarly define the  false discovery proportion (FDP) and local false discovery rate (LFNR) as $\textrm{FDP}_{t+1}=(|S_{t}\setminus S_{t+1}| \vee 1)^{-1}{\sum_{k \in S_t\setminus S_{t+1}} \ind( \tau_k\geq t) }$ and $\textrm{LFDR}_{t+1}=\expe(\textrm{FDP}_{t+1}(\TT)|\fil_t)$. 
The main difference between FNR- and FDR-type risk measures is whether focusing on the remaining streams or the streams to be deactivated. Specifically, the LFNR focuses on the remaining streams and thus is a preferred measure if the goal is to control the overall quality of the active data streams (e.g., controlling the proportion of leaked items in the item pool of an educational test). On the other hand, the LFDR is calculated based on the streams to be detected and deactivated. It is thus a better metric if the goal is to control the accuracy among the detected streams.
}

% on the data streams that are involved. That is, the LFNR is calculated based on the remaining streams while the LFDR is based on the streams to be stopped. Thus, the LFNR better reflects the quality of active streams that are still in use, while the  LFDR emphasizes more on the accuracy of the detections. In general, the choice of the risk measure  should depend on the goal of the application:  if one needs to control the overall quality of the active data streams (e.g., controlling the proportion of leaked items in the item pool of an educational test), then an FNR-type measure is preferred, while if the goal is to make accurate detections, then both FNR- and FDR-type risk measures are sensible.

%\yc{In Figure~\ref{fig:illustration}, we ...}

%to be below the same $\alpha$ level.

 %i.e., the unconditional PCFNR is controlled.

\subsection{Stream Utilization and Optimality Criteria} \label{subsec:criteria}

%\yc{Change the notation of U here.}
Given a level $\alpha$, the class $\adset_{\alpha}$ has many elements. %To compare different sequential procedures in $\adset_{\alpha}$, w
We propose to compare them based on their overall utilization of data streams.
More precisely,
we consider the following measure
\begin{equation*}
	\uti_t(\TT)= \sum_{s=1}^t\vert S_s\vert = \sum_{s=1}^t \sum_{k=1}^K \ind(T_k\geq s),
\end{equation*}
where $\TT=(T_1,\cdots, T_K)$ is a sequential change-point detection procedure and $\uti_t(\TT)$ is the total number of data points collected from the beginning to time $t$.

For two sequential procedures $\TT$ and $\TT'$ in $\adset_{\alpha}$, we say $\TT$ is more efficient %\footnote{More `powerful' or `efficient'? `powerful' sounds similar to hypothesis testing. } %\yc{``efficient" may be more appropriate. Revised.}}
than $\TT'$ at time $t$ if $ \expe(\uti_t(\TT)) \geq \expe(\uti_t(\TT'))$.
In addition, we say
$\TT$ is uniformly more efficient than
$\TT'$ if $\expe( \uti_t(\TT)) \geq  \expe(\uti_t(\TT') )$, for all $t = 1, 2, ...$. % and there exists $t_0$ such that $\uti_{t_0}(\TT) > \uti_{t_0}(\TT')$.
Following the previous discussion, our goal becomes developing an efficient procedure in terms of stream utilization, under the constraint that LFNR is below a pre-specified $\alpha$ level all the time. 
Specifically, we consider the following two optimality criteria, which will guide our development of compound detection procedures to be discussed in Section~\ref{sec:method}.

\begin{definition}[Uniform optimality]\label{def:uniform}
We say a sequential change-point detection procedure $\TT\in \adset_{\alpha}$ is uniformly optimal in $\adset_{\alpha}$, if $\TT$ is uniformly more efficient than
$\TT'$, for any $\TT' \in \adset_{\alpha}$. {That is,  $\expe(\uti_t(\TT)) = \sup_{\TT'\in \adset_{\alpha}} \expe(\uti_t(\TT')),$
for all $t = 1, 2, ...$.}
\end{definition}
%Ideally, we would like to a sequential procedure which is uniformly
%In other words, if a uniformly optimal sequential procedure $\TT^*$ exists in $\adset_{\alpha}$, then
Ideally, we would like to find this uniformly optimal procedure.
However, such a procedure does not necessarily exist  as the most efficient procedure at one time point  may be less efficient than another procedure at a different time point.
Thus, we also consider a weaker version of optimality, which is referred to as the local optimality at a given time point.
\begin{definition}[Local optimality]\label{def:local}
Given $\mathcal F_t$ at time $t$, we say the choice of $S_{t+1} \subset S_t$ is
locally optimal at time $t+1$, if $S_{t+1}$ is $\mathcal F_t$ measurable,
%$\TT^*$ is more powerful than
%$\TT$ at time $t$, for any $\TT \in \adset_{\alpha}$. That is,
$\expe\left(\frac{\sum_{k \in S_{t+1}} \ind( \tau_k<t) }{|S_{t+1}| \vee 1}\big\vert \mathcal F_t\right)\leq \alpha,$ 
{and $\vert S_{t+1}\vert \geq \vert S\vert$ a.s. for any other $S \subset S_t$ that
is also $\mathcal F_t$ measurable and satisfies}
$\expe\left(\frac{\sum_{k \in S} \ind( \tau_k<t) }{|S| \vee 1}\big\vert \mathcal F_t\right)\leq \alpha. $

%and
%$$\expe\left(\vert S_{t+1}\vert  \vert \mathcal F_t\right) \geq \expe\left(\vert S \vert \vert \mathcal F_t\right),$$
%for any $S \subset S_t$ satisfying
%$$\expe\left(\frac{\sum_{k \in S} \ind\left( \tau_k<t\right) }{|S| \vee 1}\big\vert \mathcal F_t\right)\leq \alpha.$$

\end{definition}
Note that the local optimality criterion only looks at one step forward. A procedure is locally optimal if
it maximizes the stream utilization in the next step. Achieving local optimality in each step does not necessarily lead to uniform optimality and a uniformly optimal procedure does not necessarily exist; see Example~\ref{example:noopt} in Section~\ref{sec:theory}.

{We provide a discussion on the choice of the performance metric. The expected stream utilization measure is most sensible, if the active streams have the same utility at any time point, whether having changed or not. This approximately holds for the application to item pool monitoring in educational testing, when the leaked items are only accessible by a small proportion of test takers, in which case the utilities of the leaked and unleaked items are similar.

%Alternatively, one may consider to minimize the 
 A closely related performance measure is the 
cumulative number of detections $\textrm{CD}_t=K - |S_t|$ at each time point. This performance metric is sensible when each detection (and thus deactivation) is associated with a fixed cost, in which case the goal becomes to minimize the total cost up to each time point. This metric may also be sensible for the application to item pool monitoring in educational testing. That is, once an item is deactivated, a new item needs to be developed as a replacement, for which the cost  is approximately the same across items.

In some applications, it may be more sensible to consider a performance metric based on the utilization of pre-change streams, defined as $\text{RL}_t(\TT)=\sum_{k=1}^K (T_k\wedge \tau_k\wedge t)$. {The expectation of this metric can be viewed as an online-and-compound version of the average run length to false alarm \citep{lorden1971procedures}},
%This metric can be viewed as a compound version of the run length before the change point \citep{lorden1971procedures}, 
a classical performance metric for sequential change detection. 

%the expected $\text{RL}_t(\TT)$ as the performance metric.

As will be shown in Section~\ref{sec:additional}, similar optimality results hold based on the performance measures $\textrm{RL}_t(\TT)$ and $\textrm{CD}_t(\TT)$.
% As will be shown in Section~\ref{subsec:sketch}, 
% the proposed method is also  uniformly optimal when our performance metric is based on $\textrm{RL}_t(\TT)$ or $\textrm{CD}_t(\TT)$. 
}
%As we will show in Section~\ref{subsec:sketch},
%the proposed method is also locally and uniformly optimal under this performance metric.

 %that typically has a fixed monetory

%\clearpage
%
%By definition, our expected stream utilization measure
%has a similar flavor as the average run length \citep{lorden1971procedures}, with a slight difference that the classical definition of average run length is based on  run length to false alarm while our stream utilization is defined based on  run length to detection. We choose the current definition of stream utilization for the simplicity of interpretation and theoretical proof. 
%
%\clearpage
%
%The optimality theory developed in this paper can be extended to a version of stream utilization defined based on run length to false alarm; \yc{see Section~\ref{subsec:sketch} for more details}.
%
%
%We point out that this metric is sensible for applications where  the cost of deactivating a stream is the same no matter it is a pre- or post-change stream.
%
%
%It is sensible to use the expected stream utilization as a performance metric for applications where  the cost of deactivating a stream is the same no matter it is a pre- or post-change stream. For example, consider the application to item pool monitoring in standardized educational testing, where each stream corresponds to an item. The deactivation of an item typically requires to develop a new item to replace it, which typically has a fixed cost. 

\section{Proposed Method}\label{sec:method}
\subsection{One-step Update Rule}

We first propose a one-step update rule for controlling the LFNR to be below a pre-specified level. Let a certain sequential change-point detection procedure be implemented from time 1 to $t$, and $\mathcal F_t$ be the current information filtration. A one-step update rule decides the index set $S_{t+1} \subset S_t$ based on the up-to-date information $\mathcal F_t$, so that the LFNR at time $t+1$
%$$\expe\left(\frac{\sum_{k \in S_{t+1}} \ind\left( \tau_k<t\right) }{|S_{t+1}| \vee 1}\big\vert \mathcal F_t\right)$$
is controlled below the pre-specified level $\alpha$. In the meantime, this update rule tries to maximize the size of $S_{t+1}$ to optimize stream utilization. {The details of the proposed one-step update rule is described in Algorithm~\ref{alg:one-step-rule1} below.}

%The proposed one-step update rule is described in Algorithm~\ref{alg:one-step-rule1} below.
% {\color{red} verbal explanations of each step}
% \begin{algo}\label{alg:one-step-rule1}
% One-step update rule.
% %\vspace*{-12pt}
% \begin{enumerate}
%   \item[] {\bf Input}: Threshold $\alpha$, the current index set $S_t$, and posterior probabilities $(W_{k,t})_{k\in S_t}$, where
% $W_{k,t} = \prob(\tau_k < t\vert \mathcal F_t).$
%   \item[] {\bf Step 1}: Sort the posterior probabilities in an ascending order. That is,
% $$W_{k_1,t} \leq W_{k_2,t} \leq \cdots \leq W_{k_{|S_t|},t},$$
% where $S_t = \{k_1, ..., k_{|S_t|}\}$. To avoid additional randomness, when there exists a tie $W_{k_i,t} = W_{k_{i+1},t}$, we require $k_i <k_{i+1}$.
% %\footnote{When there exists a tie, h}
% \item[] {\bf Step 2}: For $n = 1, ..., |S_t|$, define
% $R_{n} = {(\sum_{i=1}^n W_{k_i,t})}/{n}.$
% and define $R_{0} =0$.
% \item[] {\bf Step 3}:  Find the largest $n \in \{0, 1, ..., |S_t|\}$ such that $R_{n} \leq \alpha.$
%   \item[] {\bf Output}: $S_{t+1} = \{k_1, ..., k_n\}$ if $n\geq 1$ and $S_{t+1} = \emptyset$ if $n = 0$.
% \end{enumerate}
% \end{algo}
%
{\centering
%\begin{minipage}{1\linewidth}
\renewcommand{\algorithmicrequire}{\textbf{Input:}}
\renewcommand{\algorithmicensure}{\textbf{Output:}}
\begin{algorithm}[tbh]
\caption{One-step update rule.\label{alg:one-step-rule1}}
\begin{algorithmic}[1]
\Require  Threshold $\alpha$, the current index set $S_t$, and posterior probabilities $(W_{k,t})_{k\in S_t}$, where
$W_{k,t} = \prob(\tau_k < t\vert \mathcal F_t).$
\State Sort the posterior probabilities in an ascending order. That is,
$W_{k_1,t} \leq W_{k_2,t} \leq \cdots \leq W_{k_{|S_t|},t},$
where $S_t = \{k_1, ..., k_{|S_t|}\}$. To avoid additional randomness, when there exists a tie ($W_{k_i,t} = W_{k_{i+1},t}$), we require $k_i <k_{i+1}$.
%\footnote{When there exists a tie, h}
\State For $n = 1, ..., |S_t|$, define
$R_{n} = \frac{\sum_{i=1}^n W_{k_i,t}}{n}.$
and define $R_{0} =0$.
\State Find the largest $n \in \{0, 1, ..., |S_t|\}$ such that
$R_{n} \leq \alpha.$
\Ensure $S_{t+1} = \{k_1, ..., k_n\}$ if $n\geq 1$ and $S_{t+1} = \emptyset$ if $n = 0$.
\end{algorithmic}
\end{algorithm}
%\end{minipage}
}

{This algorithm contains three steps. In the first step, the stream-specific posterior probabilities are sorted in an ascending order. We tend to select the streams with small posterior probabilities into $S_{t+1}$, as 
they are more likely to be pre-change streams. In the second step, we calculate the cumulative averages of the sorted posterior probabilities. Finally, we find the largest $n$  such that the corresponding cumulative average is  no greater than $\alpha$. The corresponding streams will be kept in $S_{t+1}$ and the rest will be deactivated. The cumulative average of the $n$ streams gives the LFNR for $S_{t+1}$.}

%optimization is solved by sorting the posterior probability in an ascending order, calculating the cumulative average of the sorted posterior probability, finding the largest number of streams to keep the cumulative average  no greater than $\alpha$, and selecting the corresponding streams in $S_{t+1}$. }

The proposed one-step update rule controls the LFNR under the general model in \eqref{eq:general-model}, as formally described in Proposition~\ref{prop:control}. 
\begin{proposition}\label{prop:control}
Suppose that we obtain the index set $S_{t+1}$ using Algorithm~\ref{alg:one-step-rule1}, given
the index set $S_t$ and information filtration $\mathcal F_t$ at time $t$. Then the LFNR at time $t+1$ satisfies
$\expe\left(\frac{\sum_{k \in S_{t+1}} \ind( \tau_k<t) }{|S_{t+1}| \vee 1}\big\vert \mathcal F_t\right)\leq \alpha.$
% Moreover, suppose we apply Algorithm~\ref{alg:one-step-rule1} at every time $t=1,2,\dots$, then the resulting compound detection rule belongs to $\adset_{\alpha}$. That is, it controls the LFNR at level $\alpha$ at every time point.
\end{proposition}

\subsection{Proposed Compound Sequential Change-point Detection Procedure}

The proposed   procedure adaptively applies the above one-step update rule. {That is, at each time point $t$, we select the active set $S_{t+1}$ using Algorithm~\ref{alg:one-step-rule1},
given the information available at time $t$ including the current active set $S_t$ and the corresponding posterior probabilities $(W_{k,t})_{k\in S_t}$.}
This method is formally described in Algorithm~\ref{alg:SSS} below.
%{\color{red} verbal explanations of each step}
We will later refer to this procedure as $\TTp$.
%, where the superscript `$\strG$' stands for the `proposed' method.

% \begin{algo}\label{alg:SSS}
% Proposed Procedure  $\TTp$.
% %\vspace*{-12pt}
% \begin{enumerate}
%   \item[] {\bf Input}: Threshold $\alpha$.
% \item[] {\bf Step 1}: Let $S_1=\{1,\cdots,K\}$ and  $W_{k,1} = \prob(\tau_k < 1\vert \fil_1)$ for $k\in S_1$.
% \item[] {\bf Step 2}: For {$t=1, 2,3,\cdots$},
% input $\alpha$, $S_t$ and $(W_{k,t})_{k\in S_t}$ to Algorithm~\ref{alg:one-step-rule1}, and obtain
% $S_{t+1}$
% and $W_{k,t+1} = \prob(\tau_k < t+1\vert \fil_{t+1})$ for $k\in S_{t+1}$,
% where $\fil_{t+1} = \sigma(\fil_t, S_{t+1}, X_{k,t+1}, k\in S_{t+1})$.
% %{\color{red}Removed `$.^{\strG} $'. Should we keep it or not?}
% % :=G(W_{S^{\strG}_t,t}, S^{\strG}_t)$, where the mapping $G$ is defined in Algorithm~\ref{alg:one-step-rule}.
%  % \State $\TTp=(T^{\mathrm{p}}_{1},\cdots, T^{\mathrm{p}}_{K})$ with $T^{\mathrm{p}}_{k}= \sum_{t=1}^{\infty}I(k\in S_t)$.
% \item[] {\bf Output}: %$S^{\strG}_t$, $t = 1, 2, ...$.
% $\{S_t\}_{t=1,2,\cdots}$, or equivalently,
% $\TT^{*}=(T_1,\cdots, T_K)$, where %$T_k^{\strG} \in \mathbb Z_{+}$ is defined as
% $T_k= \sup\{t: k \in S_t\}.$
% \end{enumerate}
% \end{algo}

{\centering
%\begin{minipage}{1\linewidth}
\renewcommand{\algorithmicrequire}{\textbf{Input:}}
\renewcommand{\algorithmicensure}{\textbf{Output:}}
\begin{algorithm}[tbh]
 \caption{Proposed Procedure ($\TTp$).\label{alg:SSS}}
\begin{algorithmic}[1]
\Require  Threshold $\alpha$.
\State Let $S_1=\{1,\cdots,K\}$ and  $W_{k,1} = \prob(\tau_k < 1\vert \fil_1)$ for $k\in S_1$.
\State For {$t=1, 2,3,\cdots$},
input $\alpha$, $S_t$ and $(W_{k,t})_{k\in S_t}$ to Algorithm~\ref{alg:one-step-rule1}, and obtain
$S_{t+1}$
and $W_{k,t+1} = \prob(\tau_k < t+1\vert \fil_{t+1})$ for $k\in S_{t+1}$,
where $\fil_{t+1} = \sigma(\fil_t, S_{t+1}, X_{k,t+1}, k\in S_{t+1})$.
%{\color{red}Removed `$.^{\strG} $'. Should we keep it or not?}
% :=G(W_{S^{\strG}_t,t}, S^{\strG}_t)$, where the mapping $G$ is defined in Algorithm~\ref{alg:one-step-rule}.
% \State $\TTp=(T^{\mathrm{p}}_{1},\cdots, T^{\mathrm{p}}_{K})$ with $T^{\mathrm{p}}_{k}= \sum_{t=1}^{\infty}I(k\in S_t)$.
\Ensure %$S^{\strG}_t$, $t = 1, 2, ...$.
$\{S_t\}_{t=1,2,\cdots}$, or equivalently,
$\TTp=(T_1,\cdots, T_K)$, where %$T_k^{\strG} \in \mathbb Z_{+}$ is defined as
$T_k= \sup\{t: k \in S_t\}.$
\end{algorithmic}
\end{algorithm}
%\end{minipage}
}

%{\color{red}Earlier version:
%\begin{algorithm}[!h]
%  \caption{\color{red} Proposed Procedure ($\TTp$).\label{alg:SSS}}
%\begin{algorithmic}[1]
%\Require  Threshold $\alpha$.
%\State Let $S^{\strG}_1=\{1,\cdots,K\}$ and $\mathcal F^{\strG}_1 = \{X_{1,1}, ..., X_{K,1}\}$.
%\State For {$t=1, 2,3,\cdots$},
%input $\alpha$, $S^{\strG}_t$ and $\fil_t^{\strG}$ to Algorithm~\ref{alg:one-step-rule1} and obtain
%$S^{\strG}_{t+1}$.
%% :=G(W_{S^{\strG}_t,t}, S^{\strG}_t)$, where the mapping $G$ is defined in Algorithm~\ref{alg:one-step-rule}.
% % \State $\TTp=(T^{\mathrm{p}}_{1},\cdots, T^{\mathrm{p}}_{K})$ with $T^{\mathrm{p}}_{k}= \sum_{t=1}^{\infty}I(k\in S_t)$.
%\Ensure %$S^{\strG}_t$, $t = 1, 2, ...$.
%$\TT^{\strG}=(T_1^{\strG},\cdots, T_K^{\strG})$, where %$T_k^{\strG} \in \mathbb Z_{+}$ is defined as
%$T_k^{\strG}= \sup\{t: k \in S^{\strG}_t\}.$
%\end{algorithmic}
%\end{algorithm}
%}

Making use of Proposition~\ref{prop:control}, it is easy to show that the proposed procedure controls the LFNR at each step  under the general change-point model described in \eqref{eq:general-model}. This result is summarized in Proposition~\ref{prop:control2}.
\begin{proposition}\label{prop:control2}
Let $\TTp$ be defined in Algorithm~\ref{alg:SSS}. Then, $\TTp\in\adset_{\alpha}$.
\end{proposition}

\subsection{Calculation of Posterior Probabilities}\label{subsec:post}
% We start with the calculation of $\expe(\pcr_{t+1}|\mathcal{F}_{t})$.

The proposed update rule relies on the posterior probability
%\begin{equation*}
	$W_{k,t}=\prob\left(\tau_k<t|\mathcal{F}_{t}\right)$,
%\end{equation*}
which is the conditional probability of the change point has occurred to stream $k$ before the current time point $t$.
In general, this posterior probability depends on data from all the streams and thus its evaluation may be computationally intensive %{\color{blue}
when $K$ is large and $(\tau_1,\cdots,\tau_K)$ has a complex dependence structure. %}.
In that case, a Markov chain Monte Carlo methods may be needed for evaluating this posterior probability.
Under the special case of model {\mhomo} described in Example~\ref{example:mhomo}, this posterior probability is easy to evaluate using an iterative update rule as given in Lemma~\ref{lem:update} below.

\begin{lemma}\label{lem:update}
	Under model {\mhomo} described in Example~\ref{example:mhomo}, $W_{k,0}=0$ for $1\leq k\leq K$ and $W_{k,t}$ can be computed using the following update rule for $1\leq k\leq K$,
	\begin{equation}\label{eq:posterior-update}
		W_{k,t+1}= \begin{cases}
					\frac{q(X_{k,t+1})/p(X_{k,t+1})}{(1-\theta)(1-W_{k,t})/(\theta+(1-\theta)W_{k,t})+q(X_{k,t+1})/p(X_{k,t+1})}&\text{ if } 1\leq t\leq T_{k}-1,\\
					W_{k,T_{k}} & \text{ if } t\geq T_k.
		\end{cases}
	\end{equation}
%	{\color{red} need a careful check.}
\end{lemma}
We point out that the iteration in the above lemma is a slight modification of a classical result for Bayesian sequential change-point detection \citep{shiryaev1963optimum}. Indeed, with a single data stream,
the statistic $W_{k,t}$ is known to be the test statistic for the Shiryaev procedure, a sequential change-point detection procedure that has been proven the Bayes rule for minimizing the average detection delay while controlling the probability of false alarm.
%a change point in a single data stream.
A slight difference here is that $W_{k,t}$ stays the same after $T_k$ due to the control process that deactivates data streams.

\section{Theoretical Results}\label{sec:theory}

\subsection{Optimality Results}\label{sec:optimality}
In what follows, we establish
optimality results for the proposed one-step update rule
and the proposed procedure
%proposed procedure procedure
 $\TTp$,
%This result implies that local optimality leads to uniform optimality, under suitable conditions.
 under the optimality criteria given in Section~\ref{subsec:criteria}. % based on stream utilization.
The proposed update rule is locally optimal under the general change-point model   \eqref{eq:general-model}, following  Definition~\ref{def:local} for local optimality.
%We summarise this result in Proposition~\ref{prop:size} below.
%leads to the locally optimal $S_{t+1}$ at $t+1$, following following the notion of local optimality in Definition~\ref{def:local}.
%has the largest size among all possible subsets of $S_t$ that control the LFNR to be below $\alpha$.
%This result is described in Proposition~\ref{prop:size}.
\begin{proposition}\label{prop:size}
%For any $S \subset S_t$, satisfying
%$$\expe\left(\frac{\sum_{k \in S} \ind\left( \tau_k<t\right) }{|S| \vee 1}\big\vert \mathcal F_t\right)\leq \alpha,$$
%we have $\vert S_{t+1}\vert \geq \vert S\vert$.
Given LFNR level $\alpha$ and information filtration $\mathcal F_t$,
the index set $S_{t+1}$ given by Algorithm~\ref{alg:one-step-rule1} is locally optimal at time $t+1$.
%, following the notion of local optimality in Definition~\ref{def:local}.
\end{proposition}

In general, having local optimality in each step does not necessarily lead to uniform optimality and a uniformly optimal procedure may not even exist. However, Theorem~\ref{thm:uniformy optimality} below shows that
a uniformly optimal procedure exists under change-point model {\mhomo} and furthermore the proposed procedure is uniformly optimal. %That is,
In other words, in this case, a myopic decision rule that maximizes the next-step stream utilization under the LFNR constraint is also uniformly optimal throughout time.

%local optimality does lead to uniform optimality
%it is the case
%when change point model {\mhomo} holds.

%We first show that the proposed procedure $\TTp$ is uniformly optimal in  $\TTp\in\adset_{\alpha}$, where
%the notion of uniform optimality is given in Definition~\ref{def:uniform}.
\begin{theorem}\label{thm:uniformy optimality}
Under model {\mhomo}, the proposed method $\TTp$ is uniformly optimal in $\adset_{\alpha}$.
\end{theorem}

Although model {\mhomo} seems relatively simple, the uniform optimality result established in Theorem~\ref{thm:uniformy optimality} is highly non-trivial and requires non-standard technical tools for the proof, such as the monotone coupling on a partially ordered space for comparing stochastic processes of different dimensions.
%of stochastic processes on a partially ordered space.}
% the proof of Theorem~\ref{thm:uniformy optimality} is %surprisingly
%  very involved that requires technical tools on partial order spaces and stochastic ordering.
 Part of the challenge is from the compound nature of the problem. Below we intuitively explain
 %elaborate intuitions about
 why standard techniques for justifying the optimality of single-stream sequential change-point detection methods do not apply to our problem.
  %multiple sequential detection problem.
%relying on some new technical tools for characterizing our sequential scheme and the related data generation process.
Heuristically, for a given $t$, a larger value of $W_{k,t}=\prob(\tau_k\leq t-1|\mathcal{F}_t)$ suggests a
higher chance that a change point has already taken place for the $k$th data stream. This is why the proposed procedure chooses to detect streams with the largest posterior probabilities $W_{k,t}$. %On the other hand, however, it is likely that
Indeed, this update rule  {has been proven optimal for a single change detection problem under a Bayesian formulation \citep{shiryaev1963optimum} and is locally optimal according to Proposition~\ref{prop:size}.} However, the local optimality does not necessarily imply uniform optimality. To show uniform optimality, one needs to look into the future. More specifically, we need to deal with the situation where
a large value of  $W_{k,t}$ is due to random noise
and the posterior probability of the stream
 may become small at a future time point.
In other words, supposing that $W_{k_1,t}>W_{k_2,t}$,  we need to show that
it is more optimal to detect $k_1$ than $k_2$ at time $t$ under our optimality criteria,
even though $W_{k_1,t+s}<W_{k_2,t+s}$ can happen with  high probability for some $s >0$. To establish the uniform optimality, we need the $W_{k,t}$ process generated by the proposed procedure to have some stochastically monotone property. % to be discussed in the sequel.
%This property relies on the model {\mhomo} and may not hold under a more general model.
A proof sketch
for Theorem~\ref{thm:uniformy optimality} and a complete proof are given in the supplementary material, where some new techniques are established for the  monotone coupling of stochastic processes on a partially ordered space.

%detecting $k_1$ instead

%In other words, given $W_{k_1,t}>W_{k_2,t}$ for some $k_1,k_2,\in\{1,\cdots,K\}$,  the probability that $W_{k_1,t+s}<W_{k_2,t+s}$ is always positive for any positive $s$.

%\clearpage

%Thus, one may want to start to eliminate streams with larger values of $W_{k,t}$. On the other hand, given $W_{k_1,t}>W_{k_2,t}$ for some $k_1,k_2,\in\{1,\cdots,K\}$,  the probability that $W_{k_1,t+s}<W_{k_2,t+s}$ is always positive for any positive $s$. If we eliminate a stream because it has a larger $W_{k,t}$ value at an earlier time, we are not able to tell whether the $W_{k,t+s}$ value may become small in some future time. For this reason, we develop a monotone coupling showing that if we start to follow the proposed sorted Shiryaev procedure at some time, then the $W_{k,t}$ process will follow some stochastic monotone pattern after that time, in terms of a special partial order relationship on the union of low dimensional cubes defined below.

% It is worth pointing out the establishment of Theorem~\ref{thm:uniformy optimality} is non-trivial.
%In what follows, we provide two examples. Example~\ref{example:noopt} provides a model not in {\mhomo} for which
%there is no
%uniformly optimal procedure. Example~\ref{example:nonunique} shows a situation where the uniformly optimal procedure exists but is not unique.

%\begin{remark}\label{remark:relaxation}
In Theorem~\ref{thm:uniformy optimality}, the assumptions required by the model {\mhomo} may be relaxed. {By examining the current proof and the fact that the updating rule \eqref{eq:posterior-update} for the posterior probabilities can be extended to non-geometric priors, we believe that
the  uniform optimality can still be proved, if the change points are i.i.d. following some prior distribution with support $\{0,1,2,\dots\}$, for example, a negative binomial distribution. Similarly, the optimality results may be extended to the case where $p_{k,t}=p_t$ and $q_{k,t}=q_t$ for some time-dependent functions $p_t$ and $q_t$. }  On the other hand, we believe that it is necessary to assume the data streams $\{X_{k,t}\}_{t\geq 1}$ are identically distributed for different $k$ for the proposed method to be uniformly optimal.
Indeed, if the processes $\{X_{k,t}\}_{t\geq 1}$ are not identically distributed, then
 there may not exist a uniformly optimal procedure. One such example is given below.

%\end{remark}

%For models not in {\mhomo}, there may not be a uniformly optimal sequential procedure. One such example is given below.

\begin{example}[Non-existence of uniformly optimal procedure]\label{example:noopt}
Let $K=4$ and $\tau_k$s be independent, for $k = 1, 2,3, 4$. The change-point distributions satisfy
$\prob(\tau_k\geq 4)=0$ for $k=1,2,3, 4$. For $m=0,1,2,3$ and $k = 1, 2, 3, 4$,  the probabilities $\prob(\tau_k=m)$ are given below. In addition, let $X_{k,t}|t\leq \tau_k\sim \text{Bernoulli}(0.5)$ and $X_{k,t}|t>\tau_k\sim\text{Bernoulli}(0.51)$ for $k=1,2,3,4$. Finally, we set $\alpha=0.34$.
\begin{table}[h]
\centering
\begin{tabular}{c|cccc}
\hline
%\small
$\prob(\tau_k=m)$ & $m=0$  & $m=1$ & $m=2$ & $m=3$  \\ \hline
$k=1$                                            & 0.1  & 0 & 0   & 0.9    \\
$k=2$                                            & 0.4  & 0.6 & 0   & 0    \\
$k=3$                                            & 0.43 & 0.57   & 0   & 0 \\
$k=4$                                            & 0.55 & 0   & 0   & 0.45 \\ \hline
\end{tabular}
\end{table}
This model is not in {\mhomo}, as the change points are not identically distributed.
Enumerating all elements in $\adset_{\alpha}$, we have
$$
	\sup_{\TT\in\adset_{\alpha}}\expe\left(\uti_2(\TT)\right)=7
	\text{ and }	\sup_{\TT\in\adset_{\alpha}}\expe\left(\uti_4(\TT)\right)=10.
	% \sup_{\TT\in\adset_{\alpha}}\expe\left(\uti_3(\TT)\right) = 5 \text{ and } \sup_{\TT\in\adset_{\alpha}}\expe\left(\uti_4(\TT)\right)=6.
$$
However, there is no such a sequential procedure maximizing stream utilization at both $t=2$ and $t=4$. Consequently, there does not exist a uniformly optimal procedure in this example. %\yc{
The calculation for this example is provided in the supplementary material.%}
\end{example}
\begin{remark}
	We remark that a similar algorithm can be given for controlling $\text{LFDR}_t$ and in the meantime achieving a similar local optimality property. However, as the LFDR is calculated based on the stopped data streams rather than the active ones, the current techniques for proving uniform optimality no longer apply. The theoretical properties of the 
LFDR-control procedure is left for future investigation. 
\end{remark}
\subsection{Asymptotic Theory}
In modern multi-stream change-point detection problems, the number of data streams can be large.
% \footnote{{\color{red} To discuss here. I tend to say this section is mostly to help understand properties and deemphasize its usage in large-scale application. I'm not very sure if we can actually use the result here in practice..} Agreed. Revise the first paragraph of 4.2}
To enhance our understanding of the proposed method %investigate the behavior of the proposed method
in large-scale applications,
%In this section,
we study  the asymptotic properties of the proposed method when the number of streams $K$ goes to infinity.

We first study the structure of $\TTp$ under model \mhomo.
We define the following process
\begin{equation*}
	V_0=0 \text{ and } V_{t+1}=\frac{q(X_{1,t+1})/p(X_{1,t+1})}{(1-\theta)(1-V_{t})/(\theta + (1-\theta)V_t)+q(X_{1,t+1})/p(X_{1,t+1})},
\end{equation*}
where parameter $\theta$ and densities $p(\cdot)$ and $q(\cdot)$ are given by the model \mhomo. We further define $\lambda_{0}=1$ and
%\footnote{do we need $\lambda_0$ in (5)?}
\begin{equation}\label{eq:threholds}
	 \lambda_{t}=\sup\big\{\lambda: \lambda\in[0,1] \text{ and }\expe(V_{t}\mid V_{t}\leq \lambda, V_s\leq \lambda_s, 0\leq s \leq t-1)\leq \alpha
	\big\}
\end{equation}%\footnote{{\color{red}need modify the simulation part as the definition for $\lambda_t$ changed.} I have rerun the simulation and there is almost no difference in the result.}
%\yc{$\expe\left(V_{t}\Big|V_{t+1}\leq \lambda,V_{t}\leq \lambda_t\right)$}
for $t=1,2,\cdots.$ Theorem~\ref{thm:large-sample-iid} below shows that when $K$ grows to infinity, the proposed procedure  $\TTp$
converges to a limiting procedure $\TT^\dag$, for which the choice of index set $S_{t+1}^\dag$ is given by
$S_{t+1}^\dag = \big\{k \in S_t^\dag: W_{k,t}\leq \lambda_{t}\big\}.$
It suggests that when $K$ is large, we can replace the proposed procedure $\TTp$ by the limiting procedure $\TT^\dag$.
The latter is computationally faster, as the thresholds $\lambda_t$ can be computed offline and the updates for streams can be computed in parallel.
We make the following technical assumption.
\begin{itemize}
	\item [A1.] For $Z_1$ following density function $p(\cdot)$ and $Z_2$ following density function $q(\cdot)$, the likelihood ratios ${q(Z_1)}/{p(Z_1)}$ and ${q(Z_2)}/{p(Z_2)}$ have continuous and strictly positive density functions over $\mathbb{R}_+$ (with respect to the Lebesgue measure).
\end{itemize}
The above assumption is easily satisfied by continuous random variables. For example, it is satisfied when $p(\cdot)$
and $q(\cdot)$ are two normal density functions  with different means and/or variances.
% is the density function of the standard normal and $q(\cdot)$ is the density function of $N(\mu,\sigma^2)$ for $\mu\neq 0$ or $\sigma^2\neq 1$.

%while on the other hand, the sorting of $W_{k,t}$ and the calculation of $R_n$ in steps 1 and 2 of Algorithm~\ref{alg:one-step-rule1} can be computationally intensive when $K$ is large.

%the LFNR can be approximately controlled by taking the limiting procedure $\TT^*$.

\begin{theorem}\label{thm:large-sample-iid}
Assume that model {\mhomo} holds and Assumption A1 is satisfied.
 To emphasize the dependence on $K$, we denote the proposed procedure by $\TTp_K$, the corresponding information filtration at time $t$ by $\mathcal F_{K,t}^{*}$, and the index set at time $t$ by
$S_{K,t}^{*}$. Then, the following results hold for each $t \geq 1$.
%{\color{red} How about enumerate the results like this so that Theorem 2 and 3 have similar format?}{\color{blue}
\begin{enumerate}
	\item $\lim_{K\to\infty} \hat{\lambda}_{K, t}=\lambda_{t}$ a.s., where $\hat{\lambda}_{K, t} = \max\{W_{k,t}: k\in S_{K,t+1}^{*}\}$
is the threshold used by $\TTp_K$.
\item $\lim_{K\to\infty}\lpcr_{t+1}(\TTp_K)= \expe(V_{t}\mid V_{s}\leq \lambda_{s},0\leq s\leq t)$ a.s. Moreover,
\begin{equation}\label{eq:cond-expe-expr}
	\expe(V_{t}\mid V_{s}\leq \lambda_{s},0\leq s\leq t)
	=
		\begin{cases}
	1-(1-\theta)^{t},~~ t< \frac{\log(1-\alpha)}{\log(1-\theta)},\\
	\alpha, ~~ t\geq\frac{\log(1-\alpha)}{\log(1-\theta)}.
	\end{cases}
\end{equation}
\item $\lim_{K\to\infty} K^{-1}\vert S_{K,t+1}^{*}\vert =\prob\left(V_1\leq \lambda_1,\cdots, V_t\leq \lambda_t \right)$ a.s.
\end{enumerate}

\end{theorem}

We remark that according to the definition of $\lambda_t$ and the second statement of Theorem~\ref{thm:large-sample-iid},
when $t < {\log(1-\alpha)}/{\log(1-\theta)}$, 
%for
%{\color{blue} %$1\leq t< \frac{\log(1-\alpha)}{\log(1-\theta)}$,
$\lim_{K\to\infty}\lpcr_{t+1}(\TTp_K) <\alpha$ a.s. and no deactivation of streams is needed yet.
Otherwise, $\lim_{K\to\infty}\lpcr_{t+1}(\TTp_K)=\alpha$  a.s., which is achieved by deactivating suspicious  streams.

We also provide asymptotic theory for a special case of Example~\ref{example:model-partial} when the change points are completely dependent, i.e., $\tau_1=\cdots=\tau_K=\tau_0$.
%Completely Dependent Change Points.
%Theorem~\ref{thm:large-complete-dependent} below provides theoretical results for the case when
%We now consider the case where
%all the change points are completely dependent, i.e., $\tau_1=\cdots=\tau_K=\tau_0$.
We make the following assumption.
	\begin{itemize}
		\item [A2.] For $Z_1$ following density $p(\cdot)$ and $Z_2$ following density $q(\cdot)$, the density functions satisfy $\expe\left(\log\frac{p(Z_1)}{q(Z_1)}\right)>0, \expe\left(\log\frac{q(Z_2)}{p(Z_2)}\right)>0$, $\expe\left(\log\frac{p(Z_1)}{q(Z_1)}\right)^2<\infty$,  and $\expe\left(\log\frac{q(Z_2)}{p(Z_2)}\right)^2<\infty.$
		% \begin{equation*}
		% \begin{split}
		% 			&\expe\left(\log\frac{p(Z_1)}{q(Z_1)}\right)>0, \expe\left(\log\frac{q(Z_2)}{p(Z_2)}\right)>0,\\
		% 				&\expe\left(\log\frac{p(Z_1)}{q(Z_1)}\right)^2<\infty, \text{ and }	\expe\left(\log\frac{q(Z_2)}{p(Z_2)}\right)^2<\infty.
		% \end{split}
		% 		\end{equation*}
	\end{itemize}
	Note that $\expe\left(\log({p(Z_1)}/{q(Z_1)})\right)$ and $\expe\left(\log({q(Z_2)}/{p(Z_2)})\right)$ are the Kullback-Leibler divergence between $p(\cdot)$ and $q(\cdot)$. Requiring them to be positive is the same as requiring $p(\cdot)$ and $q(\cdot)$ to be densities of two different distributions.
	% \footnote{\yc{To be more clear, can we write $\expe\left(\log\frac{p(Z_1)}{q(Z_1)}\right)^2<\infty, \text{ and }	\expe\left(\log\frac{q(Z_2)}{p(Z_2)}\right)^2<\infty$?}}
	
%In particular, we focus on the special case, where  $\tau_0\sim Geom(\theta)$.
%Under the special case where $\tau_0$ follows a geometric distribution, we have the following result.

% this special model, we have the following result.
% A special case of the model described in Example~\ref{example:model-partial} is $\tau_1=\cdots=\tau_K=\tau_0$, and $\tau_0\sim Geom(\theta)$. In this case, the change points across different data streams are completely dependent.
\begin{theorem}\label{thm:large-complete-dependent}
Suppose that data follow a special case of the model given in Example \ref{example:model-partial} when $\eta = 1$ and
$\tau_0\sim Geom(\theta)$, and Assumption A2 holds.
 Let
	\begin{equation*}
		W_t=\prob(\tau_0 < t\mid X_{k,s},1\leq k\leq K,1\leq s\leq t),\quad 		T=\min\{t: W_t>\alpha\}.
	\end{equation*}
	
	% \begin{equation*}
	% \end{equation*}
	Then, $\TTp_K=(T,\cdots, T)$.
Moreover, the following asymptotic results hold.
	\begin{enumerate}
		\item $\lim_{K\to\infty} (T-\tau_0) =1$ a.s.,
		\item $\lim_{K\to\infty}\lpcr_{t+1}(\TTp_K)=0$ a.s., %\footnote{Earlier: $\lim_{K\to\infty}\expe(\pcr_{t+1}(\TTp)|\mathcal{F}_t)=0$ a.s.,}
	\item $ \lim_{K\to\infty} K^{-1}\vert S_{K,t+1}^{*}\vert = \ind(\tau_0\geq t)$ a.s.
	\end{enumerate}
% 	define a stochastic system $V_t\in [0,1]$ and $\gamma_t\in [0,1]$ starting from $V_0=0$, $\gamma_0=1$, and evolving according to
% \begin{equation}
% \begin{cases}
% 		V_{t+1}|V_{t}=v_t&\sim \kappa_{\gamma_{t}}(v_t,\cdot)\\
% 		\gamma_{t+1}& = (\frac{\alpha}{v_t}\wedge 1) \gamma_{t}.
% \end{cases}
% \end{equation}
% {\color{red} the kernel to be specified.}
% Then, for any $t\geq 1$,
% \begin{enumerate}
% 	\item $\lim_{K\to\infty}\expe(\pcr_{t+1}(\TTp)|\mathcal{F}_t)= V_{t}\wedge \alpha$ in distribution.
% 	\item $\lim_{K\to\infty} K^{-1}\uti_t(\TTp)= \gamma_t$ in distribution.
% \end{enumerate}
\end{theorem}
%{\color{blue}
According to the above theorem, the detection time in the proposed procedure is the same for all the data streams. This detection rule is the same as the classical Shiryaev procedure \citep{shiryaev1963optimum} for a single data stream. It thus shares all the optimality properties of the Shiryaev procedure. %}
We further remark that the last limit in the above theorem is non-degenerate in the sense that it is a Bernoulli random variable, rather than a constant as in Theorem~\ref{thm:large-sample-iid}.

{\section{Additional Theoretical Results}\label{sec:additional}
In this section, we  give extensions of Theorem~\ref{thm:uniformy optimality}. 
%which is a main theoretical result of this paper.
We first extend the uniform optimality result in Theorem~\ref{thm:uniformy optimality} to two other performance measures, $\text{RL}_t(\TT)=\sum_{k=1}^K (T_k\wedge \tau_k\wedge t)
$ and $\textrm{CD}_t=K - |S_t|$, as discussed in Section~\ref{subsec:criteria}.
% The next proposition extends Proposition~\ref{prop:control} and shows that the proposed method is locally optimal with the utility measures $\textrm{RL}_t$ and $\textrm{CD}_t$.
% {\color{red} how about only uniform optimality}
% \begin{proposition}\label{prop:size-other-measure}
% Given LFNR level $\alpha$ and information filtration $\mathcal F_t$,
% the index set $S_{t+1}$ given by Algorithm~\ref{alg:one-step-rule1} satisfies
% \begin{equation}
% 	\expe(\textrm{RL}_{t+1}(\TT^*)|\fil_t) \geq \expe(\textrm{RL}_{t+1}(\TT)|\fil) \text{ and } \expe(\textrm{CD}_{t+1}(\TT^*)|\fil_t) \geq \expe(\textrm{CD}_{t+1}(\TT)|\fil_t)
% \end{equation}
% for any $\TT\in\adset_{\alpha}$.
% \end{proposition}

\begin{theorem}\label{thm:uniform-optimal-other}
 Under model {\mhomo}, the following equations hold for all $t$,
 %the proposed method $\TTp$ satisfies
\begin{equation}
	\expe(\textrm{RL}_t(\TTp))=\sup_{\TT\in \adset_{\alpha}} \expe(\textrm{RL}_t(\TT)) \text{ and } \expe(\textrm{CD}_t(\TTp))=\inf_{\TT\in \adset_{\alpha}} \expe(\textrm{CD}_t(\TT)).
\end{equation}
%for all $t$.
\end{theorem}
}
% {\color{blue}Similar to Proposition~\ref{prop:size}, the proposed method is also locally optimal under the performance measure $\textrm{RL}_t$ and $\textrm{CD}_t$. More details are provided in  Section {\color{red}XXX} of the online supplement. }
%\yc{Add that the proposed procedure is also locally optimal when using RL or the number of detections as performance metric. The proposed procedure is also uniformly optimal when using the number of detections as performance metric. }

We then extend Theorem~\ref{thm:uniformy optimality} by investigating a comparison between an arbitrary sequential procedure in $\adset_{\alpha}$ and a procedure which switches from this procedure to the proposed procedure after a certain time point. This result provides further insights into the proposed procedure.
Specifically, we use
$\TT^{\strA}\in\adset_{\alpha}$ to denote an arbitrary sequential procedure which controls the LFNR. %We further denote
%$S_t^{\strA}, t = 1, 2, ...$ as the index set process for procedure $\TT^{\strA}$.
We further consider a procedure
$\TT^{\strGfull{t_0}}$, which takes the same procedure as $\TT^{\strA}$ for $t = 1, ..., t_0$. After time $t_0 + 1$ and onwards, each step of
$\TT^{\strGfull{t_0}}$ follows the proposed update rule in Algorithm~\ref{alg:one-step-rule1}.
Theorem~\ref{thm:comparison} compares four sequential procedures, including
$\TT^{\strA}$, $\TT^\strGfull{t_0}$, $\TT^\strGfull{t_0+1}$, and $\TTp$.

%
%\clearpage
%The proof of Theorem~\ref{thm:uniformy optimality} is involved.
%\footnote{{\color{red}How about switching the order of Theorem 4 and 5, or it doesn't matter? Theorem 4 might distract the attention a little. Or maybe say:} Theorem~\ref{thm:uniformy optimality} can be extended to Theorem~\ref{thm:comparison} below on ....  }
%To prove Theorem~\ref{thm:uniformy optimality}, we prove a stronger result in Theorem~\ref{thm:comparison} below on the comparison between an arbitrary sequential procedure in $\adset_{\alpha}$ and a procedure which switches from this procedure to the proposed procedure after a certain time point. Specifically, we use
%$\TT^{\strA}\in\adset_{\alpha}$ to denote an arbitrary sequential procedure which controls the LPNR. %We further denote
%%$S_t^{\strA}, t = 1, 2, ...$ as the index set process for procedure $\TT^{\strA}$.
%We further consider a sequential procedure
%$\TT^{\strGfull{t_0}}$, which takes the same procedure as $\TT^{\strA}$ for $t = 1, ..., t_0$. After time $t_0 + 1$ and onwards, each step of
%$\TT^{\strGfull{t_0}}$ follows the proposed update rule in Algorithm~\ref{alg:one-step-rule1}.

\begin{theorem}\label{thm:comparison}
	Let $\TT^{\strA}\in \adset_{\alpha}$ be an arbitrary sequential procedure. Further let
$\TT^{\strGfull{t_0}}$ and $\TT^{\strGfull{t_0+1}}$ be the switching procedures described above, with switching time $t_0$ and $t_0 + 1$, respectively, for some $t_0\geq0$.
Then, for all $t=1,2,\cdots$,  $\TT^{\strGfull{t_0}}, \TT^{\strGfull{t_0+1}} \in \adset_{\alpha}$ and under model {\mhomo}
\begin{equation*}
	 		\expe\left(\uti_t(\TT^{\strA})\right)\leq\expe\left(\uti_t({\TT^\strGfull{t_0+1}})\right)\leq  \expe\left(\uti_t({\TT^\strGfull{t_0}})\right)\leq \expe\left(\uti_t(\TTp)\right).
	 	\end{equation*}
	 	
%\footnote{{\color{red} removed the "admissibility part" as the proposed method is already proved to be local optimal. }}
%$\TT^{\strGfull{t_0+1}}$ be defined by Algorithm~\ref{alg:combine},  $\TTp$ be defined by Algorithm~\ref{alg:SSS}, and $t_0\geq1$, then we have
	% \begin{enumerate}
	%  	\item [i)] $\expe\left(\uti_t\left(\TT^{\strA}\right)\right)=\expe\left(\uti_t\left(\TT^{\strGfull{t_0}}\right)\right)$ for $t\leq t_0$.
	%  	\item [ii)] There exists $s>t_0$ such that $ \expe\left(\uti_s\left(\TT^{\strA}\right)\right)\leq\expe\left(\uti_s\left({\TT^\strGfull{t_0}}\right)\right)$.
	%  	\item [iii)]Under model \mhomo, we have
	%  	\begin{equation}
	%  		\expe\left(\uti_t(\TT^{\strA})\right)\leq\expe\left(\uti_t({\TT^\strGfull{t_0}})\right)\leq  \expe\left(\uti_t({\TT^\strGfull{t_0+1}})\right)\leq \expe\left(\uti_t(\TTp)\right),
	%  	\end{equation}
	%  	for all $t=1,2,\cdots$.
	%  \end{enumerate}
\end{theorem}
%\yc{Shouldn't this be \begin{equation*}
%	 		\expe\left(\uti_t(\TT^{\strA})\right)\leq\expe\left(\uti_t({\TT^\strGfull{t_0+1}})\right)\leq  \expe\left(\uti_t({\TT^\strGfull{t_0}})\right)\leq \expe\left(\uti_t(\TTp)\right),
%	 	\end{equation*}}

The above theorem implies that,
%the switching scheme $\TT^{\strGfull{t_0+1}}$ is a `safer' choice compared to $\TT^{\strA}$. Moreover,
under model \mhomo, $\TT^{\strGfull{t_0}}$ is uniformly better than
$\TT^{\strA}$. %In addition, it is never too late to switch to the proposed scheme to optimize the stream utility.
It also %shows that it is better to switch to the proposed update rule sooner than later.
suggests to switch to the proposed procedure as soon as possible, if one cannot use the proposed procedure at the beginning due to practical constraints. %Theorem~\ref{thm:comparison} is implied by the following proposition. % which shows that a stronger result holds.
%We prove Theorems~\ref{sec:optimality} and \ref{thm:comparison} by proving the following theorem
Theorems~\ref{thm:uniformy optimality} and \ref{thm:comparison} are implied by the next theorem.
 %which states a stronger result.}
%, for which a more stronger result is
%\yc{I changed the proposition to a theorem, as it is a stronger result. }

\begin{theorem}\label{prop:stronger}
Suppose that %the same conditions as in Theorem~\ref{thm:comparison} hold.
model {\mhomo} holds.
	For any $t_0, s\geq 0$ and any sequential detection procedure $\TT^{\strA}\in\adset_{\alpha}$, let
	$\fil^\strA_t$ be the information filtration  and $S^{\strA}_t$ be the set of active streams at time $t$ given by $\TT^{\strA}$.
	Then,
\begin{equation}\label{eq:induction-statement}
		\expe\left[|S_{t_0+s}^{\strA}|\middle| \fil^\strA_{{t_0}}\right]
				\leq \expe\left[|S_{t_0+s}^{\strGfull{t_0}}|\middle| \fil^\strA_{t_0}\right] \text{ a.s.}
	\end{equation}
%\footnote{{\color{red} OLD WRITING:	\begin{equation}\label{eq:induction-statement}
%		\expe\left[\uti_{t_0+s}(\TT^{\strA})\big| \mathcal{F}^\strA_{{t_0}}\right]
%				\leq \expe\left[\uti_{t_0+s}\left(\TT^{\strGfull{t_0}}\right)\big| \mathcal{F}^\strA_{t_0}\right] \text{ a.s.}
%	\end{equation}
%	}}
\end{theorem}
\section{Numerical Experiment}\label{sec:numerical}

We evaluate the proposed procedure via a simulation study under the change-point model
{\mhomo}.  Two stream sizes $K = 50$ and 500 are considered, representing problems of different scales.
For all the data streams, we let the pre- and post-change distributions be $N(0, 1)$ and $N(1,1)$, respectively.
We consider two settings for the change-point distribution, with $\theta = 0.01$ and $0.05$
in the geometric distribution, respectively. We set the threshold to be $\alpha = 0.05$ for the control of LFNR. The
combinations of $K$ and $\theta$ lead to four different settings. For each setting, we run 5000 independent replications.

{We consider two procedures, including (1) the adaptive procedure given in Algorithm~\ref{alg:SSS} and (2) a procedure in which a stream $k$ is deactivated if the posterior probability $W_{k,t}$ is greater than the non-adaptive threshold $\lambda_t$ (see  \eqref{eq:threholds}) given by the asymptotic results. The non-adaptive threshold $\lambda_t$ is approximated via a simulation with 1,000,000 streams.

We evaluate these procedures by (1) mean  FNP, (2) mean LFNR,  (3)  mean number of active streams, and (4) mean stream utilization, at each time point.
These values are obtained by averaging over the 5000 independent replications. For example, for each simulation, we can calculate the FNP at each time point following equation \eqref{eq:pcr}. The mean FNP at each time point is then calculated by averaging the corresponding FNP values from the 5000 independent simulations under each setting. The other metrics are calculated similarly.
The results are given in Figures~\ref{fig:1} through \ref{fig:4} that correspond to the settings (1) $K = 50$, $\theta = 0.01$, (2) $K = 50$, $\theta = 0.05$, (3) $K = 500$, $\theta = 0.01$, and (4) $K = 500$, $\theta = 0.05$, respectively. We discuss these results below.

\begin{figure}
  \centering
  \includegraphics[scale = 0.35]{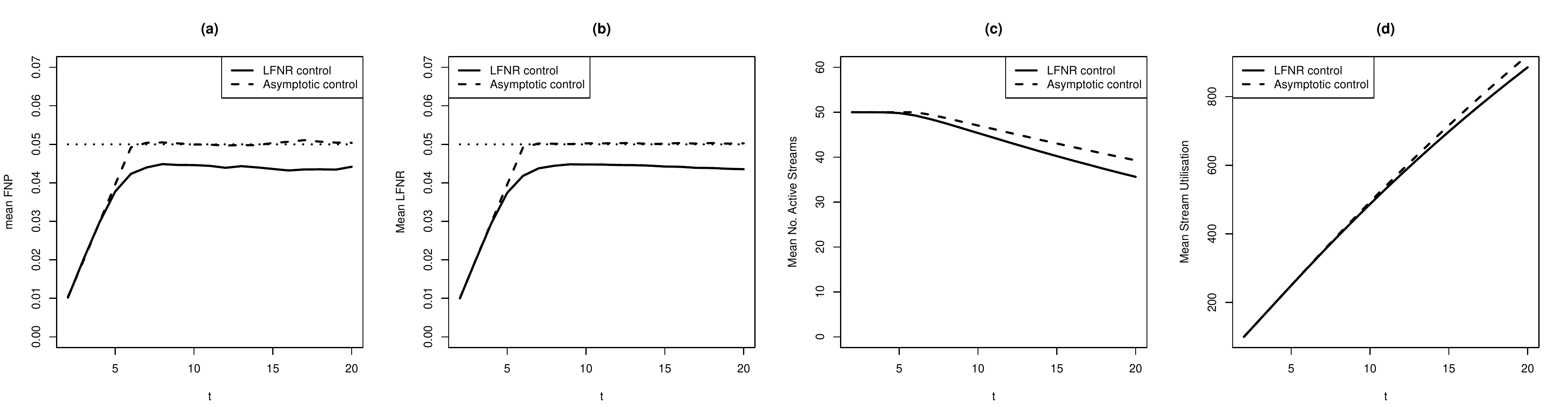}
  \caption{\footnotesize Results under the setting when $K = 50$ and  $\theta = 0.01$. Panels (a) through (d) correspond to the four metrics, (1) mean  FNP, (2) mean LFNR,  (3)  mean number of active streams, and (4) mean stream utilization, respectively. }\label{fig:1}
\end{figure}
\begin{figure}
  \centering
  \includegraphics[scale = 0.35]{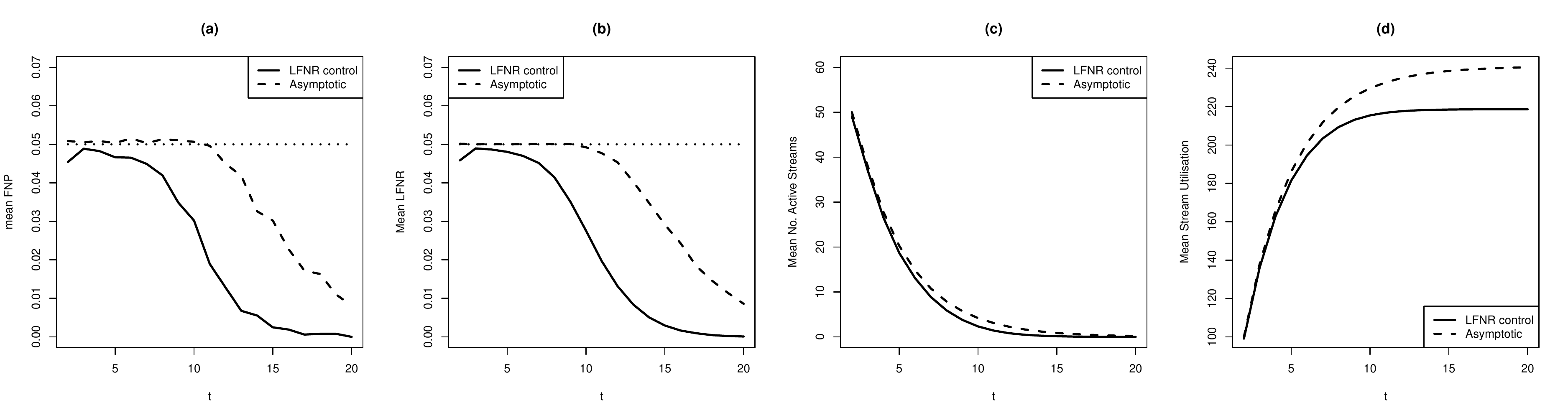}
  \caption{\footnotesize Results under the setting when $K = 50$ and  $\theta = 0.05$. The four panels show the same metrics as in Figure~\ref{fig:1}. }\label{fig:2}
% \end{figure}
% \begin{figure}
  \centering
  \includegraphics[scale = 0.35]{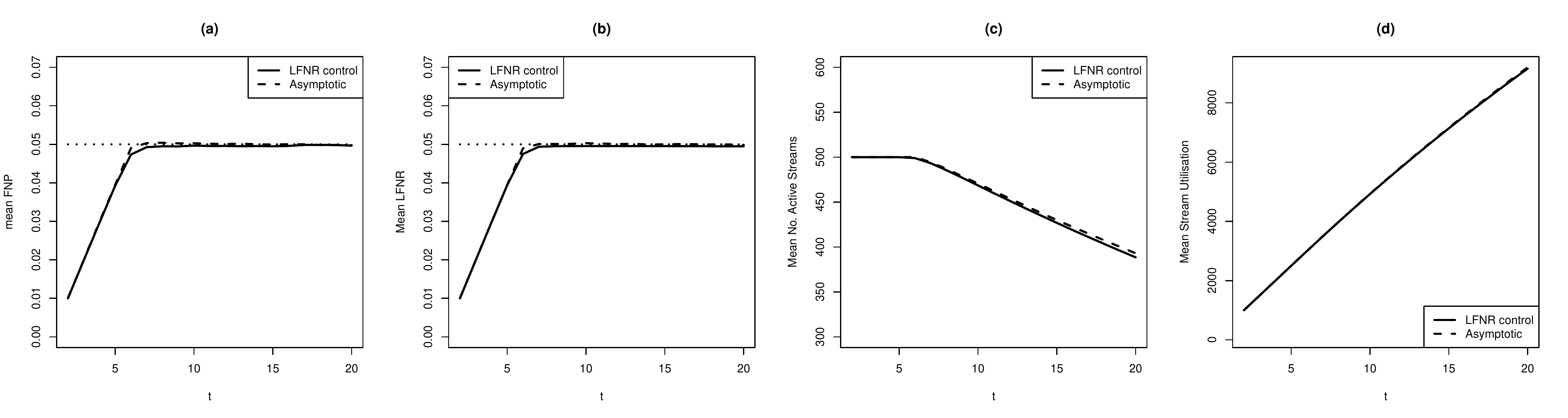}
  \caption{\footnotesize Results under the setting when $K = 500$ and  $\theta = 0.01$. The four panels show the same metrics as in Figure~\ref{fig:1}. }\label{fig:3}
% \end{figure}
% \begin{figure}
 \centering
  \includegraphics[scale = 0.35]{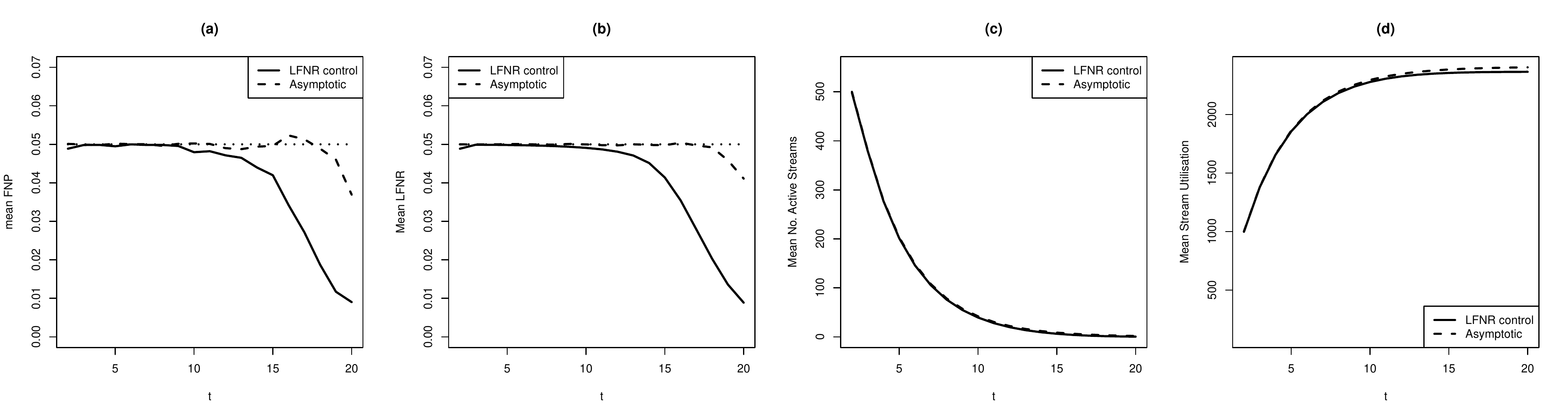}
  \caption{\footnotesize Results under the setting when $K = 500$ and  $\theta = 0.05$. The four panels show the same metrics as in Figure~\ref{fig:1}.  }\label{fig:4}
\end{figure}
First,  for the proposed adaptive procedure, the mean FNP and mean LFNR are always below the 0.05 threshold under all the four settings, suggesting that the risk of the active streams is well-controlled at the aggregate level; see Panels (a) and (b)  of the figures. The control of these quantities is a direct result of the proposed procedure controlling LFNR at every time point. More specifically, when $\theta = 0.05$, the data streams change relatively more quickly than the case when $\theta = 0.01$. In that case, the proportion of post-change streams quickly exceed 0.05 and the proposed procedure controls this proportion to be slightly below 0.05 by deactivating the ones with the highest post-change posterior probabilities. As time goes on, the mean FNP will decay towards zero, as the number of active streams decays to zero; see Panels (c) and (d) of Figures~\ref{fig:2} and \ref{fig:4}. When $\theta = 0.01$, the data streams change at a much slower rate. Thus, at the beginning, the proportion of post-change streams among the active ones tends to be smaller than 0.05 and the proposed procedure does not tend to detect and deactivate any streams. When the proportion of post-change streams accumulates to be above 0.05, the proposed procedure starts to deactivate changed streams to control the proportion to be around the targeted level. As change points occur more slowly, the number of active streams at any given time point tends to be larger than that when $\theta = 0.05$. See Figures~\ref{fig:1} and \ref{fig:3} for more details.

Second, as we can see from Panels (a) and (b) of the figures,
the non-adaptive procedure based on the asymptotic theory also controls the mean FNP and the mean LFNR to be near or below the targeted level, though the mean FNP may be slightly larger than the targeted level occasionally under the settings when $K = 50$. It tends to be slightly more aggressive than the adaptive procedure, because the LFNR can sometimes exceed the targeted threshold $\alpha$. Overall, the non-adaptive procedure also performs well, in the sense that it tends to control FNR at the targeted level $\alpha$ (i.e. the expected value of LFNR) at all time points, even though the LFNR itself is not exactly controlled.

Finally, we see that the two procedures tend to perform more similarly when the number of active streams is larger, as the non-adaptive procedure is the limiting case of the adaptive procedure when the number of streams grows to infinity.
More specifically, comparing the setting when $K = 500$ (Figures~\ref{fig:3} and \ref{fig:4}) with that when $K = 50$ (Figures~\ref{fig:1} and \ref{fig:2}) , we see that the two procedures are closer to each other when $K = 500$. For the same value of $K$, the two procedures tend to be more similar under the setting when $\theta= 0.01$ than that when $\theta = 0.05$, as data streams change more slowly  and thus there tend to be more active streams at every time point when $\theta = 0.01$. Moreover, for each setting, the two procedures tend to behave more similarly when $t$ is smaller, as the number of active streams decays with time $t$.

}

\section{Discussions}\label{sec:diss}

Motivated by real-world applications from various fields including education, engineering, and finance,
%, specifically the sequential detection of compromised items in educational
%testing,
we propose a compound decision framework for Bayesian sequential change-point detection in parallel data streams. An easy-to-implement procedure is proposed, for which theoretical properties are established. Specifically, under a class of change-point models, the proposed procedure is shown to be uniformly optimal in a non-asymptotic sense. Numerical experiments show that the proposed procedure can accurately control the aggregated risk of active streams.

%Numerical results show that the proposed procedure performs better than running a Bayesian
%changepoint detection algorithm independently for individual streams.

The current work can be extended along several directions.
First,
different optimality criteria may be considered and
the proposed procedure can be extended accordingly. For example, different streams may have different weights
due to their unequal importance in practice. In that case, more general definitions of local false non-discovery rate and
stream utilization measure can be given, for which a tailored sequential procedure can be derived.

Second, in some real applications, the change-point distribution  and the distributions for pre- and post-change data  may not be known in advance.
%Under suitable models, these distributions can be consistently estimated when the number of streams $K$ goes to infinity.
This problem may be handled by parameterizing the pre- and post-change distributions and then use a full or empirical Bayes approach that combines the proposed procedure   with sequential estimation of the unknown parameters.  Optimality theory may be established when the number of streams $K$ goes to infinity.

 %\footnote{\yc{This seems like a reinforcement learning procedure. We may be able to design such an algorithm that controls the LFNR or similar metrics. That would be interesting even if we are not able to show optimality. }}
% \yc{Do you want to discuss other types of actions and early stopping here?}

%Under the asymptotic sense that both the number of streams $K$ and time $t$ goes to infinity, these distributions
%For single-stream change detection, the unknown change
%There are two approaches
%There may be two approach
%There are two directions to proceed when these distributions are unknown. The first approach is to develop a frequentist framework for compound sequential change detection, for which

Third, optimal sequential procedures remain to be developed %when change points are dependent.
under reasonable models for dependent change points. In particular, in many multi-stream change detection problems, the change points may be driven by a low-dimensional latent process, which can be described by a dynamic latent factor model.
Several questions remain to be answered under such a change-point model,
including the existence of a uniformly optimal procedure and
the construction of the uniformly optimal procedure if it exists.

%the optimality of the proposed method.
% Does a uniformly optimal procedure exist?
%If so, is the proposed procedure still uniformly optimal?

 %where certain exchangeability property holds
%In future research, the following questions will be investigated.

Finally, a more general setting may be considered that allows new data streams to be added dynamically. For example,  in educational testing, once an item is removed from the item pool, a new one needs to be developed to maintain the size of the pool. The inclusion of new data streams changes the information filtration. Under the new information filtration which contains information from both the original and new streams, a locally optimal procedure can be developed under similar optimality criteria.
However, it is unclear whether this procedure is still uniformly optimal. This problem is worth future investigation.
\section*{Acknowledgments}
We would like to thank the editors and two referees for their helpful and constructive comments. Xiaoou Li’s research was partially supported by the NSF grant DMS-1712657. 

\section*{Supplementary material}
\label{SM}
The supplementary material includes proofs of the theoretical results.
\appendix
% \begin{center}
% 	{\Large\bf Compound Sequential Change Point Detection in Multiple Data Streams -- Supplementary Material}
% \end{center}
%\maketitle

\onehalfspacing
% We organize proofs as follows. Section~\ref{sec:notation} lists several notation that will be used from place to place in the proof. The proof for Theorem~\ref{prop:stronger} is presented in Section~\ref{sec:prop:stronger}, where the proof for its key technical steps including Proposition~\ref{prop:monotone-o} and Proposition~\ref{prop:CouplingForProp} are given in Section~\ref{sec:proof-proposition-ordering}. The proof of other theoretical results

% For readers' convenience, we list the contents below.

%We present proof for the theoretical results in the appendix, which is organized according to the following list.

\vspace{1cm}
%\tableofcontents

\newpage
 \begin{center}
 	 {\Large \bf Supplement to `Compound Sequential Change-point Detection in Parallel Data Streams'}
 \end{center}

\section{Notations}\label{sec:notation}
For the readers' convenience, we provide a list of notations below. {We will also restate these notation when they first appear in the proof.}
\begin{itemize}
	\item $X_{S,t}$ for some set $S\subset\{1,\cdots,K\}$:  $X_{S,t}=(X_{k,t})_{k\in S}$.
	\item $X_{k,s:t}$: $X_{k,s:t}=(X_{k,r})_{s\leq r\leq t}$.
	\item $\TT^{\strA}$: an arbitrary sequential procedure.
	\item $S_t^{\strA}$: The set of active streams at time $t$ given by procedure $\TT^{\strA}$.
	\item $\fil^{\strA}_t$: $\sigma$-field of information obtained up to time $t$ following $\TT^{\strA}$.
	\item  $W^{\strA}_{k,t}$: posterior probability $\prob(\tau_k<t|\fil^{\strA}_t)$ at the $k$-th stream following $\TT^{\strA}$ at time $t$.
	\item $W^{\strA}_{S,t}$ for some set $S\subset\{1,\cdots, K\}$: $W^{\strA}_{S,t}=(W^{\strA}_{k,t})_{k\in S}$.
	\item $\TTp$: the proposed sequential procedure.
	\item $S^{\strG}_t$, $\fil^{\strG}_t$, $W^{\strG}_{k,t}$, $W^{\strG}_{S,t}$ are defined similarly for procedure $\TTp$.
	\item $\TT^{\strGfull{t_0}}$: the sequential procedure that
takes the same steps as $\TT^{\strA}$
up to time $t_0$ (meaning $S^{\strGfull{t_0}}_{t}=S^{\strA}_t$ for $1\leq t\leq t_0$) and updates by  Algorithm~\ref{alg:one-step-rule1} from time $t_0+1$ and onward.
	\item $S^{\strGfull{t_0}}_t$, $\fil^{\strGfull{t_0}}_t$, $W^{\strGfull{t_0}}_{k,t}$, $W^{\strGfull{t_0}}_{S,t}$ are defined similarly for procedure $\TT^{\strGfull{t_0}}$.
	\item $\eqd$: equal in distribution.
	\item $\ab$: a vector with zero length.
	\item $\card$: length of a vector, where $\card(\ab)=0$.
	\item $Z\sim N(0,1)$: the notation `$\sim$' means that the left side follows the distribution on the right side.
\end{itemize}
\section{Proof Sketch}\label{sec:proof-sketch}
{In this section, we discuss the main steps and techniques for proving
Theorem~\ref{prop:stronger}   through an induction argument.} Its proof is involved, relying on some monotone coupling results on stochastic processes living in a special partially ordered space. %\footnote{{\color{red}Earlier version}: in a highly complex space}
In what follows, we give a sketch of the proof to provide more insights into the
proposed procedure. %\footnote{\yc{Remove ``When $s=0$, it is trivial that \eqref{eq:induction-statement} holds."? I don't see the point of having this sentence here.}}
When $s=0$, it is trivial that \eqref{eq:induction-statement} holds. The induction is %to show
 %The technical challenges lie in the induction step, which is
 to show that
for any   $\TT^\strA\in\adset_{\alpha}$ and any $t_0$,  \eqref{eq:induction-statement} holds for
  $s = s_0 +1$,
  assuming that it holds for $s \leq s_0$.
 The induction step is proved by the following three steps.
 % \footnote{\yc{It is unclear to me how the induction connects to these three steps. These three steps do prove (7), but how does it connect to the induction?}}
\begin{enumerate}
	\item Show that $\TT^\strGfull{t_0+1}$ is `better' than $\TT^\strA$ conditional on $ \mathcal{F}^\strA_{t_0}$.
	\item Show that $\TT^\strGfull{t_0}$ is `better' than $\TT^\strGfull{t_0+1}$ conditional on $ \mathcal{F}^\strA_{t_0}$.
	\item  Show that $\TT^\strGfull{t_0}$ is `better' than $\TT^\strA$ conditional on $ \mathcal{F}^\strA_{t_0}$ by combining the first two steps.
\end{enumerate}

Here, we say a procedure is `better than' the other, if its conditional expectation of  the size of index set at time $t_0+s+1$
% \yc{at any further time} %\footnote{ removed: $\uti_{t_0+s_0+1}$}
is no less than that of the other, given the information filtration $\mathcal{F}^\strA_{t_0}$.
%the word `better' is regarding to the conditional expectation of $\uti_{t_0+s+1}$ of different sequential detection-and-elimination schemes.
Roughly, we prove the first step by
replacing $t_0$ with $t_0+1$ in the
induction assumption and taking conditional expectation  given $\fil^{\strA}_{t_0}$,
%\footnote{\yc{This statement ``we prove the first step by the induction assumption" is strange. I don't understand. }}
 and prove the third step by combining the first and second steps.
% by a careful analysis of the sequential scheme and the corresponding  data generation process.
The main technical challenge lies in the second step, for which we develop several technical tools. Among these tools, an important one is the following monotone coupling result regarding a special partial order relationship.

We define a partially ordered space $(\spaceo,\lleq)$ as follows. Let
\begin{equation}
	 \spaceo=\bigcup_{k=1}^K \big\{\vv=(v_1,\cdots,v_k)\in[0,1]^k: 0\leq v_1\leq \cdots v_k\leq 1
		\big\}\cup \{\ab\},
\end{equation}
where $\ab$ represents a vector with zero length.
For $\uu\in\spaceo$, let $\card(\uu)$ be the length of the vector $\uu$.

%{\color{blue}
\begin{definition}
	For $\uu,\vv\in\spaceo$, we say $\uu\lleq
 \vv$ if $\card(\uu)\geq \card(\vv)$ and $u_i\leq v_i$ for $i=1,...,\card(\vv)$. In addition, we say $\uu\lleq \ab$ for any $\uu\in\spaceo$.
\end{definition}
%}
% {\begin{remark}
% The above definition provides a partial order relation for comparing ordered vectors that can be further used to compare unordered vectors by comparing their order statistics.  For  ordered vectors, the partial order relation defined above satisfies the antisymmetric requirement ($\uu\lleq\vv$ and $\vv\lleq\uu$ implies $\uu=\vv$). Alternatively, one can define an equivalent class over the unordered vectors (two vectors are equivalent if they have the same order statistics) and then define a partial order relation over the equivalent class. Then, stochastic ordering results can be equivalently developed for unordered vectors.
% In this paper, we choose to define partial order relation for ordered vectors for the ease of exposition.
% \end{remark}
% }

To emphasize the dependence on the sequential procedure, we use $S^{\strA}_{t}$ and $\fil^{\strA}_t$ to denote the index set
 and the information filtration at time $t$ given by the sequential procedure $\TT^{\strA}$. %\footnote{{\color{red} don't need supscript for $X_{k,t}$. Is that right?}}
We further define
$W^{\strA}_{k,t}=\prob\left(\tau_k < t\mid \fil_t^{\strA}\right).$
Similarly, we define the index set $S^{\strGfull{t_0}}_t$, information filtration $\fil^{\strGfull{t_0}}_t$, and posterior probability  $W^{\strGfull{t_0}}_{k,t}$ given by the sequential procedure $\TT^{\strGfull{t_0}}$. For any vector $\vv=(v_1,\cdots,v_m)$, we use the notation $[\vv]=(v_{(1)},\cdots,v_{(m)})$ for its order statistic. In addition, let $[\ab]=\ab$.

% $$W^{\strGfull{t_0}}_{k,s}=\prob\left(\tau_k < s\Big|\mathcal{F}_s^{\strGfull{t_0}}\right),$$
% where $S^{ \strGfull{t_0}}_{t_0+s}$ is the index set under the procedure $\TT^\strGfull{t_0}$ at time $t_0 + s$.
\begin{proposition}\label{prop:CouplingForProp}
Let $\{x_t,s_t,1\leq t\leq t_0\}$ be any sequence in the support of the stochastic process $\big\{(X_{k,t})_{k\in S^{\strA}_{t}},S^{\strA}_{t}, 1\leq t\leq t_0\big\}$
following a sequential procedure $\TT^{\strA}\in\adset_{\alpha}$. Then, there exists a coupling of $\spaceo$-valued random variables $(\hat{W},\hat{W}')$ such that
\begin{equation}
\begin{split}
		\hat{W}&\eqd \left.\left[ \left(W^{\strGfull{t_0}}_{k,{t_0+1}}\right)_{k\in S^{\strGfull{t_0}}_{t_0+1}}\right]\middle| \left\{\left(X_{k,t}\right)_{k\in S^{\strA}_{t}}=x_t,S^{\strA}_{t}=s_t, 1\leq t\leq t_0\right\}\right.,\\
	\hat{W}'&\eqd \left.\left[\left(W_{k,{t_0+1}}^{\strGfull{t_0+1} }\right)_{k\in S_{t_0+1}^{\strGfull{t_0+1}}}\right]\middle| \left\{\left(X_{k,t}\right)_{k\in S^{\strA}_{t}}=x_t,S^{\strA}_{t}=s_t, 1\leq t\leq t_0\right\}\right.,
\end{split}
\end{equation}%\footnote{\color{red}changed $\strA$ to $\strGfull{t_0+1}$ to connect it to the second step of the induction. These different notation give same value at $t_0+1$.}
and	$\hat{W}\lleq \hat{W}'$ a.s., {where $\eqd$ denotes that random variables on both sides are identically distributed.}
\end{proposition}
{We clarify that by the above proposition, the resulting $\hat{W}$ and $\hat{W}'$ are defined on the same probability space.}
%\yc{YC stops here.}
% \begin{remark}
% The support of the process $\big\{(X_{k,t})_{k\in S^{\strA}_{t}},S^{\strA}_{t}, 1\leq t\leq t_0\big\}$ depends on the sequential procedure $\TT^{\strA}$. Moreover, the coupling $(\hat{W},\hat{W}')$ in the above proposition depends on the sequential procedure $\TT^{\strA}$ and  the sequence $\{x_t,s_t,1\leq t\leq t_0\}$.
% \end{remark}
%\begin{remark}
%As can be seen in the proof of Theorem~\ref{thm:uniformy optimality}, we only need to analyze
%the conditional distribution given $\big\{(X_{k,t})_{k\in S^{\strA}_{t}}=x_t,S^{\strA}_{t}=s_t, 1\leq t\leq t_0\big\}$ when the sequence $\{x_t,s_t,1\leq t\leq t_0\}$ is in the support of the process.
%	\footnote{{\color{red} I tend to remove this remark. It seems obvious that one can only condition on random objects in its support.}}
%\end{remark}
Let \begin{equation}\label{eq:stocmono}
Y_s=\left[\left(W^{\strGfull{t_0}}_{k,t_0+s}\right)_{k\in S^{ \strGfull{t_0}}_{t_0+s}}\right] \in \spaceo.
\end{equation}
%}
%\footnote{{\color{red} Earlier version:}\begin{equation}\label{eq:stocmono}
%Y_s=\left[\left(W^{\strGfull{t_0}}_{k,s}\right)_{k\in S^{ \strGfull{t_0}}_{t_0+s}}\right] \in \spaceo.
%\end{equation}}
% where $S^{ \strGfull{t_0}}_{t_0+s}$ is the index set under the procedure $\TT^\strGfull{t_0}$ at time $t_0 + s$.
Under model {\mhomo}, the stochastic process $Y_s$ is stochastically monotone %, %
in that the following monotone coupling result holds.
%under the following definition of stochastic monotonicity.
\begin{proposition}\label{prop:monotone-o}
Suppose that model {\mhomo} holds. Then
for any $\yy, \yy'\in \spaceo$ such that $\yy \lleq \yy'$, there exists a coupling $(\hat{Y}_s,\hat{Y}'_s), s=0, 1, ... $,  satisfying
\begin{enumerate}
	\item  $\{\hat{Y}_s:s\geq 0\}$ has the same distribution as the conditional process $\{Y_s:s\geq 0\}$ given $Y_{0}=\yy$, and $\{\hat{Y}'_s:s\geq 0\}$ has the same distribution as the conditional process $\{Y_s:s\geq 0\}$ given $Y_{0}=\yy'$.
	\item  $\hat{Y}_s\lleq \hat{Y}'_s$, a.s. for all $s\geq 0$.
\end{enumerate}
Moreover, the process $(\hat{Y}_s,\hat{Y}'_s)$ does not depend on $\TT^{\strA}$, $t_0$, or the information filtration $\fil^{\strA}_{t_0}$.
\end{proposition}

Roughly, Proposition~\ref{prop:CouplingForProp} shows that the sequential procedure $\TT^{\strGfull{t_0}}$ tends to have a stochastically smaller detection statistic, in terms of the partial order $\lleq$, than that of $\TT^{\strGfull{t_0+1}}$ at time $t_0+1$, and thus tends to keep more active streams.  Proposition~\ref{prop:monotone-o} further shows that this trend will be carried over to any future time, including time $t_0+s+1$. The second step of induction is proved by formalizing this heuristic.

\section{Proof of Theorem~\ref{prop:stronger}}\label{sec:prop:stronger}
% {\color{red} We can consider to put this proof in the main text.}

\begin{customthm}{6}

Suppose that %the same conditions as in Theorem~\ref{thm:comparison} hold.
model {\mhomo} holds.
	For any $t_0, s\geq 0$ and any sequential detection procedure $\TT^{\strA}\in\adset_{\alpha}$, let
	$\fil^\strA_t$ be the information filtration  and $S^{\strA}_t$ be the set of active streams at time $t$ given by $\TT^{\strA}$.
	Then,
%\begin{equation}\label{eq:induction-statement}
\begin{equation}%\label{eq:induction-statement}
			\expe\left[|S_{t_0+s}^{\strA}|\middle| \fil^\strA_{{t_0}}\right]
				\leq \expe\left[|S_{t_0+s}^{\strGfull{t_0}}|\middle| \fil^\strA_{t_0}\right] \text{ a.s.}\tag{7}
\end{equation}
%	\end{equation}
\end{customthm}

%{propstronger}\label{prop:stronger}

%\footnote{{\color{red} OLD WRITING:	\begin{equation}\label{eq:induction-statement}
%		\expe\left[\uti_{t_0+s}(\TT^{\strA})\big| \mathcal{F}^\strA_{{t_0}}\right]
%				\leq \expe\left[\uti_{t_0+s}\left(\TT^{\strGfull{t_0}}\right)\big| \mathcal{F}^\strA_{t_0}\right] \text{ a.s.}
%	\end{equation}
%	}}
%\end{restatable}

\begin{proof}[Proof of Theorem~\ref{prop:stronger}]
We will prove the theorem by inducting on $s$.

For the base case ($s=0$) the theorem is obviously true for all $t_0$ and all $\TT^{\strA}\in\adset_{\alpha}$ as the both sides of \eqref{eq:induction-statement} are exactly the same.

We will prove the induction step in the rest of the proof.
Assume \eqref{eq:induction-statement} is true for any strategy $\TT^{\strA}\in\adset_{\alpha}$ and any $t_0$, for some $s=s_0$. Our goal is to prove that it is also true for any $t_0$, for $s=s_0+1$, using the following steps, {where we recall that $\TT^{\strGfull{t_0}}$ is defined as the sequential procedure that
takes the same steps as $\TT^{\strA}$
up to time $t_0$ and updates by  Algorithm~\ref{alg:one-step-rule1} from time $t_0+1$ and onward, and the sequential procedure $\TT^{\strGfull{t_0+1}}$ is defined similarly.}

\paragraph{Step 1: comparing $\TT^{\strGfull{t_0+1}}$ and $\TT^{\strA}$.}
	For $s=s_0+1$, since we assume \eqref{eq:induction-statement} is true for all $t_0$, we could replace $t_0$ by $t_0+1$ and $s$  by $s_0$ in \eqref{eq:induction-statement} and arrive at
	\begin{equation}
				\expe\left[\big|S^{\strA}_{t_0+s_0+1}\big|\Big\vert\fil_{t_0+1}^{\strA}\right]
				\leq \expe\left[\big|S^{\strGfull{t_0+1}}_{t_0+s_0+1}\big|\Big\vert \fil_{t_0+1}^{\strA}\right] \text{ a.s.}
	\end{equation}
	Taking conditional expectation $\expe\left[\cdot|\fil^{\strA}_{t_0}\right]$ on both sides, we arrive at
	\begin{equation}\label{eq:end-of-step1}
		\expe\left[\big|S^{\strA}_{t_0+s_0+1}\big|\Big\vert\fil_{t_0}^{\strA}\right]
				\leq \expe\left[\big|S^{\strGfull{t_0+1}}_{t_0+s_0+1}\big|\Big\vert \fil_{t_0}^{\strA}\right] \text{ a.s.}
	\end{equation}
\paragraph{Step 2: comparing $\TT^{\strGfull{t_0+1}}$ and $\TT^{\strGfull{t_0}}$.}
First, define a function $\phi_{t,s}:\spaceo\to \mathbb{R}$,
\begin{equation}\label{eq:phi-definition}
	\phi_{t,s}(\uu)=\expe\Big[|S^{\strGfull{t}}_{t+s}|\Big|\big[W^{\strGfull{t}}_{S^{\strGfull{t}}_{t},t}\big]=\uu\Big]=\expe\Big[\card\big(\big[W^{\strGfull{t}}_{S^{\strGfull{t}}_{t+s},t+s}\big]\big)\Big|\big[W^{\strGfull{t}}_{S^{\strGfull{t}}_{t},t}\big]=\uu\Big]
\end{equation}
for $t,s\geq 0$.  {Here, for a set $S$, and time points $s$ and $t$, $W^{\strGfull{t}}_{S,s} = \big(W_{k,s}^{\strGfull{t}}\big)_{k\in S}$,
where  $W_{k,s}^{\strGfull{t}} = \prob\left(\tau_k < s\Big|\fil_s^{\strGfull{t}}\right).$
 %following the procedure $\TT^{\strGfull{t}}$
}

From Proposition~\ref{prop:monotone-o}, we can see that $\phi_{t,s}(\uu)$ does not depend on the sequential procedure $\TT^{\strA}$ and the value of $t$. Thus, by replacing $\TT^{\strA}$ with $\TT^{\strGfull{t_0}}$, $t$ with $t_0+1$, and $s$ with $s_0$ in \eqref{eq:phi-definition}, we obtain
\begin{equation}\label{eq:phi-pt0}
	\phi_{t_0+1,s_0}(\uu)= \expe\Big[\card\big(\big[W^{\strGfull{t_0}}_{S^{\strGfull{t_0}}_{t_0+s_0+1},t_0+s_0+1}\big]\big)\Big|\big[W^{\strGfull{t_0}}_{S^{\strGfull{t_0}}_{t_0+1},t_0+1}\big]=\uu\Big].
\end{equation}
Here, to see the superscript of the process in the above equation is ${\strGfull{t_0}}$, we used the fact that if we follow the procedure $\TT^{\strGfull{t_0}}$ and switch to the proposed procedure at time $t_0+1$, then the overall sequential procedure is still $\TT^{\strGfull{t_0}}$.
% Similarly, by replacing $\TT^{\strA}$ with $\TT^{\strGfull{t_0+1}}$, $t$ with $t_0+1$, and $s$ with $s_0$, we obtain
% \begin{equation}
% 	\phi_{t_0+1,s_0}(\uu)= \expe\Big[\card\big(\big[W^{\strGfull{t_0+1}}_{S^{\strGfull{t_0+1}}_{t_0+s_0+1},t_0+s_0+1}\big]\big)\Big|\big[W^{\strGfull{t_0}}_{S^{\strGfull{t_0+1}}_{t_0+1},t_0+1}\big]=\uu\Big].
% \end{equation}

{Also from Proposition~\ref{prop:monotone-o}, we can see that} for any $\uu\lleq\uu'\in\spaceo$, there exists a coupling $(\hat{Y}_s,\hat{Y}'_s)$ such that $\hat{Y}_s$ has the same distribution as $\big[W^{\strGfull{t}}_{S^{\strGfull{t}}_{t+s},t+s}\big]$ given $\big[W^{\strGfull{t}}_{S^{\strGfull{t}}_{t},t}\big]=\uu$, $\hat{Y}'_s$ has the same distribution as $\big[W^{\strGfull{t}}_{S^{\strGfull{t}}_{t+s},t+s}\big]$ given $\big[W^{\strGfull{t}}_{S^{\strGfull{t}}_{t},t}\big]=\uu'$, and $\hat{Y}_s\lleq\hat{Y}'_s$ a.s. Thus,
\begin{equation}
	\phi_{t,s}(\uu)=\expe\big(\dim(\hat{Y}_s)\big) \text{ and } \phi_{t,s}(\uu')=\expe\big(\dim(\hat{Y}'_s)\big).
\end{equation}
According to the definition of the partial relationship `$\lleq$', $\hat{Y}_s\lleq\hat{Y}'_s$ implies $\dim(\hat{Y}_s)\geq \dim(\hat{Y}'_s)$. Combining this result with the above display, we conclude that $\phi_{t,s}(\uu)\geq \phi_{t,s}(\uu')$ for any $\uu\lleq\uu'\in\spaceo$.

Next, we write $\expe\left[\big|S^{\strGfull{t_0+1}}_{t_0+s_0+1}\big|\Big\vert \fil_{t_0}^{\strA}\right]$ and $\expe\left[\big|S^{\strGfull{t_0}}_{t_0+s_0+1}\big|\Big\vert \fil_{t_0}^{\strA}\right]$ in terms of the conditional expectation involving the function $\phi_{t,s}$. We start with $\expe\left[\big|S^{\strGfull{t_0}}_{t_0+s_0+1}\big|\Big\vert \fil_{t_0}^{\strA}\right]$.
By the iterative law of conditional expectation and \eqref{eq:phi-pt0}, we obtain
\begin{equation}
\begin{split}
		\expe\left[\big|S^{\strGfull{t_0}}_{t_0+s_0+1}\big|\Big\vert \fil_{t_0}^{\strA}\right]
	= &\expe\left[\expe\left\{\big|S^{\strGfull{t_0}}_{t_0+s_0+1}\big|\Big| \big[W^{\strGfull{t_0}}_{S^{\strGfull{t_0}}_{t_0+1},t_0+1}\big]\right\}\Big\vert \fil_{t_0}^{\strA}\right] \\
	= & \expe\left[\expe\left\{\card\big( [W^{\strGfull{t_0}}_{S^{\strGfull{t_0}}_{t_0+s_0+1},t_0+s_0+1}\big] \big) \Big| \big[W^{\strGfull{t_0}}_{S^{\strGfull{t_0}}_{t_0+1},t_0+1}\big]\right\}\Big\vert \fil_{t_0}^{\strA}\right]\\
	= & \expe\left[\phi_{t_0+1,s_0}\Big( \big[W^{\strGfull{t_0}}_{S^{\strGfull{t_0}}_{t_0+1},t_0+1}\big]\Big)\Big\vert \fil_{t_0}^{\strA}\right].
\end{split}
\end{equation}
According to the definition of the information filtration $\fil_{t_0}^{\strA}$, we further write the above conditional expectation as
\begin{equation}\label{eq:to-compare-1}
	\expe\left[\big|S^{\strGfull{t_0}}_{t_0+s_0+1}\big|\Big\vert \fil_{t_0}^{\strA}\right]
	= \expe\left[\phi_{t_0+1,s_0}\Big( \big[W^{\strGfull{t_0}}_{S^{\strGfull{t_0}}_{t_0+1},t_0+1}\big]\Big)\Big\vert \big\{S^{\strA}_r, X_{k,r}, k\in S^{\strA}_r, 1\leq r\leq t_0\big\}\right].
\end{equation}
Similarly, we have
\begin{equation}\label{eq:to-compare-2}
	\expe\left[\big|S^{\strGfull{t_0+1}}_{t_0+s_0+1}\big|\Big\vert \fil_{t_0}^{\strA}\right]
	= \expe\left[\phi_{t_0+1,s_0}\Big( \big[W^{\strGfull{t_0+1}}_{S^{\strGfull{t_0+1}}_{t_0+1},t_0+1}\big]\Big)\Big\vert \big\{S^{\strA}_r, X_{k,r}, k\in S^{\strA}_r, 1\leq r\leq t_0\big\}\right].
\end{equation}
% Note that $S^{\strGfull{t_0+1}}_{t_0+1}=S^{\strA}_{t_0+1}$. Thus, the last equation is simplified as
% \begin{equation}\label{eq:to-compare-2}
% 	\expe\left[\big|S^{\strGfull{t_0+1}}_{t_0+s_0+1}\big|\Big\vert \fil_{t_0}^{\strA}\right]
% 	= \expe\left[\phi_{t_0+1,s_0}\Big( \big[W^{\strA}_{S^{\strA}_{t_0+1},t_0+1}\big]\Big)\Big\vert \big\{S^{\strA}_s, X_{k,r}, 1\leq r\leq t_0, k\in S^{\strA}_r\big\}\right].
% \end{equation}
We proceed to a comparison between \eqref{eq:to-compare-1} and \eqref{eq:to-compare-2}.  According to Proposition~\ref{prop:CouplingForProp}, for each sequence $\{x_r,s_r,1\leq r\leq t_0\}$ that is in the support of the process $\{X_{S^{\strA}_{r},r},S^{\strA}_{r}, 1\leq r\leq t_0\}$ there exists a coupling $(\hat{W},\hat{W}')$ such that
\begin{equation}
	\hat{W}\eqd[W^{\strGfull{t_0}}_{S^{\strGfull{t_0}}_{t_0+1},{t_0+1}}]\Big| \{X_{S^{\strA}_{r},r}=x_r,S^{\strA}_{r}=s_r, 1\leq r\leq t_0\},
\end{equation}
\begin{equation}
	\hat{W}'\eqd [W^{\strGfull{t_0+1}}_{S^{\strGfull{t_0+1}}_{t_0+1},{t_0+1}}]\Big| \{X_{S^{\strA}_{r},r}=x_r,S^{\strA}_{t}=s_r, 1\leq r\leq t_0\},
\end{equation}
and
\begin{equation}
	\hat{W}\lleq \hat{W}' \text{ a.s.},
\end{equation}
where `$\eqd$' means two random  variables on both sides have the same distribution.
Thus,
\begin{equation}\label{eq:conditional-1}
	\expe\left[\phi_{t_0+1,s_0}\Big( \big[W^{\strGfull{t_0}}_{S^{\strGfull{t_0}}_{t_0+1},t_0+1}\big]\Big)\Big\vert X_{S^{\strA}_{r},r}=x_r,S^{\strA}_{r}=s_r, 1\leq r\leq t_0\right]=\expe\phi_{t_0+1,s_0}\big(\hat{W}\big)
\end{equation}
and
\begin{equation}\label{eq:conditional-2}
	\expe\left[\phi_{t_0+1,s_0}\Big( \big[W^{\strGfull{t_0+1}}_{S^{\strGfull{t_0+1}}_{t_0+1},t_0+1}\big]\Big)\Big\vert X_{S^{\strA}_{r},r}=x_r,S^{\strA}_{r}=s_r, 1\leq r\leq t_0\right]=\expe\phi_{t_0+1,s_0}\big(\hat{W}'\big).
\end{equation}
On the other hand, note that we have shown $\phi_{t_0+1,s_0}(\uu)\geq\phi_{t_0+1,s_0}(\uu')$ for any $\uu\lleq\uu'\in\spaceo$ and $\hat{W}\lleq\hat{W}'$ a.s. by the coupling. Thus,
\begin{equation}
 	\phi_{t_0+1,s_0}(\hat{W})\geq\phi_{t_0+1,s_0}(\hat{W}') \text{ a.s.}
 \end{equation}
 Combining the above inequality with \eqref{eq:conditional-1} and \eqref{eq:conditional-2}, we arrive at
 \begin{equation}
 	\begin{split}
 			&\expe\left[\phi_{t_0+1,s_0}\Big( \big[W^{\strGfull{t_0}}_{S^{\strGfull{t_0}}_{t_0+1},t_0+1}\big]\Big)\Big\vert X_{S^{\strA}_{r},r}=x_r,S^{\strA}_{r}=s_r, 1\leq r\leq t_0\right]\\
 			\geq & \expe\left[\phi_{t_0+1,s_0}\Big( \big[W^{\strA}_{S^{\strA}_{t_0+1},t_0+1}\big]\Big)\Big\vert X_{S^{\strA}_{r},r}=x_r,S^{\strA}_{r}=s_r, 1\leq r\leq t_0\right]
 	\end{split}
 \end{equation}
 for each sequence $\{x_r,s_r,1\leq r\leq t_0\}$ that is in the support of the process $\{X_{S^{\strA}_{r},r},S^{\strA}_{r}, 1\leq r\leq t_0\}$. Comparing the above inequality with \eqref{eq:to-compare-1} and \eqref{eq:to-compare-2}, we conclude that
 \begin{equation}\label{eq:end-of-step2}
 	\expe\left[\big|S^{\strGfull{t_0}}_{t_0+s_0+1}\big|\Big\vert \fil_{t_0}^{\strA}\right]\geq \expe\left[\big|S^{\strGfull{t_0+1}}_{t_0+s_0+1}\big|\Big\vert \fil_{t_0}^{\strA}\right]~~a.s.
 \end{equation}

\paragraph{Step 3: combining results from Steps 1 and 2.}
Combining \eqref{eq:end-of-step1} and \eqref{eq:end-of-step2}, we obtain
\begin{equation}\label{eq:end-of-step3}
	\expe\left[\big|S^{\strA}_{t_0+s_0+1}\big|\Big\vert\fil_{t_0}^{\strA}\right]\leq \expe\left[\big|S^{\strGfull{t_0}}_{t_0+s_0+1}\big|\Big\vert \fil_{t_0}^{\strA}\right] \text{ a.s.,}
\end{equation}
which implies that \eqref{eq:induction-statement} holds for arbitrary $\TT^{\strA}\in \adset_{\alpha}$, $t_0$, and $s=s_0+1$. This completes the induction.
\end{proof}
{\begin{remark}
	Proposition~\ref{prop:monotone-o} is used  in Step 2 of the above proof, where we only use the property that $\hat{Y}_s\lleq \hat{Y}'_s$ is independent of $t_0$ and $\mathbb{T}^{\strA}$.
The independence between $(\hat{Y}_s,\hat{Y}'_s)$ and $\mathcal{F}_{t_0}^{\mathbf A}$ is an additional result that further characterizes the coupling process. We did not use this additional property directly in the proof. 
\end{remark}}

\section{Proof of Theorems~\ref{thm:uniformy optimality} and \ref{thm:comparison}}
%Note that Theorem~\ref{thm:comparison} is an extension of Theorem~\ref{thm:uniformy optimality}. Thus,

It suffices to prove Theorem~\ref{thm:comparison}, as Theorem~\ref{thm:uniformy optimality} is straightforwardly implied by Theorem~\ref{thm:comparison}.
\begin{customthm}{5}%\label{thm:comparison}
	Let $\TT^{\strA}\in \adset_{\alpha}$ be an arbitrary sequential procedure. Further let
$\TT^{\strGfull{t_0}}$ and $\TT^{\strGfull{t_0+1}}$ be the switching procedures described above, with switching time $t_0$ and $t_0 + 1$, respectively, for some $t_0\geq0$.
Then, $\TT^{\strGfull{t_0}}, \TT^{\strGfull{t_0+1}} \in \adset_{\alpha}$ and under model {\mhomo}
\begin{equation*}
	 		\expe\left(\uti_t(\TT^{\strA})\right)\leq\expe\left(\uti_t({\TT^\strGfull{t_0+1}})\right)\leq  \expe\left(\uti_t({\TT^\strGfull{t_0}})\right)\leq \expe\left(\uti_t(\TTp)\right),
	 	\end{equation*}
	 	for all $t=1,2,\cdots$.
%\footnote{{\color{red} removed the "admissibility part" as the proposed method is already proved to be local optimal. }}
%$\TT^{\strGfull{t_0+1}}$ be defined by Algorithm~\ref{alg:combine},  $\TTp$ be defined by Algorithm~\ref{alg:SSS}, and $t_0\geq1$, then we have
	% \begin{enumerate}
	%  	\item [i)] $\expe\left(\uti_t\left(\TT^{\strA}\right)\right)=\expe\left(\uti_t\left(\TT^{\strGfull{t_0}}\right)\right)$ for $t\leq t_0$.
	%  	\item [ii)] There exists $s>t_0$ such that $ \expe\left(\uti_s\left(\TT^{\strA}\right)\right)\leq\expe\left(\uti_s\left({\TT^\strGfull{t_0}}\right)\right)$.
	%  	\item [iii)]Under model \mhomo, we have
	%  	\begin{equation}
	%  		\expe\left(\uti_t(\TT^{\strA})\right)\leq\expe\left(\uti_t({\TT^\strGfull{t_0}})\right)\leq  \expe\left(\uti_t({\TT^\strGfull{t_0+1}})\right)\leq \expe\left(\uti_t(\TTp)\right),
	%  	\end{equation}
	%  	for all $t=1,2,\cdots$.
	%  \end{enumerate}
\end{customthm}
\begin{proof}[Proof of Theorem~\ref{thm:comparison}]
{First, note that $\TT^{\strA}\in\adset_{\alpha}$ and $\TT^{\strGfull{t_0}}$ agrees with $\TT^{\strA}\in\adset_{\alpha}$ up to time $t_0$. Thus, $\TT^{\strGfull{t_0}}$ control the LFNR to be no greater than $\alpha$ from time $1$ to $t_0$. Also, according to Proposition~\ref{prop:control}, $\TT^{\strGfull{t_0}}$ controls the LFNR at level $\alpha$ from time $t_0+1$ and onward. Thus, $\TT^{\strGfull{t_0}}\in\adset_{\alpha}$. Similarly, $\TT^{\strGfull{t_0+1}}\in\adset_{\alpha}$. }

	Applying Theorem~\ref{prop:stronger} but replacing $t_0$ by $t_0 + 1$,
and taking expectation on both sides of the inequality, we obtain
	\begin{equation}
		\expe|S^{\strA}_{t_0+1+s}|\leq \expe|S^{\strGfull{t_0+1}}_{t_0+1+s}|
	\end{equation}
	for every $t_0\geq 0$ and $s\geq 0$. That is, for every $t\geq t_0+1$,
	\begin{equation}
				\expe|S^{\strA}_{t}|\leq \expe|S^{\strGfull{t_0+1}}_{t}|.
	\end{equation}
	For $t <  t_0+1$, as $\TT^{\strA}$ and $\TT^{\strGfull{t_0+1}}$ share the same index set, we have
	\begin{equation}
						\expe|S^{\strA}_{t}|=\expe|S^{\strGfull{t_0+1}}_{t}|.
	\end{equation}
%	for $t\leq t_0$.
Combining the above inequalities, we obtain
	\begin{equation}
		\expe|S^{\strA}_{t}|\leq \expe|S^{\strGfull{t_0+1}}_{t}|
	\end{equation}
	for all $t\geq 0$. This further implies
	\begin{equation}\label{eq:ineq-a-ap}
		\expe\{\uti_t(\TT^{\strA})\} = \sum_{s=1}^t \expe|S^{\strA}_{s}|\leq \sum_{s=1}^t\expe|S^{\strGfull{t_0+1}}_{s}| = \expe\{\uti_t(\TT^{\strGfull{t_0+1} })\}.
	\end{equation}
	This proves the inequality for comparing procedures $\TT^{\strA}$ and $\TT^{\strGfull{t_0+1}}$.
%Making the same comparison,
We then compare $\TT^{\strGfull{t_0+1}}$ and
$\TT^{\strGfull{t_0}}$, based on the same arguments above except that
we replace $\TT^{\strA}$ by $\TT^{\strGfull{t_0+1}}$,
and replace $\TT^{\strGfull{t_0+1}}$ by
$\TT^{\strGfull{t_0}}$.
We obtain
	\begin{equation}
		\expe\{
		\uti_t(\TT^{\strGfull{t_0+1}})\}
		\leq \expe\{
		\uti_t(\TT^{\strGfull{t_0}})
		\}
	\end{equation}
	for all $t\geq 0$.

	Finally, we compare $\TT^{\strGfull{t_0}}$ and $\TTp = \TT^{\strGfull{0}}$ using a similar argument, which gives
%replacing $\TT^{\strA}$ by $\TT^{\strGfull{t_0}}$ and $t_0$ by $0$ and noticing that $\TT^{\strGfull{0}}$ is the same as $\TTp$, we obtain
	\begin{equation}
		\expe\{
		\uti_t(\TT^{\strGfull{t_0}})\}
		\leq \expe\{
		\uti_t(\TTp)
		\}.
	\end{equation}
\end{proof}
{
\section{Proof of Theorem~\ref{thm:uniform-optimal-other}}
First, by Theorem~\ref{prop:stronger}, we directly see that $\expe\{|S_t^{\strA}|\}\leq \expe\{|S_t^{\strG}|\}$, for any sequential procedure $\TT^{\strA}\in\adset_{\alpha}$. Thus, $\expe(\textrm{CD}_t(\TT^*))=K-\expe|S_t^{\strG}|\leq \expe(\textrm{CD}_t(\TT^{\strA}))$, which further implies $\expe(\textrm{CD}_t(\TT^*)) = \inf_{\TT\in\adset_{\alpha}}\expe(\textrm{CD}_t(\TT))$.

We proceed to the analysis of $\textrm{RL}_t(\TT)$. By interchanging the order of double summation, we have
\begin{equation}
	\textrm{RL}_t(\TT)
	= \sum_{k=1}^K (T_k\wedge \tau_k\wedge t) = \sum_{k=1}^K\sum_{s=1}^t \ind(s\leq T_k\wedge \tau_k) = \sum_{s=1}^t\sum_{k=1}^K \ind(s\leq T_k\wedge \tau_k) = \sum_{s=1}^t \sum_{k\in S_s}\{1-\ind(\tau_k<s)\}%\sum_{s=1}^t \sum_{k\in S_s}\ind(s\leq \tau_k),
\end{equation}
which leads to
\begin{equation}
	\expe\{\textrm{RL}_t(\TT)\} = \sum_{s=1}^t \expe\big[\sum_{k\in S_s}\{1-\ind(\tau_k<s)\}\big] = \sum_{s=1}^t \expe\Big[\expe\big[\sum_{k\in S_s}\{1-\ind(\tau_k<s)\}\big] |\fil_{s} \Big].
\end{equation}
Recall that $W_{k,s}= \prob(\tau_k<s|\fil_s)$ and $S_s\in\fil_s$. The above display yields
\begin{equation}
	\expe\{\textrm{RL}_t(\TT)\} = \sum_{s=1}^t \expe\big\{\sum_{k\in S_s}(1-W_{k,s})\big\}.
\end{equation}
From the above equation, we can see that in order to show
$\expe\{\textrm{RL}_t(\TT^*)\}=\sup_{\TT\in\adset_{\alpha}}\expe\{\textrm{RL}_t(\TT)\}$, it suffices to show 
$
	\expe\big\{\sum_{k\in S_t}(1-W_{k,t})\big\}
$
is maximized 
for every $t=1,2,\cdots$, which follows directly  from the following extension of Theorem~\ref{prop:stronger}.
\begin{proposition}\label{prop:Psi-optimal}
Suppose that %the same conditions as in Theorem~\ref{thm:comparison} hold.
model {\mhomo} holds.
	For any $t_0, s\geq 0$ and any sequential detection procedure $\TT^{\strA}\in\adset_{\alpha}$, let
	$\fil^\strA_t$ be the information filtration  and $S^{\strA}_t$ be the set of active streams at time $t$ given by $\TT^{\strA}$.
	Then,
\begin{equation}\label{eq:induction-statement-other}
			\expe\left[\Psi([W_{S_{t_0+s}^{\strA},t_0+s}^{\strA}])\middle| \fil^\strA_{{t_0}}\right]
				\leq \expe\left[\Psi([W_{S_{t_0+s}^{\strGfull{t_0}},t_0+s}^{\strGfull{t_0}}])\middle| \fil^\strA_{t_0}\right] \text{ a.s.},
\end{equation}
	where $\Psi:\spaceo\to \Real$ is defined as $\Psi(\ww)=\sum_{k=1}^{m}(1-w_k)$ for $\ww=(w_1,\cdots, w_m)\in\spaceo$.
\end{proposition}
In the rest of the section, we provide the proof of Proposition~\ref{prop:Psi-optimal}.
\begin{proof}[Proof of Proposition~\ref{prop:Psi-optimal}]
	The proof of Proposition~\ref{prop:Psi-optimal} is similar to that of Theorem~\ref{prop:stronger}. We will only state the main differences and omit the repetitive details.

	First, by replacing $|S^{\strA}_{t}\big|$ with $\Psi([W^{\strA}_{S^{\strA}_{t},t}])$ for $t$ taking different values in the proof of Theorem~\ref{prop:stronger}, we obtain the following inequality that is similar to \eqref{eq:end-of-step1}
	\begin{equation}\label{eq:end-of-step1-psi}
		\expe\left[\Psi([W^{\strA}_{S^{\strA}_{t_0+s_0+1},t_0+s_0+1}])\Big\vert\fil_{t_0}^{\strA}\right]
				\leq \expe\left[\Psi(W^{\strGfull{t_0+1}}_{{S}_{t_0+s_0+1},t_0+s_0+1}\Big\vert \fil_{t_0}^{\strA}\right] \text{ a.s.}
	\end{equation}
	for all $t_0$ and $s_0$.
Next, we define a function $\tilde{\phi}_{t,s}:\spaceo\to \mathbb{R}$,
\begin{equation}\label{eq:phi-definition}
	\tilde{\phi}_{t,s}(\uu)=\expe\Big[\Psi([W^{\strGfull{t}}_{S^{\strGfull{t}}_{t+s},t+s}])\Big|\big[W^{\strGfull{t}}_{S^{\strGfull{t}}_{t},t}\big]=\uu\Big]
\end{equation}
for $t,s\geq 0$. Then, we replace the $\phi$ with $\tilde{\phi}$ in the proof of Theorem~\ref{prop:stronger} and obtain the following inequality that is similar to \eqref{eq:end-of-step2}.
\begin{equation}\label{eq:end-of-step2-psi}
 	\expe\left[\Psi([W^{\strGfull{t_0}}_{S^{\strGfull{t_0}}_{t_0+s_0+1},t_0+s_0+1}])\Big\vert \fil_{t_0}^{\strA}\right]\geq \expe\left[\Psi([W^{\strGfull{t_0+1}}_{S^{\strGfull{t_0+1}}_{t_0+s_0+1},t_0+s_0+1}])\Big\vert \fil_{t_0}^{\strA}\right]~~a.s.
 \end{equation}
 We point out that to arrive at the above inequality, the following property about $\Psi$ is used: $\Psi(\ww')\lleq\Psi(\ww)$
  for any $\ww,\ww'\in\spaceo$ satisfying $\ww\lleq\ww'$.

Combining \eqref{eq:end-of-step1-psi} and \eqref{eq:end-of-step2-psi}, we obtain
\begin{equation}\label{eq:end-of-step3-psi}
	\expe\left[\Psi([W^{\strA}_{S^{\strA}_{t_0+s_0+1},t_0+s_0+1}])\Big\vert\fil_{t_0}^{\strA}\right]\leq \expe\left[\Psi([W^{\strGfull{t_0}}_{S^{\strGfull{t_0}}_{t_0+s_0+1},t_0+s_0+1}])\Big\vert \fil_{t_0}^{\strA}\right] \text{ a.s.,}
\end{equation}
which extends \eqref{eq:end-of-step3} and completes the proof.
\end{proof}

}
\section{Proof of Propositions~\ref{prop:CouplingForProp} and \ref{prop:monotone-o}}\label{sec:proof-proposition-ordering}
The proof of Propositions~\ref{prop:CouplingForProp} and \ref{prop:monotone-o} is involved. We will first introduce some concepts in stochastic ordering, followed by several useful lemmas, and then present the proof of the propositions.
\subsection{Stochastic ordering}
We first review a few important concepts and classic results
 on partially ordered spaces. More details about stochastic ordering and coupling can be found in \cite{thorisson2000coupling,lindvall2002lectures,kamae1977stochastic,lindvall1999strassen}.
\begin{definition}[Partially Ordered Space (pospace)]
	A space $(\mathcal{S},\lleq)$ is said to be a partially ordered space (or pospace) if $\lleq$ is a partial order relation over the topological space $\mathcal{S}$ and the set $\{(x,y)\in\mathcal{S}^2: x\lleq y\}$ is a closed subset of $\mathcal{S}^2$. %\yc{You mean %$\mathcal{S}^2$? How do you talk about closeness here? Do you need a metric space here?}
\end{definition}
\begin{definition}[Increasing functions over a partially ordered space]
Let ${(\mathcal{S}_1,\lleq_{\mathcal{S}_1})}$ and ${(\mathcal{S}_2,\lleq_{\mathcal{S}_2})}$ be  partially ordered polish spaces. A map $g:\mathcal{S}_1\to\mathcal{S}_2$ is said to be increasing if $g(u)\lleq_{\mathcal{S}_2} g(v)$ for all $u\lleq_{\mathcal{S}_1} v$ with $u,v\in \mathcal{S}_1$.
\end{definition}
\begin{definition}[Stochastic ordering of real-valued random variables]
	Let $X$ and $Y$ be two random variables, we say $X$ is stochastically less than or equal to $Y$, if $\prob(X\geq x)\leq \prob(Y\geq x)$ for all real number $x$. In this case, we write $X\leq_{st}Y$.
\end{definition}
The following statements give some equivalent definitions for $X\leq_{st}Y$
\begin{fact}\label{fact:straseen-rv}
	The following statements are equivalent.
\begin{enumerate}
	\item $X\leq_{st}Y$.
	\item For all increasing, bounded, and measurable functions $g: \mathbb{R}\to \mathbb{R}$,
	$\expe(g(X))\leq \expe(g(Y))$.
	\item There exists a coupling $(\hat{X},\hat{Y})$ such that $ \hat{X}\eqd X$, $\hat{Y}\eqd Y$, and
	\begin{equation}
		\hat{X}\leq \hat{Y} \text{ a.s.}
	\end{equation}
	Here, $\eqd$ denotes that the random variables on both sides have an identical distribution.
\end{enumerate}
In particular, the equivalence between 1 and 3 is known as the Strassen's Theorem \citep{strassen1965existence}.
\end{fact}
% The idea of stochastic ordering can be extended to partial order relationship defined on polish spaces.
\begin{definition}[Stochastic ordering on a partially ordered polish space]
	Let $(\mathcal{S},\lleq)$ be a partially ordered polish space, and let $X$ and $Y$ be $\mathcal{S}$-valued random variables. We say $Y$ stochastically dominates $X$, denoted by $X\lleq_{st}Y$ if for all bounded, increasing, and measurable function $g:\mathcal{S}\to \mathbb{R}$, $\expe(g(X))\leq \expe(g(Y))$.
\end{definition}
\begin{fact}[Strassen's theorem for polish pospace, Theorem 2.4 in \cite{lindvall2002lectures}]\label{fact:strassen-pospace}
	Let $(\mathcal S,\lleq)$ be a polish partially ordered space, and let $X$ and $Y$ be $\mathcal{S}$-valued random variables. Then, $X\lleq_{st}Y$ if and only if there exists a coupling $(\hat{X},\hat{Y})$ such that $\hat{X}\eqd X$, $\hat{Y}\eqd Y$ and $\hat{X}\lleq \hat{Y}$ a.s.
\end{fact}

\begin{definition}[Stochastic dominance for Markov kernels]\label{def:kernel}
Let $K$ and $\tilde{K}$ be transition kernels for Markov chains over a partially ordered polish space $(\mathcal{S},\lleq)$. The transition kernel $\tilde{K}$ is said to stochastically dominate $K$ if
\begin{equation}
	x\lleq y \implies K(x,\cdot)\lleq_{st} \tilde{K}(y,\cdot).
\end{equation}
In particular, if the above is satisfied for the same kernel $K=\tilde{K}$, then we say $K$ is stochastically monotone.	
\end{definition}
\begin{fact}[Strassen's theorem for Markov chains over a polish pospace]\label{fact:strassen-markov}
Let $\{X_t\}$ and $\{Y_t\}$ be Markov chains over a partially ordered polish space, $(\mathcal{S},\lleq)$, with transition kernels $K$ and $\tilde{K}$ where $\tilde{K}$ stochastically dominates $K$. Then, for all initial points $x_0\lleq y_0$, there is a coupling $\{(\hat{X}_t,\hat{Y}_t)\}$ of $\{X_t\}$ starting at $x_0$ and $\{Y_t\}$ starting at $y_0$ such that
\begin{equation}
	\hat{X}_t\lleq \hat{Y}_t \qquad \forall t~~ a.s.
\end{equation}
\end{fact}
Fact~\ref{fact:strassen-markov} is a special case of Theorem 5.8 in \cite{lindvall2002lectures}.
\subsection{Stochastic ordering and Markov chains on $\spaceu$ and $\spaceo$}\label{sec:lemma-odering}
In this section, we provide some supporting lemmas regarding properties of the partial order relationship defined in Section~
\ref{sec:proof-sketch}, and show stochastic ordering of several Markov chains. The proof of these lemmas is given in Section~\ref{sec:proof-lemma}.
%{\color{red}XL: Do you prefer to organize proof of lemmas just after the statements? The current writing require a lot of page turning to read...}

Recall that in Section~\ref{sec:proof-sketch}, we define a space
\begin{equation}
 \spaceo=\bigcup_{k=1}^K \Big\{\vv=(v_1,\cdots,v_k)\in[0,1]^k: 0\leq v_1\leq \cdots v_k\leq 1
		\Big\}\cup \{\ab\}.
\end{equation}
Here, we also define a space with unordered elements.
\begin{equation}
	 \spaceu=\bigcup_{k=1}^K [0,1]^k\cup \{\ab\}.
\end{equation}
We first present a lemma showing that the space $(\spaceo,\lleq)$ is a polish partial order space.
\begin{restatable}{lemma}{pospace}\label{lemma:pospace}
	$(\spaceo,\lleq)$ is a partially ordered space. In addition, $\spaceo$ is a polish space equipped with the metric
\begin{equation}\label{eq:metric}
	d(\uu,\vv)=
	\begin{cases}
		\max_{1\leq m\leq \card(\uu)}|u_m-v_m| & \text{ if } \card(\uu)=\card(\vv)\geq 1\\
		0 & \text{ if } \uu=\vv=\ab\\
		2 & \text{ if } \card(\uu)\neq\card(\vv)
	\end{cases}
\end{equation}
for $\uu,\vv\in\spaceo$.
\end{restatable}

% In this paper, we will choose $\mathcal{S}_1=\mathcal{S}_2=\spaceo$ and $\lleq_{\mathcal{S}_1}$ and $\lleq_{\mathcal{S}_2}$ to be $\lleq$ for the most of time. In some cases, we will choose $\mathcal{S}_2$ be the real line and $\lleq_{\mathcal{S}_2}$ be $\leq$, the usual less or equal sign.

\newcommand{\Hu}{H_{\mathrm{u}}}
\newcommand{\Iu}{I_{\mathrm{u}}}
\newcommand{\Ho}{H_{\mathrm{o}}}
\newcommand{\Io}{I_{\mathrm{o}}}
We define mappings $\Io:\spaceo\to\{0,\cdots, K\}$ and $\Ho:\spaceo\to\spaceo$ as follows. For any $\uu\in\spaceo$,
define
\begin{equation}\label{eq:io}
	\Io(\uu)=
\begin{cases}
	\sup\Big\{n: \sum_{i=1}^n u_i\leq \alpha n, n\in\{0,...,\card(\uu)\}\Big\} &\text{ if } \card(\uu)\geq 1, \uu=(u_1,..., u_{\card(\uu)})\\
	0 &\text{ if }\card(\uu)=0,
\end{cases}
\end{equation}
and
\begin{equation}
 \Ho(\uu) =
	\begin{cases}
			(u_1,\cdots, u_{\Io(\uu)}) &\text{ if } \Io(\uu)\geq 1,\\
			\ab &\text{ otherwise.}
	\end{cases}
\end{equation}
%Here, we define $\Io(u)=0$ and $\Ho(u)=\ab$ if $u=\ab$ or $u_1>\alpha$.
The mapping $\Ho$ is closely related to the one-step update rule in Algorithm~\ref{alg:one-step-rule1}, as summarized in the next lemma.
\begin{restatable}{lemma}{lemmaonestepho}\label{lemma:one-step-ho}
	If we input $(W_{k,t})_{k\in S_t}=\uu$ and an index set $S_t$ with $|S_t|=\card(\uu)$ in Algorithm~\ref{alg:one-step-rule1}, then the output $S_{t+1}$ satisfies
	\begin{equation}
		|S_{t+1}|=\Io(\uu)\text{ and }[(W_{k,t})_{k\in S_{t+1}}]=\Ho([\uu]).
	\end{equation}
\end{restatable}
Other compound sequential detection rules in $\adset_{\alpha}$ are characterized through the next lemma.
\begin{restatable}{lemma}{lemmacontrol}\label{lemma:control}
	$\TT=(T_1,\cdots,T_K)\in\adset_{\alpha}$ if and only if
	\begin{equation}\label{eq:constrain-equi}
	\TT\in\adset \text{ and } \sum_{k=1}^K \ind(T_k> t) W_{k,t}\leq \alpha\cdot\sum_{k=1}^K\ind(T_{k}> t) \text{ for } t=0,1,2,\cdots
\end{equation}
The above expression is equivalent to
\begin{equation}
	S_{t+1}\text{ is } \mathcal{F}_{t} \text{ measurable } ,~S_{t+1}\subset S_{t},	\sum_{k\in S_{t+1}} W_{k,t}\leq \alpha\cdot|S_{t+1}|
\end{equation}
	for  $t=0,1,2,\cdots$, and $T_{k}=\sup\{t: k\in S_t\}$.
%	In particular, for every $\alpha\in[0,1]$ we always have $\TT=(1,\cdots,1)\in\adset_{\alpha}$.
\end{restatable}
The next lemma compares the second statement in the above lemma with the
output of the function $\Ho$.
\begin{restatable}{lemma}{lemmahodomin}\label{lemma:ho-domin}
		Let $\uu=(u_1,\cdots,u_m)\in\spaceu$ with $\card(\uu)=m\geq 1$. Let $k_1,\cdots, k_l\in\{1,\cdots,m\}$ be distinct and satisfy
		\begin{equation}
			\sum_{i=1}^l u_{k_i}\leq \alpha l.
		\end{equation}
		Then, $\Ho([\uu])\lleq [(u_{k_1},\cdots, u_{k_l})]$. Moreover, if $\Ho([\uu])=\ab$, then for any $S\subset \{1,\cdots,m\}$ with $|S|\geq 1$, $\sum_{i\in S} u_i>\alpha |S|$.
	\end{restatable}
\begin{restatable}{lemma}{hmonotone}\label{lemma:h-monotone}
For any $\uu\lleq\vv\in\spaceo$,
		$\Ho(\uu)\lleq \Ho(\vv)$. That is, the mapping $\Ho(\uu)$ is increasing in $\uu$.
\end{restatable}
% {\color{red} define $\Hu$ in an easier form here.}
% \begin{restatable}{lemma}{hmonotone}\label{lemma:h-monotone}
% 		$[\Hu(u)]\lleq [\Hu(v)]$ for all $u,v\in \spaceu$ with $[u]\lleq [v]$. That is, the map $[\Hu(u)]$ is increasing in $[u]$. In particular, if $[u]=[v]$ then $[\Hu(u)]=[\Hu(v)]$.
% \end{restatable}
Next, we present several lemmas on the stochastic ordering of random variables and Markov chains. We start with a simple but useful result regarding the stochastic monotonicity of a likelihood ratio under a mixture model.
\begin{restatable}{lemma}{likelihoodratiomonotone}\label{lemma:likelihood-ratio-monotone}
Let $p(x)$ and $q(x)$ be two density functions with respect to some baseline measure $\mu$ and assume that $p(\cdot)$ and $q(\cdot)$ have the same support.
Let $L(x)=\frac{q(x)}{p(x)}$ be the likelihood ratio. For $\delta\in[0,1]$, let $Z_{\delta}$ be a random variable with the density function $\delta q+(1-\delta)p$ and $L_{\delta}=L(Z_{\delta})$.
Then, for $0\leq\delta_1<\delta_2\leq 1$, we have
\begin{equation}
	L_{\delta_1}\leq_{st} L_{\delta_2}.
\end{equation}
\end{restatable}
This result is  intuitive: if we have more weights in $q$ for the mixture distribution, then the likelihood ratio will be larger, giving more evidence in favor of $q$.
\begin{restatable}{lemma}{lemmadistv}\label{lemma:dist-v}
Assume model \mhomo{} holds.
Let $V_{k,t}=\prob\left(\tau_k<t\vert X_{k,1},\cdots,X_{k,t}\right)$. Then,
	\begin{equation}\label{eq:v-k-t}
	V_{k,0}=0 \text{ and } V_{k,t+1}=\frac{q(X_{k,t+1})/p(X_{k,t+1})}{(1-\theta)(1-V_{k,t})/(\theta + (1-\theta)V_{k,t})+q(X_{k,t+1})/p(X_{k,t+1})}.
\end{equation}
 Moreover, $\{V_{k,t}\}_{t=0,1,\cdots}$ are independent and identically distributed processes for different $k$.
\end{restatable}
\begin{restatable}{lemma}{lemmadeltat}\label{lemma:delta_t}
Assume model \mhomo{} holds.
Let $\delta_{k,t} = \prob\left(\tau_k\leq t\vert  X_{k,1},\cdots,X_{k,t}\right)$, then
\begin{equation}
	\delta_{k,t}=\theta+(1-\theta)V_{k,t},
\end{equation}
where $V_{k,t}$ is defined in \eqref{eq:v-k-t}.
	\end{restatable}
\begin{restatable}{lemma}{lemmamonotonev}\label{lemma:monotone-v}
	Under model \mhomo{}, the process $\{V_{1,t}\}_{t\geq 0}$ defined in \eqref{eq:v-k-t} is a homogeneous Markov chain. In addition, its transition kernel is stochastically monotone. We will later refer to this transition kernel as $K(\cdot,\cdot)$.
\end{restatable}
\begin{restatable}{lemma}{lemmakka}\label{lemma:kka}
For any $t\geq 1$ and $\TT^{\strA}$, $[W^{\strA}_{S^{\strA}_{t+1},t+1}]$ is conditionally independent of $\fil^{\strA}_{t}$ given $[W^{\strA}_{S^{\strA}_{t+1},t}]$.
% \footnote{\color{red} is this writing rigorous??? Independence with a sigma field??}
Moreover, the conditional density of $\big[W^{\strA}_{S^{\strA}_{t+1},t+1}\big]$ at $\vv$ given $\big[W^{\strA}_{S^{\strA}_{t+1},t}\big]=\uu\in\spaceo$ is
\begin{equation}
	\KKa(\uu,\vv)
	:=
	\begin{cases}
		\sum_{\pi \in \PP_{\card(\uu)}}\prod_{l=1}^{\card(\uu)}K(u_l,v_{\pi(l)}) &\text{ if } \card(\uu)=\card(\vv)\geq 1,\\
		1 & \text{ if } \card(\uu)=\card(\vv)=0,\\
		0 &\text{ otherwise,}
	\end{cases}
\end{equation}
where $\PP_{m}$ denotes the set of all permutations over $\{1,\cdots, m\}$.
\end{restatable}
\begin{restatable}{lemma}{lemmasomarkovgenerationA}\label{lemma:so-markov-generation-A}
	For each $\uu\in\spaceo$ with $\card(\uu)=m\geq 1$, generate an $\spaceo$-valued random variable $V$ as follows,
		\begin{enumerate}
			\item For each $k\in\{1,\cdots,m\}$, generate $Z_k\sim K(u_k, \cdot)$ independently for different $k$.
			\item Let $V=[(Z_1,...,Z_{m})]$.
		\end{enumerate}
		In addition, if $m=0$, we let $V=\ab$.
	Then, $V\sim \KKa(\uu,\cdot)$.
	\end{restatable}
\begin{restatable}{lemma}{lemmamonotonekka}\label{lemma:monotone-kka}
	For $\uu,\uu'\in\spaceo$ with $\uu\lleq \uu'$, we have $\KKa(\uu,\cdot)\lleq_{st}\KKa(\uu',\cdot)$.
\end{restatable}
\begin{restatable}{lemma}{lemmakko}\label{lemma:kko}
For any $t,s\geq 0$ and $\TT^{\strA}$, $\big[W^{\strGfull{t}}_{S^{\strGfull{t}}_{t+s+1},t+s+1}\big]$ is conditionally independent of $\fil^{\strGfull{t}}_{t+s}$ given $\big[W^{\strGfull{t}}_{S^{\strGfull{t}}_{t+s},t+s}\big]$.
% \footnote{\color{red} Same question here. Is it rigorous to say independent with a sigma field?}
Moreover, the conditional density of $\big[W^{\strGfull{t}}_{S^{\strGfull{t}}_{t+s+1},t+s+1}\big]$ at $\vv$ given $\big[W^{\strGfull{t}}_{S^{\strGfull{t}}_{t+s},t+s}\big]=\uu\in\spaceo$ is
\begin{equation}
	\KKo(\uu,\vv)
	:= \KKa(\Ho(\uu),\vv)=
	\begin{cases}
		\sum_{\pi \in \PP_{\Io(\uu)}}\prod_{l=1}^{\Io(\uu)}K(\Ho(\uu)_l,v_{\pi(l)}) &\text{ if } \card(\vv)=\Io(\uu)\geq 1\\
		1 &\text{ if } \card(\vv)=\Io(\uu)=0\\
		0 &\text{ otherwise,}
	\end{cases}
\end{equation}
where $\PP_{m}$ denotes the set of all permutations over $\{1,\cdots, m\}$.
\end{restatable}
%{\color{red}
\begin{remark}
There is a key difference between Lemma~\ref{lemma:kka} and Lemma~\ref{lemma:kko}, though they may look similar at a first glance.
%	We remark on the difference between Lemma~\ref{lemma:kka} and Lemma~\ref{lemma:kko}, two lemmas which are similar in form.
In Lemma~\ref{lemma:kka}, we consider the conditional distribution of $\big[W^{\strA}_{S^{\strA}_{t+1},t+1}\big]$ given $\big[W^{\strA}_{S^{\strA}_{t+1},t}\big]$, where the index set $S_{t+1}^{\strA}$ is the same for the two random vectors.
In Lemma~\ref{lemma:kko}, we consider the conditional distribution of $\big[W^{\strGfull{t}}_{S^{\strGfull{t}}_{t+s+1},t+s+1}\big]$ given
$\big[W^{\strGfull{t}}_{S^{\strGfull{t}}_{t+s},t+s}\big]$, where the two random vectors are associated with two different index sets
%This is different from the result in Lemma~\ref{lemma:kko}, where the index sets are
$S^{\strGfull{t}}_{t+s+1}$ and $S^{\strGfull{t}}_{t+s}$.
%for the two random vectors $\big[W^{\strGfull{t}}_{S^{\strGfull{t}}_{t+s+1},t+s+1}\big]$ and %$\big[W^{\strGfull{t}}_{S^{\strGfull{t}}_{t+s},t+s}\big]$.
 This difference reflects a key difference between the proposed one-step update rule and an arbitrary procedure.
\end{remark}
%}
% \yc{Maybe we add a remark about the subscripts in D.13, in case the reviewers get an impression that it is wrong. }
\begin{restatable}{lemma}{lemmasomarkovgeneration}\label{lemma:so-markov-generation}
	For each $\uu\in\spaceo$ and $m=\Io(\uu)$, generate an $\spaceo$-valued random variable $V$ as follows,
		\begin{enumerate}
			\item For each $k\in\{1,\cdots,m\}$, generate $Z_k\sim K(\Ho(\uu)_k, \cdot)$ independently for different $k$.
			\item Let $V=[(Z_1,...,Z_{m})]$.
		\end{enumerate}
		In addition, if $m=0$, we let $V=\ab$.
	Then, $V\sim \KKo(\uu,\cdot)$.
\end{restatable}
\subsection{Proof of Propositions~\ref{prop:CouplingForProp} and \ref{prop:monotone-o}}\label{sec:main-prop-proof}

\begin{customprop}{B.1}%\label{prop:CouplingForProp}
Let $\{x_t,s_t,1\leq t\leq t_0\}$ be any sequence in the support of the stochastic process $\big\{(X_{k,t})_{k\in S^{\strA}_{t}},S^{\strA}_{t}, 1\leq t\leq t_0\big\}$
following a sequential procedure $\TT^{\strA}\in\adset_{\alpha}$. Then, there exists a coupling of $\spaceo$-valued random variables $(\hat{W},\hat{W}')$ such that
\begin{equation*}
\begin{split}
		\hat{W}&\eqd \left.\left[ \left(W^{\strGfull{t_0}}_{k,{t_0+1}}\right)_{k\in S^{\strGfull{t_0}}_{t_0+1}}\right]\middle| \left\{\left(X_{k,t}\right)_{k\in S^{\strA}_{t}}=x_t,S^{\strA}_{t}=s_t, 1\leq t\leq t_0\right\}\right.,\\
	\hat{W}'&\eqd \left.\left[\left(W_{k,{t_0+1}}^{\strGfull{t_0+1} }\right)_{k\in S_{t_0+1}^{\strGfull{t_0+1}}}\right]\middle| \left\{\left(X_{k,t}\right)_{k\in S^{\strA}_{t}}=x_t,S^{\strA}_{t}=s_t, 1\leq t\leq t_0\right\}\right.,
\end{split}
\end{equation*}%\footnote{\color{red}changed $\strA$ to $\strGfull{t_0+1}$ to connect it to the second step of the induction. These different notation give same value at $t_0+1$.}
and	$\hat{W}\lleq \hat{W}'$ a.s., {where $\eqd$ denotes that random variables on both sides are identically distributed.}
\end{customprop}

\begin{proof}[Proof of Proposition~\ref{prop:CouplingForProp}]
	First,  given $\big\{X_{S^{\strA}_{t}}=x_t,S^{\strA}_{t}=s_t, 1\leq t\leq t_0\big\}$, $\big[W^{\strA}_{S^{\strA}_{t_0},t_0}\big]$ is determined. To simplify the notation, we assume $W^{\strA}_{S^{\strA}_{t_0},t_0}=\ww_{t_0}\in \spaceu$ given $\big\{X_{S^{\strA}_{t}}=x_t,S^{\strA}_{t}=s_t, 1\leq t\leq t_0\big\}$.

	In addition, $\big[W^{\strA}_{S^{\strA}_{t_0+1},t_0}\big]$  is determined by $\big\{X_{S^{\strA}_{t}}=x_t,S^{\strA}_{t}=s_t, 1\leq t\leq t_0\big\}$ and the sequential procedure $\TT^{\strA}$. To simplify the notation, we assume $\big[W^{\strA}_{S^{\strA}_{t_0+1},t_0}\big]=\ww_{t_0+1}^*\in \spaceo$ given $\{X_{S^{\strA}_{t}}=x_t,S^{\strA}_{t}=s_t, 1\leq t\leq t_0\}$. We clarify that $\ww_{t_0+1}^*$ is a deterministic (and measurable) function of $x_t,s_t$ for $1\leq t\leq t_0$ (depending on the sequential procedure $\TT^{\strA}$).	According to Lemma~\ref{lemma:kka} (replacing $t$ by $t_0$), the conditional distribution of
	$\big[W^{\strA}_{S^{\strA}_{t_0+1},{t_0+1}}\big]$ given $\big\{X_{S^{\strA}_{t}}=x_t,S^{\strA}_{t}=s_t, 1\leq t\leq t_0\big\}$ is the same as the conditional distribution given $\big[W^{\strA}_{S^{\strA}_{t_0+1},t_0}\big]=\ww^*_{t_0+1}$. Moreover, the conditional density is $\KKa(\ww_{t_0+1}^*,\cdot)$.

	We perform a similar analysis by
	replacing $\TT^{\strA}$ by $\TT^{\strGfull{t_0}}$ in the above analysis. We denote $\big[W^{\strGfull{t_0}}_{S^{\strGfull{t_0}}_{t_0+1},t_0}\big]=\ww_{t_0+1}$ and obtain that the conditional density of $\big[W^{\strGfull{t_0}}_{S^{\strGfull{t_0}}_{t_0+1},t_0+1}\big]$ given $\big\{X_{S^{\strA}_{t}}=x_t,S^{\strA}_{t}=s_t, 1\leq t\leq t_0\big\}$ is $\KKa(\ww_{t_0+1},\cdot)$.

	According to the above analysis and Strassen Theorem for pospace (Fact~\ref{fact:strassen-pospace}), to prove the proposition, it is sufficient to show $\KKa(\ww_{t_0+1},\cdot)\lleq_{st}\KKa(\ww_{t_0+1}^*,\cdot)$. By Lemma~\ref{lemma:monotone-kka}, we have $\KKa(\uu,\cdot)\lleq_{st}\KKa(\uu',\cdot)$ for any $\uu\lleq\uu'\in\spaceo$. Thus, it is sufficient to show that $\ww_{t_0+1}\lleq \ww_{t_0+1}^*$.

	Now we compare $\ww_{t_0+1}$ and $\ww^*_{t_0+1}$. According to the definition of $\TT^{\strGfull{t_0}}$ and Lemma~\ref{lemma:one-step-ho}, we know $\ww_{t_0+1}=\Ho([\ww_{t_0}])$. There are two cases: 1) $\ww_{t_0+1}^*=\ab$, and 2) $\ww_{t_0+1}^*\neq\ab.$ We analyze these cases separately. For the first case, $\ww_{t_0+1}\lleq \ww_{t_0+1}^*$ by definition of the partial order. For the second case, according to Lemma~\ref{lemma:control} and Lemma~\ref{lemma:ho-domin}, we can see that $\ww_{t_0+1}=\Ho([\ww_{t_0}])\neq \ab$. Write $\ww_{t_0} = (w_{t_0,1},\cdots, w_{t_0,m})$ for some $m$, then $\ww^*_{t_0+1}$ can be written as $\ww^*_{t_0+1}=(w_{t_0,k_1},\cdots, w_{t_0,k_l})$ for some distinct $k_1,\cdots,k_l\in\{1,\cdots,m\}$. According to Lemma~\ref{lemma:control}, for $\TT^{\strA}$ to control $\lpcr$ at time $t_0+1$, $\ww^*_{t_0+1}$ satisfies $\sum_{i=1}^l w_{t_0,k_i}\leq \alpha l$. Thus, according to Lemma~\ref{lemma:ho-domin}, $\ww_{t_0+1}=\Ho([\ww_{t_0}])\lleq [\ww^*_{t_0+1}]=\ww^*_{t_0+1}$.
	\end{proof}
	\bigskip

\begin{customprop}{B.2}%\label{prop:monotone-o}
Suppose that model {\mhomo} holds. Then
for any $\yy, \yy'\in \spaceo$ such that $\yy \lleq \yy'$, there exists a coupling $(\hat{Y}_s,\hat{Y}'_s), s=0, 1, ... $,  satisfying
\begin{enumerate}
	\item  $\{\hat{Y}_s:s\geq 0\}$ has the same distribution as the conditional process $\{Y_s:s\geq 0\}$ given $Y_{0}=\yy$, and $\{\hat{Y}'_s:s\geq 0\}$ has the same distribution as the conditional process $\{Y'_s:s\geq 0\}$ given $Y_{0}=\yy'$.
	\item  $\hat{Y}_s\lleq \hat{Y}'_s$, a.s. for all $s\geq 0$.
\end{enumerate}
Moreover, the process $(\hat{Y}_s,\hat{Y}'_s)$ does not depend on $\TT^{\strA}$, $t_0$, or the information filtration $\fil^{\strA}_{t_0}$.
\end{customprop}
\begin{proof}[Proof of Proposition~\ref{prop:monotone-o}]
	Recall $Y_s=\left[\left(W^{\strGfull{t_0}}_{k,t_0+s}\right)_{k\in S^{ \strGfull{t_0}}_{t_0+s}}\right]$. By letting $t=t_0$ in Lemma~\ref{lemma:kko}, we obtain that $\{Y_s\}_{s\geq 0}$ is a homogeneous Markov chain, whose transition kernel is $\KKo$, which is independent of the sequential procedure $\TT^{\strA}$, $t_0$, and the information filtration $\fil^{\strA}_{t_0}$.
For the rest of the proof, according to Definition~\ref{def:kernel} and Fact~\ref{fact:strassen-markov}, it is sufficient to show that $\KKo$ is  stochastically monotone. That is,
%it is sufficient to show that
$\KKo(\uu,\cdot)\lleq_{st}\KKo(\uu',\cdot)$ for any $\uu,\uu'\in\spaceo$ with $\uu\lleq \uu'$. Thus, it is sufficient to show that for all $\uu\lleq \uu'$ {there exists a coupling $(\hat{V},\hat{V}')$} such that $\hat{V}\sim  \KKo(\uu,\cdot)$, $\hat{V}'\sim  \KKo(\uu',\cdot)$ and $\hat{V}\lleq \hat{V}'$ a.s. In what follows, we  construct such a coupling.

	For $\uu\lleq \uu'$ with $\uu,\uu'\in\spaceo$, we know that $\Hu(\uu)\lleq \Hu(\uu')$ by Lemma~\ref{lemma:h-monotone}.
	By the definition of the partial order, this implies that $\card(\Hu(\uu'))\leq \card(\Hu(\uu))$ and $\Hu(\uu)_k\leq \Hu(\uu')_k$	for each $1\leq k\leq \card(\Hu(\uu'))$. By Lemma~\ref{lemma:monotone-v}, this further implies
	\begin{equation}
		K(\Hu(\uu)_k,\cdot)\leq_{st}K(\Hu(\uu')_k,\cdot)
	\end{equation}
for $k=1,...,\card(\Hu(\uu'))$. Thus, by Strassen's Theorem for random variables (Fact~\ref{fact:straseen-rv}), this implies that there exists a coupling $(\hat{Z}_k,\hat{Z}'_k)$ such that
\begin{equation}\label{eq:z-coupling}
	\hat{Z}_k\sim K(\Hu(\uu)_k,\cdot), \hat{Z}'_k\sim K(\Hu(\uu')_k,\cdot), \text{ and } \hat{Z}_k \leq \hat{Z}'_k \text{ a.s.}
\end{equation}
for $k=1,...,\card(\Hu(\uu'))$.
In addition, we choose the coupling so that $(\hat{Z}_k,\hat{Z}'_k)$ are independent for different $k$. For $\card(\Hu(\uu'))<k\leq \card(\Hu(\uu))$, we construct $\hat{Z}_k\sim K(\Hu(\uu)_k,\cdot)$  so that $\hat{Z}_k$'s are independent for different $k$.
Let $\hat{Z}=(\hat{Z}_1,\cdots, \hat{Z}_{\card(\Hu(\uu))})$ and $\hat{Z}'=(\hat{Z}'_1,\cdots,\hat{Z}'_{\card(\Hu(\uu'))})$.

For this coupling, it is easy to verify
\begin{equation}
	\card(\hat{Z})\geq \card(\hat{Z}') \text{ and } \hat{Z}_k \leq \hat{Z}'_k \text{ for } 1\leq k\leq \card(\hat{Z}') ~~ a.s.
\end{equation}
Thus, $[\hat{Z}]\lleq [\hat{Z}']$ a.s.
Let $\hat{V}=[\hat{Z}]$ and $\hat{V}'=[\hat{Z}']$. Then, our coupling $(\hat{V}, \hat{V}')$ gives
\begin{equation}
	\hat{V}\lleq \hat{V}' \text{ a.s.}
\end{equation}
On the other hand, by Lemma~\ref{lemma:so-markov-generation}, we have
\begin{equation}
	\hat{V}\sim \KKo(\uu,\cdot)\text{ and } \hat{V}'\sim\KKo (\uu',\cdot).
\end{equation}
Therefore,
\begin{equation}
	\KKo(\uu,\cdot)\lleq_{st} \KKo(\uu',\cdot).
\end{equation}

\end{proof}

\subsection{Proof of supporting lemmas in Section~\ref{sec:lemma-odering}}\label{sec:proof-lemma}
\pospace*
\begin{proof}[Proof of Lemma~\ref{lemma:pospace}]
	First, $\spaceo$ is the union of polish spaces $\{\uu=(u_1,\cdots,u_m):0\leq u_1\leq\cdots \leq u_m\leq 1\}$ and $\{\ab\}$. Thus, it is also a polish space. Second, it is straightforward to verify that  $d(\uu,\vv)$ is a metric defined over $\spaceo$.

	Now, we verify that the partial order relationship $\lleq$ is closed over $\spaceo$. To see this, let $\uu,\vv\in\spaceo$ satisfying $\uu\not\lleq\vv$. There are two cases: 1) $\card(\uu)<\card(\vv)$, or 2) $\card(\uu)\geq\card(\vv)$ and there exists $m\in\{1,\cdots,\card(\vv)\}$ such that $u_m>v_m$. Let $B_{d}(\uu,\delta)$ and $B_{d}(\vv,\delta)$ be $d$-balls centering at $\uu$ and $\vv$ with $\delta$ chosen according to different cases: $\delta=1/2$ for the first case; and $\delta=\frac{u_m-v_m}{4}$ for the second case. Then, it is easy to verify that for all $\uu'\in B_{d}(\uu,\delta)$ and $\vv'\in B_{d}(\vv,\delta)$, we have $\uu'\not\lleq \vv'$. That is, the partial order relationship $\lleq$ is closed over $\spaceo$.
\end{proof}
\lemmaonestepho*
\begin{proof}[Proof of Lemma~\ref{lemma:one-step-ho}]
If $\uu=\ab$, then $[\uu]=\ab$ and $|S_t|=0$. This implies $\Io([\uu])=0$ and $\Ho([\uu])=\ab$. In the rest of the proof we assume that $\uu\neq \ab$.
By Step 1 of Algorithm~\ref{alg:one-step-rule1}, we obtain that
$[\uu]=(W_{k_1,t},\cdots,W_{k_{|S_t|},t})$ where $S_{t}=\{k_1,\cdots,k_{|S_t|}\}$ and $W_{k_1,t}\leq\cdots W_{k_{|S_t|},t}$. According to Step 2 and 3 of the algorithm and the definition of $\Io([\uu])$ in \eqref{eq:io}, the largest $n$ making $R_{n}\leq \alpha$ is $\Io([\uu])$ and $\Ho([\uu])=[(W_{k,t})_{k\in S_{t+1}}]$.
\end{proof}
\lemmacontrol*
\begin{proof}[Proof of Lemma~\ref{lemma:control}]
By definition and the $\fil_t$ measurability of $S_{t+1}$,
\begin{equation}
	\lpcr_{t+1}(\TT)=\expe\left[\frac{\sum_{k\in S_{t+1}}\ind(\tau_{k}<t)}{|S_{t+1}|\vee 1}\Big|\mathcal{F}_t\right]= \frac{\sum_{k\in S_{t+1}}\prob(\tau_{k}<t|\mathcal{F}_t)}{|S_{t+1}|\vee 1}=\frac{\sum_{k\in S_{t+1}}W_{k,t}}{|S_{t+1}|\vee 1}.
\end{equation} Thus, $\TT\in \adset_{\alpha}$ if and only if
\begin{equation}
	\frac{\sum_{k\in S_{t+1}}W_{k,t}}{|S_{t+1}|\vee 1}\leq \alpha \text{ a.s.,}
\end{equation}
which is equivalent to
\begin{equation}
	\sum_{k\in S_{t+1}}W_{k,t}\leq \alpha |S_{t+1}| ~~ a.s.,
\end{equation}
for every $t$.
\end{proof}
\lemmahodomin*
\begin{proof}[Proof of Lemma~\ref{lemma:ho-domin}]
We first prove the `Moreover' part of the lemma by contradiction. If on the contrary $\Ho([\uu])=\ab$ and there exists a non-empty set $S\subset\{1,\cdots, m\}$ such that $\sum_{i\in S}u_i\leq \alpha|S|$, then there exists $i\in S$ such that $u_i\leq \alpha$. This further implies $[\uu]_1\leq u_i\leq \alpha$ and $\Io([\uu])\geq 1$, which contracts with the assumption $\Ho([\uu])=\ab$.

We proceed to the proof of the rest of the lemma. We first prove that  $l$ in the lemma satisfies $l\leq \Io([\uu])$. To see this, recall that $([\uu]_1,\cdots,[\uu]_m)$ is the order statistic of $(u_1,\cdots,u_m)$. Thus,
%\begin{equation}
%	[\uu]_1\leq u_{k_1}, \cdots, [\uu]_l\leq u_{k_l}.
%\end{equation}
%Then, $\sum_{i=1}^l u_{k_i}\leq \alpha l$ implies
\begin{equation}\label{eq:l}
	\sum_{i=1}^l [\uu]_i\leq  \sum_{i=1}^l u_{k_i} \leq  \alpha l.
\end{equation}
Recall $\Io([\uu])=\sup\{n: \sum_{i=1}^n[\uu]_i\leq \alpha n, n\in\{0,\cdots, m\}\}$. Thus, \eqref{eq:l} implies $l\leq \Io([\uu])$.

Next, we prove that $\Ho([\uu])\lleq [(u_{k_1},\cdots, u_{k_l})]$. Without loss of generality, assume $u_{k_1},\cdots,u_{k_l}$ are ordered.
That is, $u_{k_1}\leq\cdots\leq u_{k_l}$ and $[(u_{k_1},\cdots, u_{k_l})]=(u_{k_1},\cdots,u_{k_l})$.
Then, according to the definition of the order statistic $[\uu]$, we have $[\uu]_i\leq u_{k_i}$ for $i=1,\cdots, l$. Recall $\Ho([\uu])=([\uu]_1,\cdots,[\uu]_{\Io(\uu)})$.  This implies $\Ho([\uu])\lleq [(u_{k_1},\cdots, u_{k_l})]$.
\end{proof}
\hmonotone*
\begin{proof}[Proof of Lemma~\ref{lemma:h-monotone}]
If $\vv=\ab$, then $\Ho(\vv)=\ab$ and $\Ho(\uu)\lleq \ab=\Ho(\vv)$ by the definition of the partial order. In the rest of the proof we assume $\card(\vv)\geq 1$ and $\vv=(v_1,\cdots,v_{\card(\vv)})$. As we assumed $\uu\lleq \vv$, this implies $\card(\uu)\geq\card(\vv)\geq 1$. We further denote $\uu=(u_1,\cdots,u_{\card(\uu)})$

We first show that if $\sum_{i=1}^{L+1}v_i\leq \alpha(L+1)$ for some $L$, then $\sum_{i=1}^{L}v_i\leq \alpha L$. That is, $(\sum_{i=1}^{L}v_i)/L$ is increasing in $L$. To see this, consider two cases. If $v_{L+1}\leq \alpha$, then $v_1\leq\cdots\leq v_L\leq \alpha$ and thus $\sum_{i=1}^{L}v_i\leq \alpha L$. If $v_{L+1}>\alpha$, then $\sum_{i=1}^{L}v_i\leq \sum_{i=1}^{L+1}v_i -\alpha\leq \alpha L$. This result implies that $\sum_{i=1}^{L}v_i\leq \alpha L$ for all $1\leq L\leq\Io(\vv)$.

Now we show that $\Io(\uu)\geq\Io(\vv)$ by contradiction. If on the contrary $\Io(\uu)<\Io(\vv)$, then $\Io(\uu)+1\leq \Io(\vv)\leq \card(\vv)$ and
	\begin{equation}
		\sum^{\Io(\uu)+1}_{i=1} u_{i}\leq \sum_{i=1}^{\Io(\uu)+1} v_i\leq \alpha (\Io(\uu)+1).
	\end{equation}
	This contradicts with the definition of $\Io(\uu)$. Therefore, $\Io(\uu)\geq\Io(\vv)$.

We proceed to showing $\Ho(\uu)\lleq\Ho(\vv)$.
By the definition of $\Ho$, we have
$\Ho(\uu)=(u_1,\cdots, u_{\Io(\uu)})$ and $\Ho(\vv)=(v_1,\cdots,v_{\Io(\vv)})$. Since we assume $\uu\lleq \vv$, we have $u_i\leq v_i$ for all $i=1,\cdots,\Io(\vv)$. This shows that $\Ho(\uu)\lleq\Ho(\vv)$.

\end{proof}
\likelihoodratiomonotone*
\begin{proof}[Proof of Lemma~\ref{lemma:likelihood-ratio-monotone}]
	Let $g$ be a bounded increasing function. Then,
	\begin{equation}
	\begin{split}
				&\expe
		g(L_{\delta_2})-\expe g(L_{\delta_1})\\
		= & \expe_{Z\sim \delta_2q + (1-\delta_2)p}g\big(L(Z)\big)-\expe_{Z\sim \delta_1q + (1-\delta_1)p}g\big(L(Z)\big)\\
		= & \delta_2 \expe_{Z\sim q}g\big(L(Z)\big)+(1-\delta_2)\expe_{Z\sim p}g\big(L(Z)\big)\\
		& - \left\{\delta_1 \expe_{Z\sim q}g\big(L(Z)\big)+(1-\delta_1)\expe_{Z\sim p}g\big(L(Z)\big)\right\}\\
		= & (\delta_2-\delta_1) \left\{\expe_{Z\sim q}g\big(L(Z)\big) - \expe_{Z\sim p}g\big(L(Z)\big)\right\}.
	\end{split}
	\end{equation}
	Note that $L(Z)=q(Z)/p(Z)$ and $\expe_{Z\sim q}g\big(L(Z)\big)= \expe_{Z\sim p}\left\{L(Z)g\big(L(Z)\big)\right\}$. Thus, the above display can be further written as
	\begin{equation}
		\expe
		g(L_{\delta_2})-\expe g(L_{\delta_1})
		= (\delta_2-\delta_1)\expe_{Z\sim p}\left\{\big(L(Z)-1\big) g\big(L(Z)\big)\right\}.
	\end{equation}
	For notational simplicity, let $Y=L(Z)$ with $Z\sim p$. Then, $\expe(Y)=1$ and the above display implies
	\begin{equation}
				\expe g(L_{\delta_2})-\expe g(L_{\delta_1})
		= (\delta_2-\delta_1)\expe\left\{(Y-1)g(Y)\right\}=(\delta_2-\delta_1)\expe\left\{(Y-1)(g(Y)-g(1))\right\}\geq 0.
	\end{equation}
	The last inequality in the above display is due to the fact that $(Y-1)(g(Y)-g(1))\geq 0$ for all increasing function $g$. We remark that it is also a special case of Harris inequality \citep{harris1960lower}.
\end{proof}
\lemmadistv*
\begin{proof}[Proof of Lemma~\ref{lemma:dist-v}]
First, it is easy to see that $\{V_{k,s}\}_{s\geq 0}$ are independent and identically distributed processes for different $k$.
For the rest of the proof, it is sufficient to prove the lemma for $k=1$.
{For the ease of exposition, we use the notation $X_{k,s:t}$ to denote $(X_{k,r})_{s\leq r\leq t}$.}
First, $\prob\left(\tau_1<0\vert X_{1,1:0}\right)=\prob(\tau_1<0)=0=V_0$. Thus, it is sufficient to verify the update rule for $V_{1,t}$. A direct calculation gives
\begin{equation}
	\begin{split}
		 & \prob(\tau_1\leq t-1|X_{1,1:t})\\
		= & \frac{\sum_{s=0}^{t-1}\prob(\tau_1= s) \prod_{r=1}^s p(X_{1,r})\prod_{r=s+1}^t q(X_{1,r})}{\sum_{s=0}^{t-1}P(\tau_1= s) \prod_{r=1}^s p(X_{1,r})\prod_{r=s+1}^t q(X_{1,r}) + P(\tau_1\geq t) \prod_{r=1}^t p(X_{1,r})}\\
		= & \frac{\sum_{s=0}^{t-1}\theta(1-\theta)^{s} L_{1,(s+1):t}}{\sum_{s=0}^{t-1}\theta(1-\theta)^{s} L_{1,(s+1):t}+ (1-\theta)^{t}}\\
		= &\frac{Q_{1,t}}{Q_{1,t}+(1-\theta)^{t}}\\
	\end{split}
\end{equation}
where we write $L_{k,(s+1):t}:=\prod_{r=s+1}^t \frac{q(X_{k,r})}{p(X_{k,r})}$, the likelihood ratio between $p(\cdot)$ and $q(\cdot)$ based on the data $X_{1,(s+1):t}$, and $Q_{1,t}=\sum_{s=0}^{t-1}\theta(1-\theta)^{s} L_{1,(s+1):t}$.
Then,
\begin{equation}
	Q_{1,t}=\frac{(1-\theta)^t \prob(\tau_1\leq t-1|X_{1,1:t})}{1-\prob(\tau_1\leq t-1|X_{1,1:t})}.
\end{equation}
Note that
\begin{equation}
	Q_{1,t+1}=\sum_{s=0}^{t}\theta(1-\theta)^{s} L_{1,(s+1):t+1}
	= q(X_{1,t+1})/p(X_{1,t+1})\left\{\theta(1-\theta)^t + Q_{1,t}\right\}.
\end{equation}
Thus,
\begin{equation}
	\begin{split}
	&\prob(\tau_1\leq t|X_{1,1:t+1})\\
	 =  &\frac{Q_{1,t+1}}{Q_{1,t+1}+(1-\theta)^{t+1}}\\
	= &  \frac{q(X_{1,t+1})/p(X_{1,t+1})\left\{\theta(1-\theta)^t + Q_{1,t}\right\}}{q(X_{1,t+1})/p(X_{1,t+1})\left\{\theta(1-\theta)^t + Q_{1,t}\right\}+ (1-\theta)^{t+1}}\\
	= & \frac{q(X_{1,t+1})/p(X_{1,t+1})}{q(X_{1,t+1})/p(X_{1,t+1})+ (1-\theta)/\left\{\theta + (1-\theta)^{-t}Q_{1,t}\right\}}\\
	= & \frac{q(X_{1,t+1})/p(X_{1,t+1})}{q(X_{1,t+1})/p(X_{1,t+1})+ (1-\theta)/\left\{\theta + \frac{ \prob(\tau_1\leq t-1|X_{1,1:t})}{1-\prob(\tau_1\leq t-1|X_{1,1:t})}\right\}}.
	\end{split}
\end{equation}
We complete the proof by simplifying the above result.
\end{proof}
\lemmadeltat*
\begin{proof}[Proof of Lemma~\ref{lemma:delta_t}]
By symmetry, it is sufficient to prove the lemma for $k=1$. Recall
$L_{k,(s+1):t}=\prod_{r=s+1}^t \frac{q(X_{k,r})}{p(X_{k,r})}$ and $Q_{k,t}=\sum_{s=0}^{t-1}\theta(1-\theta)^{s} L_{1,(s+1):t}$.

A direct calculation using Bayes formula gives
	\begin{equation}
	\begin{split}
		  \delta_{k,t}
		= & \frac{\sum_{s=0}^{t-1}\prob(\tau_1= s) \prod_{r=1}^s p(X_{1,r})\prod_{r=s+1}^t q(X_{1,r})+P(\tau_1 = t) \prod_{r=1}^t p(X_{1,r})}{\sum_{s=0}^{t-1}P(\tau_1= s) \prod_{r=1}^s p(X_{1,r})\prod_{r=s+1}^t q(X_{1,r}) + P(\tau_1\geq t) \prod_{r=1}^t p(X_{1,r})}\\
		= & \frac{\sum_{s=0}^{t-1}\theta(1-\theta)^{s} L_{1,(s+1):t}+\theta(1-\theta)^t}{\sum_{s=0}^{t-1}\theta(1-\theta)^{s} L_{1,(s+1):t}+ (1-\theta)^{t}}\\
		= &\frac{Q_{1,t}+\theta(1-\theta)^t}{Q_{1,t}+(1-\theta)^{t}}\\
		= & V_{1,t} + \theta(1-V_{1,t})\\
		= &\theta+ (1-\theta)V_{1,t}.
	\end{split}
\end{equation}
\end{proof}
\lemmamonotonev*
\begin{proof}[Proof of Lemma~\ref{lemma:monotone-v}]
	We first study the conditional distribution of $X_{1,t+1}$ given $V_{1,0},\cdots,V_{1,t}$. According to the change point model \mhomo{}, we know that $X_{1,t+1}$ is conditionally independent of $V_{1,0},\cdots,V_{1,t}$ given the event $\{\tau_1\leq t\}$. That is, given $V_{1,0},\cdots, V_{1,t}$, the conditional density function of $X_{1,t+1}$ is
	$
	\delta_{1,t}q(x)+(1-\delta_{1,t}) p(x),
	$
	which depends on $X_{1,1},\cdots, X_{1,t}$ only through $V_{1,t}$.

	Let the function $L(x):=q(x)/p(x)$ and let $L_{k,t+1}:=q(X_{k,t+1})/p(X_{k,t+1})$. Then, $L_{1,t+1}= L(X_{1,t+1})$, whose conditional distribution given $V_{1,0},\cdots, V_{1,t}$ only depends on $V_{1,t}$. According to the iteration \eqref{eq:v-k-t}, this implies that the process $\{V_{1,t}\}_{t\geq 0}$ is a Markov process. Note that  $\delta_{1,t}$ and the iteration \eqref{eq:v-k-t} depend on $t$ only through $V_{1,t}$. Thus, this Markov chain is a homogeneous Markov chain. We now show that its transition kernel is stochastically monotone.

	Let $\delta(x)=\theta+(1-\theta)x$. For $x\in (0,1)$, we consider the following steps of generating a random variable $V(x)$.
	\begin{enumerate}
		\item Generate $Z(x)$ with the density $\delta(x)q(\cdot)+(1-\delta(x))p(\cdot)$.
		\item Let
		\begin{equation}
			V(x)=\frac{L(Z(x))}{L(Z(x))+(1-\theta)(1-x)/(\theta+(1-\theta)x)}.
		\end{equation}
	\end{enumerate}
	From the iteration \eqref{eq:v-k-t} and $X_{1,t+1}|V_t=x\sim (1-\delta(x))q(\cdot)+\delta(x)p(\cdot)$, we can see that $V(x)$ has the same distribution as that of $V_{1,t+1}$ given $V_{1,t}=x$. In other words, $V(x)$ has the density function $K(x,\cdot)$.

	Now we show that $K(x,\cdot)\leq_{st} K(x',\cdot)$ for any $0<x\leq x'<1$ by coupling. Specifically, since $\delta(x)$ is increasing in $x$, $\delta(x)\leq \delta(x')$. Then, 	by Lemma~\ref{lemma:likelihood-ratio-monotone}, we know $L(Z(x))\leq_{st}L(Z(x'))$. According to the Strassen Theorem for random variables (Fact~\ref{fact:straseen-rv}), there exists a coupling $(\hat{L},\hat{L}')$, such that $\hat{L}\eqd L(Z(x))$, $\hat{L}'\eqd L(Z(x'))$ and $\hat{L}\leq \hat{L}'$ a.s. Then, let $\hat{V}=\frac{\hat{L}}{\hat{L}+(1-\theta)(1-x)/(\theta+(1-\theta)x)}\eqd V(x)$ and $\hat{V}'=\frac{\hat{L}'}{\hat{L}'+(1-\theta)(1-x')/(\theta+(1-\theta)x')}\eqd V(x')$.

	Because $\hat{L}\leq \hat{L}'$ and $x\leq x'$,
	\begin{equation}
	\begin{split}
				\hat{V}&=\frac{\hat{L}}{\hat{L}+(1-\theta)(1-x)/(\theta+(1-\theta)x)}\\
		&\leq 		\frac{\hat{L}'}{\hat{L}'+(1-\theta)(1-x)/(\theta+(1-\theta)x)}\\
		&\leq \frac{\hat{L}'}{\hat{L}'+(1-\theta)(1-x')/(\theta+(1-\theta)x')}\\
		& = \hat{V}' ~~ a.s.
	\end{split}
	\end{equation}
That is, $\hat{V}\leq\hat{V}'$ a.s., and $(\hat{V},\hat{V}')$ is a coupling of $(V(x),V(x'))$. Thus, $V(x)\leq_{st} V(x')$ and so is $K(x,\cdot)\leq_{st} K(x',\cdot)$.
\end{proof}
\lemmakka*
\begin{proof}[Proof of Lemma~\ref{lemma:kka}]
First, if $\card(\uu)=0$, then $\uu=\ab$, and $\big[W^{\strA}_{S^{\strA}_{t+1},t}\big]=\uu$ means that $S^{\strA}_{t+1}=\emptyset$. Thus, the conditional distribution of
$\big[W^{\strA}_{S^{\strA}_{t+1},t+1}\big]$  given $\big[W^{\strA}_{S^{\strA}_{t+1},t}\big]=\uu$ is a point mass at $\ab$, and $\KKa(\ab,\ab)=1$. In the rest of the proof, we focus on the case that $\uu\neq\ab$.

We start with deriving the conditional density of $W^{\strA}_{S^{\strA}_{t+1},t+1}$ at $\vv\in \spaceu$ given $X_{S^{\strA}_1,1}=x_1, S^{\strA}_1=s_1,\cdots, X_{S^{\strA}_{t},t}=x_t, S^{\strA}_{t}=s_t$, $S^{\strA}_{t+1}=s_{t+1}$ and $W^{\strA}_{S^{\strA}_{t+1},t}=\uu$ for some $x_1,\cdots,x_t$ and $s_1,\cdots, s_{t+1}$, and $\uu\in\spaceu$.
Clearly, the conditional density is $0$ when $\card(\uu)\neq\card(\vv)$, and is arbitrary when $\card(\uu)\neq |s_{t+1}|$ (the density of the random variable being conditional on is zero). Thus, we will focus on the case where $\card(\uu)=\card(\vv)=|s_{t+1}|=m$ for some $m\in\{1,\cdots,K\}$, and we will write $\uu=(u_1,\cdots,u_m)$ and $\vv=(v_1,\cdots,v_m)$.

Note that given $S^{\strA}_{t+1}=s_{t+1}, W^{\strA}_{S^{\strA}_{t+1},t}=\uu$, $W^{\strA}_{k,t+1}$'s are independent for different $k\in s_{t+1}$. Moreover, given $S^{\strA}_{t+1}=s_{t+1}, W^{\strA}_{S^{\strA}_{t+1},t}=\uu$, $W^{\strA}_{k,t+1}$ is the same as $V_{k,t+1}$ (defined in \eqref{eq:v-k-t}) for $k\in s_{t+1}$, and is independent of $X_{S^{\strA},1}=x_1, S^{\strA}=s_1,\cdots, X_{S^{\strA},t}=x_t$ and $S^{\strA}_{t}=s_t$. Thus, $W^{\strA}_{S^{\strA}_{t+1},t}$ is conditionally independent of $\fil^{\strA}_{t}$ given $S^{\strA}_{t+1}=s_{t+1}, W^{\strA}_{S^{\strA}_{t+1},t}=\uu$, and its conditional density (by Lemma~\ref{lemma:monotone-v}) is
\begin{equation}
	\prod_{l=1}^m K(u_{l},v_l),
\end{equation}
Because $\big[W^{\strA}_{S^{\strA}_{t+1},t}\big]$ is the order statistic of $W^{\strA}_{S^{\strA}_{t+1},t}$, we further obtain its conditional density at $\vv\in\spaceo$  given $S^{\strA}_{t+1}=s_{t+1}, W^{\strA}_{S^{\strA}_{t+1},t}=\uu$,
\begin{equation}
		\sum_{\pi\in\PP_m}\prod_{l=1}^m K(u_{l},v_{\pi(l)})=\sum_{\pi\in\PP_m}\prod_{l=1}^m K([\uu]_l,v_{\pi(l)})=\KKa([\uu],\vv),
\end{equation}
for $\vv\in\spaceo$ with $\card(\vv)=m$.
Observe that the above function is independent of $s_{t+1}$ for $|s_{t+1}|=m$ and depend on $\uu$ only through its order statistic $[\uu]$. Thus, we further conclude that $\big[W^{\strA}_{S^{\strA}_{t+1},t+1}\big]$ is conditionally independent of $\fil^{\strA}_{t}$ given $\big[W^{\strA}_{S^{\strA}_{t+1},t}\big]=\uu\in\spaceo$ satisfying $\card(\uu)=m$, and its conditional density is $\KKa(\uu,\cdot)$.	
\end{proof}
\lemmasomarkovgenerationA*
\begin{proof}[Proof of Lemma~\ref{lemma:so-markov-generation-A}]
	The lemma is obviously true when $m=0$. When $m\geq 1$, let $\zz=(z_1,\cdots,z_m)$.
		By step 1, the joint density for $(Z_{1},\cdots,Z_m)$ at $\zz$ is
		$$
		\prod_{i=1}^m K(u_i,z_i).
		$$
		By step 2, $V$ is the order statistic of $(Z_1,\cdots,Z_m)$. Thus, its density is
		\begin{equation}
			\sum_{\pi\in\PP_m} \prod_{i=1}^m K(u_i,z_{\pi(i)})=\KKa(\uu,\zz).
		\end{equation}
	\end{proof}

\lemmamonotonekka*
	\begin{proof}[Proof of Lemma~\ref{lemma:monotone-kka}]
	The lemma is obvious if $\uu'=\ab$. In what follows, we assume $\card(\uu')=m'\geq 1$ and $\card(\uu)=m$. Then, $\uu\lleq \uu'$ means $m\geq m'\geq 1$ and $u_l\leq u'_l$ for $1\leq l\leq m'$. Let $(Z_1,Z_1'),\cdots (Z_{m},Z_m')$ be independent random vectors such that $Z_l\sim K(u_l,\cdot)$, $Z_l'\sim K(u_l',\cdot)$ and $Z_l\leq Z_l'$ a.s. Such random vectors exists because of Strassen Theorem and 	Lemma~\ref{lemma:monotone-v} that the kernel $K(\cdot,\cdot)$ is stochastically monotone. In addition, for $m<l\leq m'$, let $Z_{l}'\sim K(u_l',\cdot)$  be independent random variables.

	Let $Z=(Z_1,\cdots,Z_m)\sim \KKa(\uu,\cdot)$, $Z'=(Z_1',\cdots,Z_{m'}')$, $V=[Z]$ and $V'=[Z']$. Then, $V\lleq V'$ a.s.
	 On the other hand, by Lemma~\ref{lemma:so-markov-generation-A}, we have
	\begin{equation}
		V\sim \KKa(\uu,\cdot) \text{ and } V'\sim\KKa(\uu',\cdot),
	\end{equation}
	and $V\lleq V'$ a.s. By Fact~\ref{fact:strassen-pospace}, the existence of such a coupling implies $\KKa(\uu,\cdot)\lleq_{st}\KKa(\uu',\cdot)$.
\end{proof}

\lemmakko*
\begin{proof}[Proof of Lemma~\ref{lemma:kko}]
	Apply Lemma~\ref{lemma:kka} by replacing $\TT^{\strA}$ by $\TT^{\strGfull{t}}$ and $t$ by $t+s$, we obtain that $\big[W^{\strGfull{t}}_{S^{\strGfull{t}}_{t+s+1},t+s+1}\big]$ is conditionally independent of $\fil^{\strGfull{t}}_{t+s}$ given $\big[W^{\strGfull{t}}_{S^{\strGfull{t}}_{t+s+1},t+s}\big]$.
	On the other hand, according to the one-step update rule in Algorithm~\ref{alg:one-step-rule1} and Lemma~\ref{lemma:one-step-ho}, we can see that $\big[W^{\strGfull{t}}_{S^{\strGfull{t}}_{t+s+1},t+s}\big]=\Ho\Big(\big[W^{\strGfull{t}}_{S^{\strGfull{t}}_{t+s},t+s}\big]\Big)$. Therefore, we further obtain that $\big[W^{\strGfull{t}}_{S^{\strGfull{t}}_{t+s+1},t+s+1}\big]$ is conditionally independent of $\fil^{\strGfull{t}}_{t+s}$ given $\big[W^{\strGfull{t}}_{S^{\strGfull{t}}_{t+s},t+s}\big]$.

	We proceed to derive its conditional density at $\vv$ given $\big[W^{\strGfull{t}}_{S^{\strGfull{t}}_{t+s},t+s}\big]=\uu$. We first notice that
	$\card(\vv)=|S^{\strGfull{t}}_{t+s+1}|=\Io(\uu)$ (by Lemma~\ref{lemma:one-step-ho}). Thus, the conditional density is zero when $\card(\vv)\neq \Io(\uu)$. For $\card(\vv)=\Io(\uu)$, by Lemma~\ref{lemma:kka} and the above analysis, the conditional density is
	$$
	\KKa(\Ho(\uu),\vv)=	\sum_{\pi \in \PP_{\Io(\uu)}}\prod_{l=1}^{\Io(\uu)}K(\Ho(\uu)_l,v_{\pi(l)})=\KKo(\uu,\vv).
	$$
	This completes the proof of the lemma.
\end{proof}
\lemmasomarkovgeneration*
\begin{proof}[Proof of Lemma~\ref{lemma:so-markov-generation}]
	The lemma is a direct application of Lemma~\ref{lemma:so-markov-generation-A} and $\KKo(\uu,\vv)=\KKa(\Ho(\uu),v)$.
\end{proof}

\section{Proof of Lemma~\ref{lem:update} and Propositions~\ref{prop:control} - \ref{prop:size}}
% In this section, we first present a useful lemma, and then followed by proof of Lemma~\ref{lem:update} and Propositions~\ref{prop:control} - \ref{prop:size}.
% % \subsection{Sequential procedures in $\adset_{\alpha}$}

% \subsection{Proof of Lemma~\ref{lem:update} and Propositions~\ref{prop:control} - \ref{prop:size}}

\begin{customlem}{1}%\label{lem:update}
	Under model {\mhomo}, $W_{k,0}=0$ for $1\leq k\leq K$ and $W_{k,t}$ can be computed using the following update rule for $1\leq k\leq K$,
	\begin{equation*}
		W_{k,t+1}= \begin{cases}
					\frac{q(X_{k,t+1})/p(X_{k,t+1})}{(1-\theta)(1-W_{k,t})/(\theta+(1-\theta)W_{k,t})+q(X_{k,t+1})/p(X_{k,t+1})}&\text{ for } 1\leq t\leq T_{k}-1,\\
					W_{k,T_{k}} & \text{ for } t\geq T_k.
		\end{cases}
	\end{equation*}
%	{\color{red} need a careful check.}
\end{customlem}
\begin{proof}[Proof of Lemma~\ref{lem:update}]
For each $k\in S_{t+1}$, according to the independence assumption for model \mhomo{},
\begin{equation}
	W_{k,t+1}=\prob(\tau_k<t+1|\mathcal{F}_{t+1})
	= \prob(\tau_k<t+1|X_{k,1:t+1}).
\end{equation}
On the other hand, according to Lemma~\ref{lemma:dist-v}, we have
\begin{equation}
\prob(\tau_k<t+1|X_{k,1:t+1})=\frac{q(X_{k,t+1})/p(X_{k,t+1})}{(1-\theta)(1-W_{k,t})/(\theta+(1-\theta)W_{k,t})+q(X_{k,t+1})/p(X_{k,t+1})}.
\end{equation}
Thus, for $k\in S_{t+1}$,
\begin{equation}\label{eq:iter-W}
	W_{k,t+1}=\frac{q(X_{k,t+1})/p(X_{k,t+1})}{(1-\theta)(1-W_{k,t})/(\theta+(1-\theta)W_{k,t})+q(X_{k,t+1})/p(X_{k,t+1})}.
\end{equation}
Note that $k\in S_{t+1}$ is equivalent to $T_{k}\geq t+1$. Thus, \eqref{eq:iter-W} holds for $1\leq t\leq T_k-1$.
Moreover, for $t\geq T_{k}$,
\begin{equation}
	W_{k,t+1}=\prob(\tau_k<t+1|\mathcal{F}_{t+1})=\prob(\tau_k<t|X_{1,1:T_k}, T_k) = W_{k,T_k}.
\end{equation}
This completes our proof.
\end{proof}

We proceed to the proofs of propositions.
\begin{customprop}{1}%\label{prop:control}
Suppose that we obtain the index set $S_{t+1}$ using Algorithm~\ref{alg:one-step-rule1}, given
the index set $S_t$ and information filtration $\mathcal F_t$ at time $t$. Then the LFNR at time $t+1$ satisfies
$$\expe\left(\frac{\sum_{k \in S_{t+1}} \ind\left( \tau_k<t\right) }{|S_{t+1}| \vee 1}\big\vert \mathcal F_t\right)\leq \alpha.$$
\end{customprop}

\begin{proof}[Proof of Proposition~\ref{prop:control}]
	First, it is easy to see that $S_{t+1}$ obtained from Algorithm~\ref{alg:one-step-rule1} is $\mathcal{F}_t$ measurable. Thus,
	\begin{equation}
		\expe\left(\frac{\sum_{k \in S_{t+1}} \ind\left( \tau_k<t\right) }{|S_{t+1}| \vee 1}\big\vert \mathcal F_t\right)
		= \frac{\sum_{k\in S_{t+1}}W_{k,t}}{|S_{t+1}|\vee 1}.
	\end{equation}
	 On the other hand, according to the second and third steps of the algorithm,
	\begin{equation}
	 	\frac{\sum_{k\in S_{t+1}}W_{k,t}}{|S_{t+1}|\vee 1}
	 	= R_n \leq \alpha.
	 \end{equation}
	 Therefore, $\expe\left(\frac{\sum_{k \in S_{t+1}} \ind\left( \tau_k<t\right) }{|S_{t+1}| \vee 1}\big\vert \mathcal F_t\right)\leq \alpha$.

%	 {The `moreover' part of the theorem follows directly by applying the above results to every time points.}
\end{proof}
\begin{customprop}{2}%\label{prop:control2}
Let $\TTp$ be defined in Algorithm~\ref{alg:SSS}. Then, $\TTp\in\adset_{\alpha}$.
\end{customprop}
\begin{proof}[Proof of Proposition~\ref{prop:control2}]
	This proposition is proved by combining the results of Proposition~\ref{prop:control} and Lemma~\ref{lemma:control}.
\end{proof}

\bigskip
\begin{customprop}{3}%\label{prop:size}
%For any $S \subset S_t$, satisfying
%$$\expe\left(\frac{\sum_{k \in S} \ind\left( \tau_k<t\right) }{|S| \vee 1}\big\vert \mathcal F_t\right)\leq \alpha,$$
%we have $\vert S_{t+1}\vert \geq \vert S\vert$.
Given LFNR level $\alpha$ and information filtration $\mathcal F_t$,
the index set $S_{t+1}$ given by Algorithm~\ref{alg:one-step-rule1} is locally optimal at time $t+1$.
%, following the notion of local optimality in Definition~\ref{def:local}.
\end{customprop}

\begin{proof}[Proof of Proposition~\ref{prop:size}]
Let $S_{t+1}$ be the index set obtained by Algorithm~\ref{alg:one-step-rule1}. By Lemma~\ref{lemma:one-step-ho}, $|S_{t+1}|=\Io([W_{S_{t},t}])$ and $[W_{S_{t+1},t}]=\Ho([W_{S_{t},t}])$. There are two cases: 1) $|S_{t+1}|=0$, and 2) $|S_{t+1}|=n\geq 1$. For the first case, $[W_{S_{t+1},t}]=\ab$.
Note that $\expe\left(\frac{\sum_{k \in S} \ind\left( \tau_k<t\right) }{|S| \vee 1}\big\vert \mathcal F_t\right)=\frac{\sum_{k \in S} W_{k,t}}{|S| \vee 1}$. By the `Moreover' part of Lemma~\ref{lemma:ho-domin}, we can see that the only set $S$ satisfying $\expe\left(\frac{\sum_{k \in S} \ind\left( \tau_k<t\right) }{|S| \vee 1}\big\vert \mathcal F_t\right)\leq \alpha$ is $S=\emptyset$. That is $|S|=0$. Thus, $|S_{t+1}|\geq |S|$.

For the second case where $|S_{t+1}|=n\geq 1$ and any set $|S|$ satisfying $\expe\left(\frac{\sum_{k \in S} \ind\left( \tau_k<t\right) }{|S| \vee 1}\big\vert \mathcal F_t\right)\leq \alpha$, we use Lemma~\ref{lemma:ho-domin} again and obtain that $[W_{S_{t+1},t}]=\Ho([W_{S_{t},t}])\lleq[W_{S,t}]$. This implies $|S_{t+1}|=\card([W_{S_{t+1},t}] )\geq \card([W_{S,t}])=|S|$.
% Let $S_{t+1}=\{k_1,\cdots, k_n\}$ be the index set obtained by Algorithm~\ref{alg:one-step-rule1}, and let  $S=\{k_{l_1},\cdots k_{l_m}\}$ be any set $S\subset S_{t}$ such that $\mathcal F_t$ measurable and
% $\expe\left(\frac{\sum_{k \in S} \ind\left( \tau_k<t\right) }{|S| \vee 1}\big\vert \mathcal F_t\right)\leq \alpha$. That is, $
% 	\sum_{k\in S}W_{k,t}=\sum_{k\in S}\prob(\tau_k<t|\mathcal{F}_t)\leq \alpha |S|.$
% Then,
% \begin{equation}
% 	\sum_{j=1}^m W_{k_{l_j},t}\leq \alpha m.
% \end{equation}
% On the other hand, since $W_{k_1,t}\leq\cdots\leq W_{k_{|S_t|},t}$ (first step of Algorithm~\ref{alg:one-step-rule1}), it is easy to see that
% \begin{equation}
% 	\sum_{i=1}^{m} W_{k_i,t}\leq  	\sum_{j=1}^m W_{k_{l_j},t}\leq \alpha m,
% \end{equation}
% which implies $R_m\leq \alpha$. Then, the proposition is proved by noticing (the third step of Algorithm~\ref{alg:one-step-rule1}) $$|S_{t+1}|=n=\sup\left\{m: R_m\leq \alpha, m\in \{1,\cdots, |S_t|\}\right\}.$$
\end{proof}

\section{Proof of Theorem~\ref{thm:large-sample-iid} and  Theorem~\ref{thm:large-complete-dependent}}\label{sec:proof-asymp}
\subsection{Proof of Theorem~\ref{thm:large-sample-iid}}\label{sec:proof-thm:large-sample-iid}
\begin{customthm}{2}%\label{thm:large-sample-iid}
Assume that model {\mhomo} holds and Assumption A1 is satisfied.
 To emphasize the dependence on $K$, we denote the proposed procedure by $\TT^{\strG}_K$, the corresponding information filtration at time $t$ by $\mathcal F_{K,t}^{\strG}$, and the index set at time $t$ by
$S_{K,t}^{\strG}$. Then, the following results hold for each $t \geq 1$.
%{\color{red} How about enumerate the results like this so that Theorem 2 and 3 have similar format?}{\color{blue}
\begin{enumerate}
	\item $\lim_{K\to\infty} \hat{\lambda}_{K, t}=\lambda_{t}$ a.s., where $\hat{\lambda}_{K, t} = \max\left\{W_{k,t}: k\in S_{K,t+1}^{\strG}\right\}$
is the threshold used by $\TTp_K$.
\item $\lim_{K\to\infty}\lpcr_{t+1}(\TTp_K)= \expe\left(V_{t}\Big|V_{s}\leq \lambda_{s},0\leq s\leq t\right),$ a.s. Moreover,
\begin{equation*}%\label{eq:cond-expe-expr}
	\expe\left(V_{t}\Big|V_{s}\leq \lambda_{s},0\leq s\leq t\right)
	=
		\begin{cases}
	1-(1-\theta)^{t},~~ t< \frac{\log(1-\alpha)}{\log(1-\theta)},\\
	\alpha, ~~ t\geq\frac{\log(1-\alpha)}{\log(1-\theta)}.
	\end{cases}
\end{equation*}
\item $\lim_{K\to\infty} K^{-1}\vert S_{K,t+1}^{\strG}\vert =\prob\left(V_1\leq \lambda_1,\cdots, V_t\leq \lambda_t \right)$ a.s.
\end{enumerate}

\end{customthm}

We start with a lemma that is useful for the proof of  Theorem~\ref{thm:large-sample-iid}. Its proof is provided in Section~\ref{sec:proof-lemma-asymp}.
\begin{restatable}{lemma}{lemmacontinuousdensity}\label{lemma:continuous-density}
Under model \mhomo{} and Assumption A1, we have the following results.
\begin{enumerate}
	\item For each $t\geq 1$, $(V_1,\cdots,V_t)$ has a continuous and strictly positive joint density function over $(0,1)^t$ (with respect to the Lebesgue measure).
	\item For any $(v_1,\cdots,v_t)\in (0,1)^t$, $\prob(V_1\leq v_1,\cdots,V_t\leq v_t)>0$.
	\item For any $(v_1,\cdots,v_t)\in (0,1)^t$, the conditional distribution of $V_{t+1}$  given $V_1\leq v_1,\cdots, V_t\leq v_t$ has a continuous and positive density function over $(0,1)$.
\end{enumerate}
	
\end{restatable}

\begin{proof}[Proof of Theorem~\ref{thm:large-sample-iid}]
	For a sufficiently large $t_0$ ($t_0>t$), let $\prob^*$ denote the probability measure for $(V_1,\cdots, V_{t_0})$, and let $\qrob$ be an arbitrary probability measure for a $t_0$-dimensional random vector. We define several mappings iteratively as follows. We initialize the mapping $\Lambda_0(\qrob)=1$ for every $\qrob$. Then, for $t\geq 1$, define
	\begin{equation}
	\begin{split}
			D_{t}(\lambda,\qrob) & = \qrob\left(V_{t}\leq\lambda, \VV_{t-1}\leq \bLam_{t-1}(\qrob)\right),\\
				N_{t}(\lambda,\qrob)&= \expe_{\qrob}\left[V_{t}\ind\left\{V_{t}\leq\lambda, \VV_{t-1}\leq \bLam_{t-1}(\qrob)\right\}\right],\\
			G_{t}(\lambda,\qrob) & = \frac{N_{t}(\lambda,\qrob)}{D_{t}(\lambda,\qrob)}=\expe_{\qrob}\left[V_{t}|V_{t}\leq\lambda, \VV_{t-1}\leq \bLam_{t-1}(\qrob)\right],\\
	\end{split}
	\end{equation}
	and
\begin{equation}
		\Lambda_{t}(\mathbb{Q})=\sup\left\{\lambda:G_{t}(\lambda,\mathbb{Q})\leq\alpha\text{ and }\lambda\in[0,1]\right\}.
	\end{equation}
In the above equations, we use notation $\VV_t=(V_1,\cdots,V_t)$ and $\bLam_t(\qrob)=(\Lambda_1(\qrob), \cdots, \Lambda_t(\qrob))$. In addition, $ \{\VV_t\leq \bLam_t(\qrob) \}$ denotes the event $\{V_1\leq \Lambda_1(\qrob),\cdots, V_t\leq \Lambda_t(\qrob)\}$.

The next lemma, whose proof is given in Section~\ref{sec:proof-lemma-asymp}, provides results about the above mappings. For two probability measures $\qrob$ and $\qrob'$ for a $t_0$-dimensional random vector $\VV_t$, their sup-norm is defined as $\|\qrob-\qrob'\|_{\infty}=\sup_{\vv\in \mathbb{R}^{t_0}}|\qrob(\VV_t\leq \vv)-\qrob'(\VV_t\leq \vv)|$. Then, we say a mapping $f(\qrob')$ is sup-norm continuous at $\qrob'=\qrob$ if $\lim_{\delta\to 0}\sup_{\qrob':\|\qrob'-\qrob\|_{\infty}\leq \delta} |f(\qrob')-f(\qrob)|=0.$
\begin{restatable}{lemma}{lemmamappings}\label{lemma:mappings}
	For each $1\leq t\leq t_0$, we have the following results.
\begin{enumerate}
	\item For any fixed $\qrob$,  $G_t(\lambda,\qrob)$ is non-decreasing in $\lambda$. Moreover, $G_t(\lambda,\prob^*)$ is strictly increasing in $\lambda\in (0,1]$ under Assumption A1.
	\item For any fixed $\lambda\in(0,1]$, $D_{t}(\lambda,\qrob)$, $N_{t}(\lambda,\qrob)$, and $G_t(\lambda,\qrob)$ are sup-norm continuous in $\qrob$ at $\qrob=\prob^*$ under Assumption A1.
	\item $\Lambda_t(\qrob)$ is sup-norm continuous at  $\qrob=\prob^*$ under Assumption A1. In addition, $\Lambda_t(\prob^*)>0$.
\end{enumerate}
\end{restatable}
By definition, $\lambda_{t}=\Lambda_{t}(\prob^*)$, where $\prob^*$ denotes the true probability measure of $(V_1,\cdots,V_{t_0})$. On the other hand, define the empirical measure (recall $V_{k,t}=\prob(\tau_k<t|X_{k,1},\cdots,X_{k,t})$)
	\begin{equation}
		\prob_{K}=\frac{1}{K}\sum_{k=1}^K \delta_{(V_{k,1},\cdots,V_{k,t_0})}.
	\end{equation}
It is not hard to verify that
	\begin{equation}
		\hat{\lambda}_{K,t}=\Lambda_{t}(\prob_{K}).
	\end{equation}
% It is not hard to see that following the proposed procedure, we have
% 		\begin{equation}
% 			S_{t}=\left\{
% 			k: W_{t,k}\leq \hat{\lambda}_{K,t}, \cdots, W_{1,k}\leq \hat{\lambda}_1
% 			\right\} ~~a.s.
% 		\end{equation}
Now we are able to prove the first part of theorem. Let	\begin{equation}
		\mathcal{C}=\left\{
		(-\infty,\xx]:\xx\in \mathbb{R}^{t_0}
		\right\}
	\end{equation}
	where $(-\infty,\xx]$ denotes the set $(-\infty,x_1]\times\cdots\times (-\infty,x_{t_0}]$.
	It is known that $\mathcal{C}$ is a Vapnik-\v{C}hervonenkis class and thus,
		$\lim_{K\to\infty}\sup_{C\in \mathcal{C}}\left|\prob_{K}(\VV_{t_0}\in C) - \prob^*(\VV_{t_0}\in C) \right|=0~~ a.s.$ (see, e.g., \cite{shorack2009empirical}).
	In other words,
\begin{equation}\label{eq:vc-class}
	\lim_{K\to\infty}\|\prob_{K}-\prob^*\|_{\infty}=0 \text{ a.s.}
\end{equation}  This result combined with the third statement of Lemma~\ref{lemma:mappings} implies
	\begin{equation}\label{eq:sup-conv}
		\lim_{K\to\infty} \Lambda_{t}(\prob_{K}) = \Lambda_{t}(\prob^*)~~ a.s.
	\end{equation}
	That is, $\lim_{K\to\infty}\hat{\lambda}_{K,t}=\lambda_t$ a.s. This completes our proof for the first statement of the theorem. We proceed to the second and third statements of the theorem.
Let  %\footnote{\yc{Old writing: $\expe_{(V_1,\cdots,V_{t+1})\sim \mathbb{Q}}\left(V_{t}\ind\left\{\VV_{t}\leq %\bLam_{t}(\qrob)\right\}\right)$?}}
	\begin{equation}
		J_{t}(\mathbb{Q})=\expe_{\mathbb{Q}}\left(V_{t}\ind\left\{\VV_{t}\leq \bLam_{t}(\qrob)\right\}\right)
	\text{ and }
		H_{t}(\mathbb{Q})=\mathbb{Q}\left(\VV_{t}\leq \bLam_t(\qrob)\right).
	\end{equation}
	We can see that the mapping $H_{t}$ is the composition of $D_{t}(\cdot,\qrob)$ and $\Lambda_{t}(\qrob)$. According to Lemma~\ref{lemma:continuous-density} and Lemma~\ref{lemma:mappings}, both mappings are sup-norm continuous at $\qrob=\prob^*$, and as a result, their composition $H_{t}(\qrob)$ is also sup-norm continuous at $\qrob=\prob^*$. Similarly, according to Lemma~\ref{lemma:continuous-density} and Lemma~\ref{lemma:mappings}, we can also see that $J_{t}(\qrob)$ is sup-norm continuous at $\qrob=\prob^*$.

	These results, combined with \eqref{eq:vc-class}, give
	\begin{equation}\label{eq:three-1}
		\lim_{K\to\infty} H_{t}(\prob_K)=		 H_{t}(\prob^*)\text{ a.s.,}
	\end{equation}
	and
	\begin{equation}\label{eq:three-2}
				\lim_{K\to\infty} J_{t}(\prob_K)=		 J_{t}(\prob^*) \text{ a.s.}
	\end{equation}
Note that
\begin{equation}\label{eq:three-3}
	H_{t}(\prob_K)=K^{-1}|S_{t+1}^{\strG}| \text{ and }\frac{J_{t}(\prob_K)}{H_{t}(\prob_K)}=\expe(\pcr_{t+1}(\TT)|\mathcal{F}_t).
\end{equation}
\eqref{eq:three-1}, \eqref{eq:three-2}, and \eqref{eq:three-3} together complete the second and third statements of the theorem.

In the rest of the proof, we show that \eqref{eq:cond-expe-expr} holds.

We first show that for $t\leq L:=\frac{\log(1-\alpha)}{\log(1-\theta)}$, $\lambda_t=1$. We show this by induction. For $t=0$, $\lambda_0=1$ by definition. Assume that for some $t\geq 1$, $\lambda_0=\cdots=\lambda_{t-1}=1$, then
\begin{equation}
	G_t(\lambda,\prob^*) = \expe\left[V_t|V_t\leq \lambda,\VV_{t-1}\leq \bLam_{t-1}(\prob^*)\right] = \expe\left[V_t|V_t\leq \lambda\right].
\end{equation}
In addition, $G_t(1,\prob^*)=\expe(V_t)=\prob(\tau_1<t)=1-(1-\theta)^{t}\leq\alpha$ for $t\leq L$.
By Lemma~\ref{lemma:mappings}, we know that $G_t(\lambda,\prob^*)$ is  increasing in $\lambda$.
Thus,
\begin{equation}
	\lambda_t=\sup\left\{\lambda:G_t(\lambda,\prob^*)\leq \alpha \text{ and }\lambda\in[0,1] \right\}=1.
\end{equation}
This completes the induction.
As a result, for $1\leq t\leq L$, $\expe\left[V_t|V_t\leq \lambda_t,\VV_{t-1}\leq\llll_{t-1}\right]=G_t(1,\prob^*)=1-(1-\theta)^{t}$.

We proceed to the proof of \eqref{eq:cond-expe-expr} for $t\geq L$. Note that $N_t(\lambda,\prob^*)$ and $D_t(\lambda,\prob^*)$ are continuous in $\lambda\in(0,1)$ (note that $\VV_t$ has a joint probability density function by Lemma~\ref{lemma:continuous-density}). Moreover, by Lemma~\ref{lemma:mappings} and Lemma~\ref{lemma:continuous-density}, $D_t(\lambda,\prob^*)>0$ for $\lambda>0$. Thus, for each $t$, $G_{t}(\lambda_t,\prob^*)=\alpha$ is equivalent to
% $G_{t}(\lambda,\prob^*)= \frac{N_t(\lambda,\prob^*)}{D_t(\lambda,\prob^*)}$ is also continuous in $\lambda\in(0,1)$. Thus, to show
% \eqref{eq:cond-expe-expr}, it suffices to show
\begin{equation}\label{eq:greater-alpha}
	G_{t}(1,\prob^*)\geq \alpha.
\end{equation}
We will show \eqref{eq:greater-alpha} $t> L$ by induction.
Let $\lfloor L\rfloor$ be the largest integer smaller or equal to $L$. According to the definition of $L$, we can see that
$$
G_{\lfloor L\rfloor+1}(1,\prob^*) = \expe (V_{\lfloor L\rfloor+1})= 1-(1-\theta)^{\lfloor L\rfloor+1}>\alpha.$$
This proves the base case for the induction.

Assume that for $1\leq s\leq t-1$, $G_{s}(1,\prob^*)>\alpha$. Then,
\begin{equation}\label{eq:gt-1-iter}
	G_{t}(1,\prob^*)
	= \expe\left[V_t|\VV_{t-1}\leq \llll_{t-1}\right]
	=\expe\left[ \expe(V_{t}|X_{1,1:t-1}) \middle|\VV_{t-1}\leq \llll_{t-1}\right],
\end{equation}
where $\llll_{t-1}=(\lambda_1,\cdots,\lambda_{t-1})$.
On the other hand,
\begin{equation}
\begin{split}
		&\expe(V_{t}|X_{1,1:t-1})\\
	=&\expe\left[ \prob\left(\tau_1<t|X_{1,1:t}\right)\middle|X_{1,1:t-1}\right]\\
	= & \prob\left(\tau_1<t|X_{1,1:t-1}\right)\\
	=& \prob\left(\tau_1\leq t-1|X_{1,1:t-1}\right)\\
	= & \delta_{1,t-1}\\
	= & \theta+(1-\theta)V_{t-1},
\end{split}
\end{equation}
where the last two equations are due to Lemma~\ref{lemma:delta_t}. The above display and \eqref{eq:gt-1-iter} give
\begin{equation}
		G_{t}(1,\prob^*)= \expe\left[ \theta+(1-\theta)V_{t-1} \middle|\VV_{t-1}\leq \llll_{t-1}\right] = \theta+(1-\theta) \expe\left[ V_{t-1} |\VV_{t-1}\leq \llll_{t-1} \right].
\end{equation}
% {\color{red} It is confusing. Why conditional on $\VV_{t-1}\leq \llll_{t-1}$ and $V_{t-1}\leq\lambda_{t-1}$ give the same result? Is $\expe(V_{t-1}|V_{t-}\leq\lambda_{t-1})=\expe(V_{t-1}|\VV_{t-1}\leq \llll_{t-1})$???}
By induction assumption, we have
\begin{equation}
	\expe\left[ V_{t-1} |\VV_{t-1}\leq \llll_{t-1} \right]=\alpha.
\end{equation}
The above two equations give
\begin{equation}
	G_{t}(1,\prob^*) = \theta+(1-\theta)\alpha>\alpha.
\end{equation}
This completes our proof.
\end{proof}
{\begin{remark}
A key observation in the above proof is that $\hat{\lambda}_{K, t}=\Lambda_{t}(\prob_{K})$ while $\lambda_t=\Lambda_{t}(\prob^*)$, where $\prob_{K}$ is the empirical measure and $\prob^*$ is the underlying probability measure of the process $\{V_{k,t}\}_{1\leq t\leq t_0}$. Thus, to show that $\hat{\lambda}_{K, t}$ converges to $\lambda_t$ (i.e., $\Lambda_{t}(\prob_{K})$ converges to $\Lambda_{t}(\prob^*)$), it suffices  to show that the functional $\Lambda_t(\cdot)$ is continuous and the empirical measure $\prob_{K}$ converges to $\prob^*$ in some sense as $K\to\infty$. In the proof, the above heuristics are justified through Vapnik-\v{C}hervonenkis (VC) theory. In particular, as a standard result in VC theory, the empirical measure converges to the underlying measure uniformly over the set $\mathcal{C}=\left\{
		(-\infty,\xx]:\xx\in \mathbb{R}^{t_0}
		\right\}$. That is, $\prob_{K}$ converges to $\prob^*$ in  $\|\cdot\|_{\infty}$ norm almost surely. The supporting lemma (Lemma~\ref{lemma:mappings}) is mainly  arguing that the functional of interest is continuous under this norm.

		Moreover,  VC theory and theory of empirical processes in general are helpful in understanding the convergence of empirical measure over general probability spaces. Based on VC theory, many additional results (e.g., convergence rate) can be developed in addition to the uniform convergence result over the set $\mathcal{C}$ mentioned above. We refer the readers to the book \citep{shorack2009empirical} and references therein for a comprehensive review.
\end{remark}

}

\subsection{Proof of Theorem~\ref{thm:large-complete-dependent}}\label{sec:proof-thm:large-complete-dependent}
\begin{customthm}{3}%\label{thm:large-complete-dependent}
Suppose that data follow a special case of the model given in Example \ref{example:model-partial} when $\eta = 1$ and
$\tau_0\sim Geom(\theta)$, and Assumption A2 holds.
 Let
	\begin{equation*}
		W_t=\prob\left(\tau_0 < t\Big|X_{k,s},1\leq k\leq K,1\leq s\leq t\right),
	\end{equation*}
	and
	\begin{equation*}
		T=\min\{t: W_t>\alpha\}.
	\end{equation*}
	Then, $\TTp_K=(T,\cdots, T)$.
Moreover, the following asymptotic results hold.
	\begin{enumerate}
		\item $\lim_{K\to\infty} (T-\tau_0) =1$ a.s.,
		\item $\lim_{K\to\infty}\lpcr_{t+1}(\TTp_K)=0$ a.s., %\footnote{Earlier: $\lim_{K\to\infty}\expe(\pcr_{t+1}(\TTp)|\mathcal{F}_t)=0$ a.s.,}
	\item $ \lim_{K\to\infty} K^{-1}\vert S_{K,t+1}^{\strG}\vert = \ind(\tau_0\geq t)$ a.s.
	\end{enumerate}
% 	define a stochastic system $V_t\in [0,1]$ and $\gamma_t\in [0,1]$ starting from $V_0=0$, $\gamma_0=1$, and evolving according to
% \begin{equation}
% \begin{cases}
% 		V_{t+1}|V_{t}=v_t&\sim \kappa_{\gamma_{t}}(v_t,\cdot)\\
% 		\gamma_{t+1}& = (\frac{\alpha}{v_t}\wedge 1) \gamma_{t}.
% \end{cases}
% \end{equation}
% {\color{red} the kernel to be specified.}
% Then, for any $t\geq 1$,
% \begin{enumerate}
% 	\item $\lim_{K\to\infty}\expe(\pcr_{t+1}(\TTp)|\mathcal{F}_t)= V_{t}\wedge \alpha$ in distribution.
% 	\item $\lim_{K\to\infty} K^{-1}\uti_t(\TTp)= \gamma_t$ in distribution.
% \end{enumerate}
\end{customthm}
\begin{proof}[Proof of Theorem~\ref{thm:large-complete-dependent}]
	We first note that under the model considered in this theorem, $W_{1,t}=\cdots=W_{K,t}= \prob(\tau_0<t|\mathcal{F}_t)$. Thus, according to $\TTp$, if $W_{1,t}\leq \alpha$, then $\sum_{k\in S_{t}}W_{t,k}\leq \alpha|S_{t}|$, and $S_{t+1}=S_t$. Moreover, if for some $t$ such that $S_{t}=\{1,\cdots,K\}$ and $W_{1, t+1}>\alpha$, then for any $S\neq\emptyset$, $\sum_{k\in |S|}W_{k,t+1}=W_{k,t+1}|S|>\alpha |S|$, and thus $S_{t+1}=\emptyset$. Thus,  $\TTp=(T,\cdots, T)$. In other words, $S_{t}=\{1,\cdots, K\}$ for $t\leq T$ and $S_{t}=\emptyset$ for $t>T$.

	Note that for $t\leq T$, $\fil_t=\sigma(\{W_{k,s}\}, 1\leq s\leq t, 1\leq k\leq K )$. Let $\tilde{W}_{k,t}=\prob(\tau_0<t|X_{k,s}, 1\leq k\leq K, 1\leq s\leq t)$, which is the conditional probability without deactivating any stream. Then, $W_{k,t}=\tilde{W}_{k,t}$ for $t\leq T$ where we recall $T=\inf\{t: \tilde{W}_{1,t}>\alpha\}$.
	We have
	\begin{equation}\label{eq:w-tilde}
	\begin{split}
		\tilde{W}_{k,t}
		=& \frac{\sum_{s=0}^{t-1}\theta(1-\theta)^s \prod_{r=s+1}^t\prod_{k=1}^K q(X_{k,r})/p(X_{k,r})}{\sum_{s=0}^{t-1}\theta(1-\theta)^s \prod_{r=s+1}^t\prod_{k=1}^K q(X_{k,r})/p(X_{k,r}) + (1-\theta)^{t}}\\
		= &\frac{\sum_{s=0}^{t-1}\theta(1-\theta)^s  \exp\{\sum_{k=1}^Kl_{k,s,t}\}}{\sum_{s=0}^{t-1}\theta(1-\theta)^s  \exp\{\sum_{k=1}^Kl_{k,s,t}\}+(1-\theta)^t},
	\end{split}
	\end{equation}
	where we define $l_{k,s,t}=\sum_{r=s+1}^t\log(q(X_{k,r})/p(X_{k,r}))$.

	For each $u\in \mathbb{Z}_+\cup\{0\}$, let $A_u=\{\tau_0=u\}$. By the strong law of large numbers, under Assumption A2,	
	\begin{equation}\label{eq:W-complte-dependent}
	 	\prob\left(\lim_{K\to\infty}\frac{1}{K}\sum_{k=1}^K l_{k,s,t} =\expe(l_{1,s,t}|\tau_0=u) \Big|A_{u}\right)=1
	\end{equation}
	for each $s,t,u\in \mathbb{Z}_+\cup\{0\}$ with $s<t$. In particular,
	\begin{equation}
		\expe(l_{1,s,t}|\tau_0=u)
		=
		\begin{cases}
			-(t-s)\expe_{Z_1\sim p}\log (p(Z_1)/q(Z_1))<0 & \text{ if } t\leq u\\
			 \expe_{Z_2\sim q} \log(q(Z_2)/p(Z_2))>0&\text{ if } t = u+1 \text{ and } s=u.
		\end{cases}
	\end{equation}
%\yc{So $t-s$ = 1 in the second equation above?}
	Thus, for each $s<t\leq u$
	we have
	\begin{equation}\label{eq:LLN-1}
	 	\prob\left(\lim_{K\to\infty}\sum_{k=1}^K l_{k,s,t} = -\infty \Big|A_{u}\right)=1,
	\end{equation}
%\yc{$-\infty$ in the equation above?}
	and for $t=u+1=s+1$,
	\begin{equation}\label{eq:LLN-2}
			 	\prob\left(\lim_{K\to\infty}\sum_{k=1}^K l_{k,s,t} = \infty \Big|A_{u}\right)=1.
	\end{equation}
%\yc{$\infty$ in the equation above?}
	According to \eqref{eq:w-tilde}, \eqref{eq:W-complte-dependent} and \eqref{eq:LLN-1}, we have that for each $t\leq u$
	\begin{equation}
			 	\prob\left(\lim_{K\to\infty}\tilde{W}_{k,t}=0\Big|A_{u}\right)=1.
	\end{equation}
	% which further implies  that for each $t\leq u$
	% \begin{equation}\label{eq:w-t-less-u}
	% 				 	\prob\left(\lim_{K\to\infty}\tilde{W}_{k,s}=0 \text{ for all } 1\leq s\leq t\Big|A_{u}\right)=1.
	% \end{equation}
	Moreover, for $t\geq u+1$,
	\begin{equation}
					 	\prob\left(\lim_{K\to\infty}\tilde{W}_{k,t}=1\Big|A_{u}\right)=1.
	\end{equation}
	% which further implies
	% \begin{equation}
	% 				 	\prob\left(\lim_{K\to\infty}\tilde{W}_{k,t}=1 \text{ for all } t\geq u+1\Big|A_{u}\right)=1,
	% \end{equation}
	
	Combining the above two equations for different $u\in \mathbb{Z}_+\cup\{0\}$, we arrive at
	\begin{equation}
					 	\prob\left(\lim_{K\to\infty}\tilde{W}_{k,t} =\ind(t\geq \tau_0+1)\right)=1.
	\end{equation}
	In other words,
	\begin{equation}
		\lim_{K\to\infty}\tilde{W}_{1,t}=\ind(t\geq \tau_0+1) \text{ a.s.}
	\end{equation}

Now we turn to the analysis of $W_{k,t}$ and $S_{t}$ for the proposed procedure. Let $\omega$ be a sample path with $\lim_{K\to\infty}\tilde{W}_{k,t}(\omega)=\ind(t\geq \tau_0(\omega)+1)$ for all $t=1,2,\cdots$. Then, there exists $K_0(\omega)$ large enough such that $\tilde{W}_{1,t}(\omega)<\alpha$  for $t\leq \tau_0(\omega)$ and $\tilde{W}_{1,\tau_0(\omega)+1}(\omega)>\alpha$ for all $K\geq K_0(\omega)$. Then, we have $T(\omega)=\inf\{t: \tilde{W}_{1,t}(\omega)>\alpha\}=\tau_0(\omega)+1$. Note that the set of such sample path $\omega$ has a probability of one.
Thus,
\begin{equation}
	\lim_{K\to\infty} (T-\tau_0)=1 \text{ and } \lim_{K\to\infty}W_{k,t}= 0 \text{ for } t\leq \tau_0 \text{ a.s.}
\end{equation}
This proves the first statement of the theorem. For the second statement, we have
\begin{equation}
	\lim_{K\to\infty} \expe\left(\pcr_{t+1}(\TTp)|\mathcal{F}_t\right)
	=\lim_{K\to\infty}\frac{\sum_{k=1}^K \ind(T> t) W_{k,t}}{\{\sum_{k=1}^K \ind(T> t)\}\vee 1}
	=\lim_{K\to\infty}W_{k,t}\ind(T> t)=0 \text{ a.s.}
\end{equation}
For the third statement, we have
\begin{equation}
	\lim_{K\to\infty}K^{-1}|S_{t+1}|=\lim_{K\to\infty} \ind(T> t)= \ind(\tau_0\geq t) \text{ a.s.}
\end{equation}
\end{proof}
\subsection{Proof of supporting lemmas in Section~\ref{sec:proof-thm:large-sample-iid}}\label{sec:proof-lemma-asymp}
\lemmacontinuousdensity*
\begin{proof}[Proof of Lemma~\ref{lemma:continuous-density}]
Note that the second statement of the lemma is obvious given the first statement, and the third statement is a straightforward application of a combination of the first and second statements. Thus, it suffices to show the first statement of the lemma. In what follows, we prove the first statement by induction.

For $Z_1$ follow the density function $p(\cdot)$, $Z_2$ follows the density function $q(\cdot)$, let $f_1(\cdot)$ and $f_2(\cdot)$ be the density functions of $q(Z_1)/p(Z_1)$ and $q(Z_2)/p(Z_2)$. By Assumption A1, $f_i(z)>0$ for all $z>0$ and $i=1,2$.

For $t=1$, under the model \mhomo{}, $X_{1,1}$ follows the mixture density $(1-\theta)p(\cdot)+\theta q(\cdot)$. Thus, $q(X_{1,1})/p(X_{1,1})$ has the density function $(1-\theta) f_1+\theta f_2$, which is strictly positive and continuous over $\mathbb{R}_+$. Note that
$V_1=\frac{q(X_{1,1})/p(X_{1,1})}{(1-\theta)/\theta+q(X_{1,1})/p(X_{1,1})}$. By standard calculation of density of random variable after transformation, we can see that the density of $V_1$ is
\begin{equation}\label{eq:density-v1}
	f_{V_1}(v)=\frac{c}{(1-v)^2}\left\{(1-\theta) f_1\left(\frac{cv}{1-v}\right)+\theta f_2\left(\frac{cv}{1-v}\right)\right\},
\end{equation}
where $c=(1-\theta)/\theta$. This density function is strictly positive and continuous for $v\in(0,1)$.

Assume the induction assumption that the joint density for $(V_1,\cdots,V_t)$, denoted by $f_{V_1,\cdots,V_t}(v_1,\cdots,v_t)$, is strictly positive and continuous over $(0,1)^{t}$. We proceed to showing $f_{V_{1},\cdots,V_{t+1}}(v_1,\cdots,v_{t+1})$ is strictly positive and continuous over $(0,1)^{t+1}$. Recall that $V_{t+1}=\frac{q(X_{t+1,1})/p(X_{t+1,1})}{(1-\theta)(1-V_t)/(\theta+(1-\theta)V_t)+q(X_{1,1})/p(X_{1,1})}$. With a similar derivation as that for \eqref{eq:density-v1}, we have the conditional density of $V_{t+1}$ given $V_1=v_1,\cdots, V_t=v_t$ is
\begin{equation}
\begin{split}
			&f_{V_{t+1}|V_1=v_1,\cdots, V_t=v_t}(v)\\
			=&\frac{c_t}{(1-v)^2}\left\{(1-\theta_{t}) f_1\left(\frac{c_tv}{1-v}\right)+\theta_t f_2\left(\frac{c_tv}{1-v}\right)\right\},
\end{split}
\end{equation}
where we define $c_t=\frac{(1-\theta)(1-v_t)}{\theta+(1-\theta)v_t}>0$ and
$\theta_t=\prob(\tau_1\leq t|V_1=v_1,\cdots,V_t=v_t)=v_t(1-\theta)+\theta\in(0,1)$. It is easy to see that both $c_t$ and $\theta_t$ are continuous in $v_t$. As a result, $f_{V_{t+1}|V_1=v_1,\cdots, V_t=v_t}(v_{t+1})$ is strictly positive and is continuous in $v_1,\cdots, v_{t+1}$ for $v_1,\cdots, v_{t+1}\in (0,1)$ and so is $f_{V_1,\cdots,V_{t+1}}(v_{1},\cdots,v_{t+1})=f_{V_1,\cdots,V_t}(v_1,\cdots, v_t)f_{V_{t+1}|V_1=v_1,\cdots, V_t=v_t}(v_{t+1})$. This completes our induction and the proof of the lemma.
\end{proof}
\lemmamappings*
%\yc{Do you want the footnote to be repeated here?}
\begin{proof}[Proof of Lemma~\ref{lemma:mappings}]
For $t=0,1,\cdots$ and $\lambda<\lambda'$, let $\tV$ be a random variable following the same distribution as $V_{t}|\VV_{t-1}\leq \bLam_{t-1}(\qrob)$. Then, by the definition of conditional expectation, we have
\begin{equation}
 			\begin{split}
 				&G_{t}\left(\lambda',\mathbb{Q}\right)- G_{t}\left(\lambda,\qrob\right)\\
 				= & Z^{-1} \left[
 				\expe_{\qrob}\left(\tV\ind\left\{\tV\leq \lambda'\right\}\right)\qrob\left(\tV\leq \lambda\right)-\expe_{\qrob}\left(\tV\ind\left\{\tV\leq \lambda\right\}\right)\qrob\left(\tV\leq \lambda'\right)
 				\right]\\
 					= &  Z^{-1} \left[
 				\expe_{\qrob}\left(\tV\ind\left\{\lambda<\tV\leq \lambda'\right\}\right)\qrob\left(\tV\leq \lambda\right)-\expe_{\qrob}\left(\tV\ind\left\{\tV\leq \lambda\right\}\right)\qrob\left(\lambda<\tV\leq \lambda'\right)
 				\right]
 				 				 			\end{split} 		
 			\end{equation}
 			where
 				$Z=\qrob\left(\tV\leq \lambda\right)\qrob\left(\tV\leq \lambda'\right)$.  Let $\tV'$ be an independent copy of $\tV$, then the above display implies
 				\begin{equation}
 				 				\begin{split}\label{eq:g-lam-mono-eq}
&G_{t}\left(\lambda',\mathbb{Q}\right)- G_{t}\left(\lambda,\qrob\right)\\
 				= & Z^{-1} \left[
 				\expe_{\qrob}\left(\tV' \ind\left\{\lambda<\tV'\leq \lambda',\tV\leq \lambda\right\}\right)-\expe_{\qrob}\left(\tV\ind\left\{\lambda<\tV'\leq \lambda',\tV\leq \lambda\right\}\right)
 				\right] \\
 				= &Z^{-1}
 				\expe_{\qrob}\left[\left(\tV'-\tV\right) \ind\left\{\lambda<\tV'\leq \lambda',\tV\leq \lambda\right\}\right] 				,
 					\end{split}
 				\end{equation}
Because $\left(\tV'-\tV\right) \ind\left\{\lambda<\tV'\leq \lambda',\tV\leq \lambda\right\}\geq 0$,  $G_{t}\left(\lambda',\mathbb{Q}\right)- G_{t}\left(\lambda,\qrob\right)\geq 0$ from the above display.

In what follows, we use induction to prove the rest of the lemma. Namely, for $\lambda\in (0,1)$, we will prove the following statements for $t=1,2,\cdots,t_0$:
\begin{equation}\label{eq:state-1}
G_t(\lambda,\prob^*) \text{ is strictly increasing in }\lambda;
\end{equation}
\begin{equation}\label{eq:state-2}
	D_{t}(\lambda,\qrob), N_{t}(\lambda,\qrob), \text{ and } G_{t}(\lambda,\qrob) \text{ are sup-norm continuous at } \qrob=\prob^*;
\end{equation}
\begin{equation}\label{eq:state-3}
	\Lambda_{t}(\qrob) \text{ is sup-norm continuous at } \qrob=\prob^*.
\end{equation}
%and $\Lambda_{t}(\prob^*)>0$.

We start with the base case that $t=1$. In this case, the conditional distribution $V_1|\VV_0\leq \bLam_{0}(\qrob)$ is the same as the unconditional distribution of $V_1$ for any $\qrob$. According to Lemma~\ref{lemma:continuous-density}, $V_1$ has a strictly positive and continuous density function over $(0,1)$ under $\prob^*$. Thus, $\prob^*\left(\left(\tV'-\tV\right) \ind\left\{\lambda<\tV'\leq \lambda',\tV\leq \lambda\right\}\geq 0\right)>0$ for $\tV$ and $\tV'$ are identically distributed as $V_1$.
%appeared in \eqref{eq:g-lam-mono-eq}.
According to \eqref{eq:g-lam-mono-eq},  $G_1(\lambda',\prob^*)-G_1(\lambda,\prob^*)> 0$. That is, $G_1(\lambda,\prob^*)$ is strictly increasing in $\lambda$. This proves the base case for \eqref{eq:state-1}. For \eqref{eq:state-2} and \eqref{eq:state-3} the proof of the base cases is similar to that of the induction given below. Thus, we omit the proof for their base cases here.

% {\color{red} add definition $D_1(\lambda,\qrob)=N_1(\lambda,\qrob)=G_1(\lambda,\qrob)=\Lambda_1(\qrob)=1$.}

Now we assume that \eqref{eq:state-1}, \eqref{eq:state-2}, and \eqref{eq:state-3} hold for $t=1,2,\cdots, s-1$. We proceed to prove these equations for $t=s$. First, note that $V_t|\VV_{t-1}\leq \bLam_{t-1}(\prob^*)$ has a continuous and strictly positive density function over $(0,1)$. Thus, \eqref{eq:state-1} is proved by combining \eqref{eq:g-lam-mono-eq} with similar arguments as those for the base case where $t=1$.

{\bf Proof of \eqref{eq:state-2} for $t=s$.}
By the induction assumption, $\Lambda_1(\qrob),\cdots,\Lambda_{s-1}(\qrob)$ is sup-norm continuous in $\qrob$ at $\qrob=\prob^*$. This implies that $(\lambda,\bLam_{s-1}(\qrob))$, a vector-valued mapping, is also sup-norm continuous in $\qrob$ at $\qrob=\prob^*$. On the other hand, $(\lambda,\bLam_{s-1}(\prob^*))\in(0,1]^s$ by induction assumptions, and $\VV_t$ has a continuous joint probability cumulative function at $(\lambda,\bLam_{s-1}(\prob^*))$ (by Lemma~\ref{lemma:continuous-density}). Combining these results, we can see that  $\prob^*\left(V_{s}\leq\lambda, \VV_{s-1}\leq  \bLam_{s-1}(\qrob)\right)$ is sup-norm continuous at $\qrob=\prob^*$.

% That is,
% \begin{equation}
% 	 \lim_{\|\qrob-\prob^*\|_{\infty}\to 0} \prob^*\left(V_{s}\leq\lambda, \VV_{s-1}\leq  \bLam_{s-1}(\qrob)\right)= \prob^*\left(V_{s}\leq\lambda, \VV_{s-1}\leq  \bLam_{s-1}(\prob^*)\right).
% \end{equation}

% {\color{red} define the notation $\|\qrob-\qrob'\|_{\infty}=\sup_{\vv_{t_0}\in\mathbb{R}^{t_0}}|\qrob(\VV_{t_0}\leq \vv_{t_0})-\qrob'(\VV_{t_0}\leq \vv_{t_0})|$.}

% Now we analyze the mapping $D_{s}(\lambda,\qrob)=\qrob\left(V_{s}\leq\lambda, \VV_{s-1}\leq \bLam_{s-1}(\qrob)\right)$.
% By the definition of sup-norm continuous, it is easy to see that $\qrob(\VV_s\leq \xx )$ is sup-norm continuous in $\qrob$ for any fixed $\xx\in \mathbb{R}^{s}$. In addition,

Now we analyze the mapping $D_{s}(\lambda,\qrob)=\qrob\left(V_{s}\leq\lambda, \VV_{s-1}\leq \bLam_{s-1}(\qrob)\right)$.
\begin{equation}
\begin{split}
	&\left|D_{s}(\lambda,\qrob)-D_{s}(\lambda,\prob^*)\right|\\
	 =&  \left|\qrob\left(V_{s}\leq\lambda, \VV_{s-1}\leq \bLam_{s-1}(\qrob)\right)-\prob^*\left(V_{s}\leq\lambda, \VV_{s-1}\leq \bLam_{s-1}(\prob^*)\right)\right|\\
	 \leq & \left|\qrob\left(V_{s}\leq\lambda, \VV_{s-1}\leq \bLam_{s-1}(\qrob)\right)-\prob^*\left(V_{s}\leq\lambda, \VV_{s-1}\leq \bLam_{s-1}(\qrob)\right)\right|\\
	 &+\left|\prob^*\left(V_{s}\leq\lambda, \VV_{s-1}\leq  \bLam_{s-1}(\qrob)\right)-\prob^*\left(V_{s}\leq\lambda, \VV_{s-1}\leq \bLam_{s-1}(\prob^*)\right)\right|\\
	 \leq & \|\qrob-\prob^*\|_{\infty}\\
	 &+\left|\prob^*\left(V_{s}\leq\lambda, \VV_{s-1}\leq  \bLam_{s-1}(\qrob)\right)-\prob^*\left(V_{s}\leq\lambda, \VV_{s-1}\leq \bLam_{s-1}(\prob^*)\right)\right|.
\end{split}
\end{equation}
Therefore, %\yc{$s+1$ to $s$ in the equation below?}
\begin{equation}
\begin{split}
	 &\limsup_{\|\qrob-\prob^*\|_{\infty}\to 0}\left|D_{s}(\lambda,\qrob)-D_{s}(\lambda,\prob^*)\right|\\
	 =&\lim_{\|\qrob-\prob^*\|_{\infty}\to 0} \|\qrob-\prob^*\|_{\infty} \\
&+ \lim_{\|\qrob-\prob^*\|_{\infty}\to 0} \left|\prob^*\left(V_{s}\leq\lambda, \VV_{s-1}\leq  \bLam_{s-1}(\qrob)\right)-\prob^*\left(V_{s}\leq\lambda, \VV_{s-1}\leq \bLam_{s-1}(\prob^*)\right)\right|\\
	 =&0.
\end{split}
	\end{equation}
	That is, $D_{s}(\lambda,\qrob)$ is sup-norm continuous at $\prob^*$. Moreover, by Lemma~\ref{lemma:continuous-density} and $(\lambda,\bLam_{s-1}(\prob^*))\in(0,1]^s$, we have $D_{s}(\lambda,\prob^*)>0$. This further implies that $D_{s}(\lambda,\qrob)^{-1}$ is also sup-norm continuous at $\prob^*$.

We proceed to the analysis of $N_{s}(\lambda,\qrob)$. We have
\begin{equation}
\begin{split}
		N_{s}(\lambda,\qrob)=&\expe_{\qrob}\left[V_{s}\ind\left\{V_{s}\leq\lambda, \VV_{s-1}\leq \bLam_{s-1}(\qrob)\right\}\right]\\
		= &\expe_{\qrob}\left[\int_{0}^1\ind\{r< V_s\}dr\ind\left\{V_{s}\leq\lambda, \VV_{s-1}\leq \bLam_{s-1}(\qrob)\right\}\right]\\
		 = &\int_{0}^1 \qrob\left(r< V_s\leq \lambda, \VV_{s-1}\leq \bLam_{s-1}(\qrob)\right) dr\\
		 = &  \qrob\left(V_s\leq \lambda, \VV_{s-1}\leq \bLam_{s-1}(\qrob)\right) \\
		& - \int_{0}^{\lambda} \qrob\left(V_s\leq r, \VV_{s-1}\leq \bLam_{s-1}(\qrob)\right) dr\\
		= & D_{s}(\lambda,\qrob)-  \int_{0}^{\lambda} D_{s}(r,\qrob) dr.\label{eq:part-N-1}
\end{split}
\end{equation}
We have already shown that the first term $D_{s}(\lambda,\qrob)$ on the right-hand side of the above display is sup-norm continuous at $\prob^*$. We take a closer look at the second term,
\begin{equation}
	\begin{split}
		&\left|\int_{0}^{\lambda} D_{s}(r,\qrob) dr-\int_{0}^{\lambda} D_{s}(r,\prob^*) dr\right|\\
		% \leq & \left|\int_{0}^{\epsilon} D_{s}(r,\qrob) dr-\int_{0}^{\epsilon} D_{s}(r,\prob^*) dr\right|
		% +\left|\int_{\lambda-\epsilon}^{\lambda} D_{s}(r,\qrob) dr-\int_{\lambda-\epsilon}^{\lambda} D_{s}(r,\prob^*) dr\right| \\
		% &+ \left|\int_{\epsilon}^{\lambda-\epsilon} D_{s}(r,\qrob) dr-\int_{\epsilon}^{\lambda-\epsilon} D_{s}(r,\prob^*) dr\right|\\
		% \leq & 2\epsilon + \left|\int_{\epsilon}^{\lambda-\epsilon} D_{s}(r,\qrob) dr-\int_{\epsilon}^{\lambda-\epsilon} D_{s}(r,\prob^*) dr\right|\\
		\leq & \int_{0}^{\lambda}\left|\qrob\left(V_{s}\leq r, \VV_{s-1}\leq \bLam_{s-1}(\qrob)\right)-\prob^*\left(V_{s}\leq r, \VV_{s-1}\leq \bLam_{s-1}(\qrob)\right)\right| dr\\
	 &+\int_{0}^{\lambda}\left|\prob^*\left(V_{s}\leq r, \VV_{s-1}\leq  \bLam_{s-1}(\qrob)\right)-\prob^*\left(V_{s}\leq r, \VV_{s-1}\leq \bLam_{s-1}(\prob^*)\right)\right| dr\\
	 	\leq & \|\qrob-\prob^*\|_{\infty}\\
	 &+\int_{0}^{\lambda}\left|\prob^*\left(V_{s}\leq r, \VV_{s-1}\leq  \bLam_{s-1}(\qrob)\right)-\prob^*\left(V_{s}\leq r, \VV_{s-1}\leq \bLam_{s-1}(\prob^*)\right) \right|dr.\label{eq:bound-intde}
	\end{split}
\end{equation}
% \begin{equation}
% 	\begin{split}
% 		&\left|\int_{0}^{\lambda} D_{s}(r,\qrob) dr-\int_{0}^{\lambda} D_{s}(r,\prob^*) dr\right|\\
% 		\leq & \left|\int_{0}^{\epsilon} D_{s}(r,\qrob) dr-\int_{0}^{\epsilon} D_{s}(r,\prob^*) dr\right|
% 		+\left|\int_{\lambda-\epsilon}^{\lambda} D_{s}(r,\qrob) dr-\int_{\lambda-\epsilon}^{\lambda} D_{s}(r,\prob^*) dr\right| \\
% 		&+ \left|\int_{\epsilon}^{\lambda-\epsilon} D_{s}(r,\qrob) dr-\int_{\epsilon}^{\lambda-\epsilon} D_{s}(r,\prob^*) dr\right|\\
% 		\leq & 2\epsilon + \left|\int_{\epsilon}^{\lambda-\epsilon} D_{s}(r,\qrob) dr-\int_{\epsilon}^{\lambda-\epsilon} D_{s}(r,\prob^*) dr\right|\\
% 		\leq & 2\epsilon + \int_{\epsilon}^{\lambda-\epsilon}\left|\qrob\left(V_{s}\leq r, \VV_{s-1}\leq \bLam_{s-1}(\qrob)\right)-\prob^*\left(V_{s}\leq r, \VV_{s-1}\leq \bLam_{s-1}(\qrob)\right)\right| dr\\
% 	 &+\left|\int_{\epsilon}^{\lambda-\epsilon}\prob^*\left(V_{s}\leq r, \VV_{s-1}\leq  \bLam_{s-1}(\qrob)\right)-\prob^*\left(V_{s}\leq r, \VV_{s-1}\leq \bLam_{s-1}(\prob^*)\right) dr\right|\\
% 	 	\leq & 2\epsilon +\|\qrob-\prob^*\|_{\infty}\\
% 	 &+\left|\int_{\epsilon}^{\lambda-\epsilon}\prob^*\left(V_{s}\leq r, \VV_{s-1}\leq  \bLam_{s-1}(\qrob)\right)-\prob^*\left(V_{s}\leq r, \VV_{s-1}\leq \bLam_{s-1}(\prob^*)\right) dr\right|\label{eq:bound-intde}
% 	\end{split}
% \end{equation}
Since $\bLam_{s-1}(\qrob)$ is sup-norm continuous at $\qrob=\prob^*$, for any $\varepsilon>0$, there exists $\delta>0$ such that $\|\qrob-\prob^*\|_{\infty}\leq \delta$ implies $\|\bLam_{s-1}(\qrob)-\bLam_{s-1}(\prob^*)\|\leq \varepsilon$. Then, for each $r\in[0,1]$ and $\|\qrob-\prob^*\|_{\infty}\leq \delta$, $\|(r,\bLam_{s-1}(\qrob))-(r,\bLam_{s-1}(\prob^*))\|\leq \varepsilon$, and
\begin{equation}\label{eq:cdf-bound}
\begin{split}
		&\sup_{\|\qrob-\prob^*\|\leq\delta}\left|\prob^*\left(V_{s}\leq r, \VV_{s-1}\leq  \bLam_{s-1}(\qrob)\right)-\prob^*\left(V_{s}\leq r, \VV_{s-1}\leq \bLam_{s-1}(\prob^*)\right)\right|\\
	\leq &\sup_{\|\vv_{s}-\vv_{s}'\|\leq\varepsilon, \vv_s,\vv_s'\in[0,1]^s}\left|
	\prob^*\left(\VV_s\leq \vv_s\right)-\prob^*\left(\VV_s\leq \vv_s'\right)
	\right|.
\end{split}
\end{equation}
By Lemma~\ref{lemma:continuous-density}, $\VV_s$ has a continuous density function. Thus, its cumulative distribution function, $\prob^*\left(\VV_s\leq \vv_s\right)$, is continuous over $[0,1]^s$. As $[0,1]^s$ is compact, this continuity implies that the cumulative distribution is also uniformly continuous over $[0,1]^s$. That is, for any $\epsilon_1$ small enough, there is $\epsilon>0$, such that
$$\sup_{\|\vv_{s}-\vv_{s}'\|\leq\varepsilon, \vv_s,\vv_s'\in[0,1]^s}\left|
	\prob^*\left(\VV_s\leq \vv_s\right)-\prob^*\left(\VV_s\leq \vv_s'\right)
	\right|\leq\varepsilon_1.$$
Combine the above inequality with \eqref{eq:bound-intde} and \eqref{eq:cdf-bound}, we can see that for any $\varepsilon_1>0$, there is $0<\delta<\varepsilon_1$ such that for $\|\qrob-\prob^*\|_{\infty}\leq \delta$,
\begin{equation}
\begin{split}
	\left|\int_{0}^{\lambda} D_{s}(r,\qrob) dr-\int_{0}^{\lambda} D_{s}(r,\prob^*) dr\right|
		\leq \delta+ \varepsilon_1\leq 2\varepsilon_1.
\end{split}
\end{equation}
Therefore, $\int_{0}^{\lambda} D_{s}(r,\qrob) dr$ is sup-norm continuous at $\qrob=\prob^*$. This result, combined with \eqref{eq:part-N-1}, shows that $N_s(\lambda,\qrob)$ is sup-norm continuous at $\qrob=\prob^*$.

Finally, the sup-norm continuity of $G_{s}(\lambda,\qrob)$ is implied by that of $D_{s}(\lambda,\qrob)^{-1}$ and $N_{s}(\lambda,\qrob)$ for $\lambda\in (0,1]$.

{\bf Proof of \eqref{eq:state-3} for $t=s$.}
Recall $\Lambda_{s}(\mathbb{Q})=\sup\left\{\lambda:G_{s}(\lambda,\mathbb{Q})\leq\alpha\text{ and }\lambda\in[0,1]\right\}.$
We discuss two cases.

{\bf Case 1: $\Lambda_{s}(\prob^*)=1$.}
For any sufficiently small $\varepsilon>0$, by the strict increasing property of $G_{s}(\lambda,\prob^*)$ there exists $\varepsilon_1>0$ such that $G_{s}(\lambda',\prob^*)<G_{s}(\Lambda_{s}(\prob^*),\prob^*)-2\varepsilon_1$ for all $\lambda'\leq\Lambda_{s}(\prob^*)-\varepsilon$. On the other hand, according to the sup-norm continuity of $G_{s}(\Lambda_{s}(\prob^*)-\varepsilon,\qrob)$ at $\qrob=\prob^*$, there exists $\delta>0$ such that $|G_{s}(\Lambda_{s}(\prob^*)-\varepsilon,\qrob)-G_{s}(\Lambda_{s}(\prob^*)-\varepsilon,\prob^*)|\leq \varepsilon_1$ for all $\|\qrob-\prob^*\|_{\infty}\leq \delta$. Then, for all $\|\qrob-\prob^*\|_{\infty}\leq \delta$ and $\lambda'\leq\Lambda_{s}(\prob^*)-\varepsilon$, we have
\begin{equation}
	\begin{split}
		&G_{s}(\lambda',\qrob)\\
		\leq& G_{s}(\Lambda_{s}(\prob^*)-\varepsilon,\qrob)\\
		\leq & G_{s}(\Lambda_{s}(\prob^*)-\varepsilon,\prob^*)+|G_{s}(\Lambda_{s}(\prob^*)-\varepsilon,\qrob)-G_{s}(\Lambda_{s}(\prob^*)-\varepsilon,\prob^*)|\\
		\leq &G_{s}(\Lambda_{s}(\prob^*)-\varepsilon,\prob^*)+\varepsilon_1\\
		\leq &G_{s}(\Lambda_{s}(\prob^*),\prob^*)-\varepsilon_1\\
		\leq & \alpha-\varepsilon_1.
\end{split}
\end{equation}
This implies $1-\varepsilon=\Lambda_{s}(\prob^*)-\varepsilon\leq \Lambda_{s}(\qrob)\leq 1$ for all $\|\qrob-\prob^*\|_{\infty}\leq \delta$.

{\bf Case 2: $\Lambda_{s}(\prob^*)<1$.}
Using similar arguments as those for the Case 1, we arrive at that for any $\varepsilon>0$ there exists $\delta>0$ such that  $\Lambda_{s}(\prob^*)-\varepsilon\leq \Lambda_{s}(\qrob)$ for all $\|\qrob-\prob^*\|_{\infty}\leq \delta$. We proceed to an upper bound of $\Lambda_{s}(\qrob)$.

Note that in this case, $G_{s}(\Lambda_{s}(\prob^*),\prob^*)=\alpha$. According to the definition of $\Lambda_{s}(\prob^*)$, for any $\varepsilon>0$, then there
exists $\varepsilon_1>0$ such that $G_{s}(\lambda',\prob^*)>\alpha+2\varepsilon_1$ for all $\lambda'\geq\Lambda_{s}(\prob^*)+\varepsilon$. On the other hand, according to the sup-norm continuity of $G_{s}(\Lambda_{s}(\prob^*)+\varepsilon,\qrob)$ at $\qrob=\prob^*$, there exists $\delta$ such that $|G_{s}(\Lambda_{s}(\prob^*)+\varepsilon,\qrob)-G_{s}(\Lambda_{s}(\prob^*)+\varepsilon,\prob^*)|\leq \varepsilon_1$ for all $\|\qrob-\prob^*\|_{\infty}\leq \delta$. Then, for all $\|\qrob-\prob^*\|_{\infty}\leq \delta$ and $\lambda'>\Lambda_{s}(\prob^*)+\varepsilon$, we have
\begin{equation}
\begin{split}
		&G_{s}(\lambda',\qrob)\\
		\geq& G_{s}(\Lambda_{s}(\prob^*)+\varepsilon,\qrob)\\
		\geq &\alpha+2\varepsilon_1-|G_{s}(\Lambda_{s}(\prob^*)+\varepsilon,\qrob)-G_{s}(\Lambda_{s}(\prob^*)+\varepsilon,\prob^*)|\\
		\geq & \alpha+\varepsilon_1.
\end{split}
\end{equation}
This implies that for $\lambda'>\Lambda_{s}(\prob^*)+\varepsilon$ and $\|\qrob-\prob^*\|_{\infty}\leq \delta$, $G_{s}(\lambda',\qrob)>\alpha$. Thus,  $\Lambda_{s}(\qrob)\leq \Lambda_{s}(\prob^*)+\varepsilon$ for $\|\qrob-\prob^*\|_{\infty}\leq \delta$. Combining the upper bound and lower bound of $\Lambda_{s}(\qrob)$, we arrive at %\yc{$|\cdot|$ instead of norm }
\begin{equation}
	|\Lambda_{s}(\qrob)-\Lambda_{s}(\prob^*)|\leq\varepsilon
\end{equation}
for $\|\qrob-\prob^*\|_{\infty}\leq \delta$.

This completes the proof of \eqref{eq:state-3}.

Finally, we show $\Lambda_{t}(\prob^*)>0$.
This is true because  $G_{t}(\lambda,\prob^*)$ is continuous and strictly increasing in $\lambda$ and $\lim_{\lambda\to 0+} G_{t}(\lambda,\prob^*)=0<\alpha$. %(I think you need the continuity here. It should be easy to show. Need also to modify the statement of the %lemma.).
%In addition, since $\lim_{\lambda\to 0+} G_{s}(\lambda,\prob^*)=0<\alpha$, we have $\Lambda_{s}(\prob^*)>0$.
%Thus, the induction step is completed.

\end{proof}
\section{Calculations for Example~\ref{example:noopt}}
We start with calculating $\prob\left(\tau_k=0|X_{k,1}=x_{k,1},\cdots, X_{k,t}=x_{k,t}\right)$. Let $t_1=t_4=3$ and $t_2=t_3=1$. Under the model specified in the example, we have $\tau_k=0$ or $\tau_k=t_k$ a.s. for $k=1,\cdots, 4$. As a result, we have
\begin{equation}
	\prob(\tau_k\leq t-1|X_{k,1}=x_{k,1},\cdots, X_{k,t}=x_{k,t}) = 1
\end{equation}
for $t\geq t_k+1$.

To simplify the calculation for the other cases, we  first prove the following auxiliary result:
under the model specified in this example,
for any $x_{k,1},\cdots,x_{k,t}\in \{0,1\}$ and $0\leq t \leq t_k$,
\begin{equation}\label{eq:best-and-worst-case}
\begin{split}
		&\prob\left(\tau_k\leq t-1|X_{k,1}=0,\cdots, X_{k,t}=0\right)\\
		\leq & \prob\left(\tau_k\leq t-1|X_{k,1}=x_{k,1},\cdots, X_{k,t}=x_{k,t}\right)\\
		\leq & \prob\left(\tau_k \leq t-1|X_{k,1}=1,\cdots, X_{k,t}=1\right).
\end{split}
\end{equation}
Indeed, direct calculation gives
\begin{equation}\label{eq:all-case}
\begin{split}
		&\prob\left(\tau_k\leq t-1|X_{k,1}=x_{k,1},\cdots, X_{k,t}=x_{k,t}\right)\\
		= & \frac{\prob(\tau_k=0)(0.51)^{\sum_{s=1}^t x_{k,t}}(0.49)^{t-\sum_{s=1}^t x_{k,t}}}{\prob(\tau_k=0)(0.51)^{\sum_{s=1}^t x_{k,t}}(0.49)^{t-\sum_{s=1}^t x_{k,t}} + \prob(\tau_k=t_k)(0.5)^t }.
\end{split}
\end{equation}
The above display is monotonically increasing in $\sum_{s=1}^t x_{k,t}$. Thus, \eqref{eq:best-and-worst-case} is proved.

Let $\tilde{W}_{k,t}:=\prob\left(\tau_k\leq t-1|X_{k,1}=x_{k,1},\cdots, X_{k,t}=x_{k,t}\right)$. Using \eqref{eq:best-and-worst-case} and \eqref{eq:all-case}, we obtain that for $0\leq t\leq t_k$,
\begin{equation}
 \tilde{W}_{k,t}\in\left[\frac{\prob(\tau_k=0)(0.49)^{t}}{\prob(\tau_k=0)(0.49)^{t} + \prob(\tau_k=t_k)(0.5)^t }, \frac{\prob(\tau_k=0)(0.51)^{t}}{\prob(\tau_k=0)(0.51)^{t} + \prob(\tau_k=t_k)(0.5)^t }\right].
\end{equation}
Plugging  $\prob(\tau_k=0)$ and $\prob(\tau_k=t_k)=1-\prob(\tau_k=0)$ into the above equations, we obtain that 		$\tilde{W}_{k,t}=1$ for $t\geq 4$, and for $0\leq t\leq 3$, the a.s. range of $\tilde{W}_{k,t}$s are given below (numbers are rounded to the third decimal place).
$$
\begin{tabu}{c|ccc}
\hline
\tilde{W}_{k,t}\in & t=1  & t=2 & t=3   \\ \hline
k=1                                            & [0.098,0.102]  & [0.096,0.104] & [0.095,0.105]   \\
k=2                                            & [0.395,0.405]  & \{1\}   & \{1\}    \\
k=3                                            & [0.425,0.435] &  \{1\}   & \{1\} \\
k=4                                            & [0.545,0.555] & [0.540,0.560]   & [0.535,0.565] \\ \hline
\end{tabu}
$$
With these numbers, the following inequalities can be verified.
\begin{equation}
\begin{split}
		\tilde{W}_{1,1}<\alpha< \tilde{W}_{2,1}<\tilde{W}_{3,1}<\tilde{W}_{4,1},\\
		\frac{1}{3}(\tilde{W}_{1,1}+ \tilde{W}_{2,1}+\tilde{W}_{3,1})\leq 0.314 <\alpha = 0.34,\\
		\frac{1}{3}(\tilde{W}_{1,1}+ \tilde{W}_{2,1}+\tilde{W}_{4,1})\geq 0.346 >\alpha\\
				\frac{1}{2}(\tilde{W}_{1,1}+ \tilde{W}_{4,1})\leq 0.329 <\alpha.
\end{split}
\end{equation}
The above inequalities implies that $\expe[\pcr_2(\TT)|\fil_1]\leq \alpha$ is equivalent to
\begin{equation}
	S_2\in\big\{
\{1,2,3\},\{1,2\},\{1,3\},\{1,4\},\{1\},\emptyset
	\big\}.
\end{equation}
 Now we consider $S_3$. We can verify the following inequalities.
\begin{equation}
 	\begin{split}
 	\tilde{W}_{1,2}<\alpha<\tilde{W}_{4,2}<\tilde{W}_{2,2}=\tilde{W}_{3,2},\\
 	\frac{1}{2}(\tilde{W}_{1,2}+ \tilde{W}_{2,2})=\frac{1}{2}(\tilde{W}_{1,2}+ \tilde{W}_{3,2})\geq 0.548 >\alpha\\
 		\frac{1}{2}(\tilde{W}_{1,2}+ \tilde{W}_{4,2})\leq 0.332 <\alpha.
 	\end{split}
 \end{equation}
The above inequalities implies that $\expe[\pcr_3(\TT)|\fil_2]\leq \alpha$ is equivalent to that $S_3\subset S_2$ and
$$
S_3\in\big\{
\{1,4\}, \{1\},\emptyset
\big\}.
$$
Similarly, for $S_4$, we have
\begin{equation}
 	\begin{split}
 	\tilde{W}_{1,3}<\alpha<\tilde{W}_{4,3}<\tilde{W}_{2,3}=\tilde{W}_{3,3},\\
 	\frac{1}{2}(\tilde{W}_{1,2}+ \tilde{W}_{2,2})=\frac{1}{2}(\tilde{W}_{1,2}+ \tilde{W}_{3,2})\geq 0.547 >\alpha\\
 		\frac{1}{2}(\tilde{W}_{1,2}+ \tilde{W}_{4,})\leq 0.336 <\alpha.
 	\end{split}
\end{equation}
This implies that $\expe[\pcr_4(\TT)|\fil_3]\leq \alpha$ is equivalent to that $S_4\subset S_3$ and
$$
S_4\in\big\{
\{1,4\}, \{1\},\emptyset
\big\}.
$$
Finally, since $\tilde{W}_{k,t}=1$ for all $t\geq 4$ and $k=1,\cdots,4$, we obtain $S_t=\emptyset$ for $t\geq 5$.

Enumerating all the index sets satisfying the constraint, we obtain that $\sup_{\TT\in\adset_{\alpha}}\expe(\uti_{2}(\TT))=7$ and the maximum achieved if and only if $S_1=\{1,2,3,4\}$ and $S_2=\{1,2,3\}$. In addition, $\sup_{\TT\in\adset_{\alpha}}\expe(\uti_{4}(\TT))=10$ and the maximum is achieved if and only if $S_1=\{1,2,3,4\}$, $S_2 = \{1,4\}$, $S_3=\{1,4\}$ and $S_4=\{1,4\}$. However, these two maxima cannot be achieved at the same time as they require different choices of $S_2$.

\bibliography{SeqChange,coupling}
\bibliographystyle{apalike}
% \bibliography{coupling}
\end{document}